\documentclass[11pt,a4paper]{article}
\usepackage[margin=27mm]{geometry}
\usepackage[T1]{fontenc}
\usepackage[utf8]{inputenc}
\usepackage{lmodern}
\usepackage{amsmath,amssymb,amsthm,mathtools}
\usepackage{enumitem}
\usepackage{hyperref}
\usepackage{microtype}
\hypersetup{colorlinks=true,linkcolor=blue,citecolor=blue,urlcolor=blue}
\newtheorem{theorem}{Theorem}[section]
\newtheorem{lemma}[theorem]{Lemma}
\newtheorem{corollary}[theorem]{Corollary}
\newtheorem{proposition}[theorem]{Proposition}
\newtheorem{remark}[theorem]{Remark}
\newcommand{\Om}{\Omega}
\newcommand{\QT}{Q_T}
\newcommand{\HH}{H}
\newcommand{\VV}{V}
\newcommand{\dual}[2]{\langle #1,#2\rangle}
\newcommand{\Rn}{\mathbb R^m}
\newcommand{\dd}{\,\mathrm d}
\title{Energy Equality, Vectorial $L^1$ Contraction, and\\ Mixed-Norm Stability for Parabolic Systems with Rough Skew-Symmetric Potentials}
\author{Francisco Javier Fern\'andez Fern\'andez}
\date{17 October 2026}
\begin{document}\maketitle
\begin{abstract}
We consider Dirichlet parabolic systems with measurable, real
skew-symmetric zero-order potentials on bounded Lipschitz domains in
three dimensions. Although skew-symmetry formally removes the potential
from the energy balance, singular coefficients may leave the product
$CU$ merely integrable, so testing the equation with the solution is
not justified. We overcome this obstruction through bounded-gradient
entropy identities and stable approximation. For
$C\in L^r(0,T;L^q(\Omega;\mathbb R^{m\times m}))$ with
$1\le r\le\infty$, $6/5\le q\le\infty$, and
$2/r+3/q\le7/2$, we establish existence and uniqueness in the full
energy--distributional class, exact energy equality at every pair of
times, Euclidean-vector $L^1$ contraction, and strong convergence of
both solutions and coefficient--solution products under strong
coefficient approximation. We compare a vectorial entropy proof with
a derivation by scalar entropies and spherical averaging, and explain
the narrower roles of temporal regularization, convex invariance, and
adjoint transposition. An application to diffusion--precession gives
quantitative stability under $L^1$ perturbations of prescribed magnetic
fields for bounded initial data, including singular fields outside the
main mixed-norm range. 
\end{abstract}

\section{Introduction}\label{sec:intro}

The flow generated by a skew-symmetric matrix preserves Euclidean
lengths. When such a matrix acts as a zero-order coefficient
in a diffusion equation, one therefore expects the diffusion to
dissipate energy while the coupling contributes nothing to the energy
balance. For a smooth solution $U$, the reason is the elementary
identity $U\cdot CU=0$ whenever $C^\top=-C$. But this formal argument
leaves a genuine question unanswered: does the same energy equality
hold for an arbitrary weak solution when the matrix field is
singular?

The difficulty lies not in the cancellation, but in making it a legal
step of the weak formulation. We consider
$U_t-\nu\Delta U+CU=0$ with homogeneous Dirichlet boundary conditions
on a bounded Lipschitz domain $\Omega\subset\mathbb R^3$. In the base
case $C\in L^2(0,T;L^{6/5}(\Omega;\mathbb R^{m\times m}))$, the
energy-space bounds give $U\in L^2(0,T;L^6)$ and hence
$CU\in L^1((0,T)\times\Omega)$. This product can be paired with
bounded smooth tests, but it need not belong to
$L^2(0,T;(H_0^1(\Omega;\mathbb R^m))')$. Thus the distributional
equation does not, by itself, permit the choice of $U$ as a test
function. An assertion of energy equality based only on the pointwise
identity $U\cdot CU=0$ would miss precisely the analytic issue we must
resolve.

Our starting point is to replace the forbidden energy test by
{bounded-gradient entropies}. Their derivatives remain legitimate
tests against an $L^1$ source, while their radial structure retains
the skew cancellation. This requires a chain rule for time derivatives
that split between a Sobolev dual space and $L^1$ in space--time. It
also requires genuine time traces: the chain rule is conditional on
$C([0,T];L^1)$ continuity, which we obtain independently by comparing
solutions of bounded-coefficient problems. The proof therefore follows
a deliberate order: bounded approximation produces a continuous
limit, the chain rule recovers exact energy, and the resulting
identities yield universal uniqueness and stronger convergence.

We first prove the theorem for $C\in L^2(0,T;L^{6/5}(\Omega;\mathbb R^{m\times m}))$.
This case already contains the central difficulty: $CU$ is integrable,
but it need not belong to the Sobolev dual needed for the usual energy
test. Treating it first allows us to establish the chain rule,
exact-time continuity, and uniqueness without simultaneously managing
the different temporal and spatial endpoints of the mixed range.
We then keep these mechanisms and identify the additional multiplier
and approximation estimates required for the complete range.
The full result, Theorem~\ref{thm:mixed}, covers every
real skew-symmetric
$C\in L^r(0,T;L^q(\Omega;\mathbb R^{m\times m}))$ satisfying
\begin{displaymath}
 1\le r\le\infty,\qquad \frac65\le q\le\infty,
 \qquad \frac2r+\frac3q\le\frac72.
\end{displaymath}
For each $L^2$ initial datum it gives existence and uniqueness among
{all} solutions in the specified energy--distributional class,
not only among limits produced by our approximation. It proves exact
energy equality between any two times and contraction of the spatial
$L^1$ norm of the {Euclidean vector length}. Under strong
convergence of skew coefficients in the same mixed norm, solutions
converge strongly both in $C([0,T];L^2(\Omega;\mathbb R^m))$
and in $L^2(0,T;H_0^1(\Omega;\mathbb R^m))$, and the products $C_nU_n$ converge to $CU$
strongly in $L^1(Q_T;\mathbb R^m)$. That last assertion matters because convergence
of the solutions alone does not automatically identify the product
with a singular, varying coefficient.

There are several useful ways to see parts of this argument, and we
include them to clarify rather than multiply proofs. The main route
uses a vectorial conditional chain rule and radial entropies. A second
derivation of the complete mixed theorem starts from scalar entropy
identities and reconstructs the radial vector quantities by averaging
over directions on a sphere; the two derivations share independently
established auxiliary results. Temporal convolution offers a
different regularization of the chain rule. Convex invariance explains
heat domination in the bounded-form regime, while adjoint
transposition provides a separate route to contraction and uniqueness.
These methods do not all prove the same assertions, and their
respective hypotheses and roles are stated where they are used.

The ingredients of the argument have important precedents, but
their conclusions and hypotheses must be distinguished. The scalar
sum-space chain rules of \cite{AlvarezVazquezJDE2008,AlvarezVazquez2009}
and the renormalization methods of
\cite{BlanchardMurat1997,BlanchardMuratRedwane2001} explain how one
can pair a bounded entropy derivative with merely integrable source
terms. Related vectorial entropy identities, formulated under
other time-derivative hypotheses, are presented in \cite{Jungel2025}.
Our argument makes the vectorial skew cancellation and the required
exact-time traces available in the solution class used here.

Classical parabolic systems theory is also a genuine antecedent, not
merely a source of techniques. Chapter~VII, \S1 of
\cite{LadyzhenskayaSolonnikovUraltseva1968} gives a system
existence--uniqueness theorem in an overlapping coefficient region,
including the spatially bounded slice. Lions' nonautonomous form
theory \cite{Lions1961} covers further uniformly form-bounded cases;
vector-modulus domination and convex invariance have earlier
semigroup and form formulations
\cite{Ouhabaz1999,Shigekawa1997,ArendtDierOuhabaz2014}.
Remark~\ref{rem:mixed-lsu} checks the precise overlap, including
solution classes and domain qualifications. Thus the claimed
contribution is not the invention of weak parabolic systems or
vector contraction. It is the joint result in the stated mixed
coefficient region: uniqueness among all energy--distributional
competitors, exact energy at all times, vector contraction, and
strong convergence of both solutions and coefficient--solution
products. The particularly singular endpoints fall outside the
sufficient hypotheses of several cited classical theories. Finally,
Remark~\ref{rem:mixed-mult-sharpness} concerns sharpness of a
specific energy-space multiplication bound, not an optimal
solvability threshold for the PDE. 

An explicit singular potential in Section~\ref{sec:example} makes the
comparison with spatial Kato and heat-kernel conditions concrete,
including the nonautonomous theories in
\cite{Gulisashvili2002,Gulisashvili2005,LVV2006,RRSV2000}.
The comparison concerns the hypotheses of particular published
results, not every conceivable existence theory. We then turn to a
simple but informative application: a prescribed magnetic field
induces a skew precession coupling in a diffusion equation. For
bounded initial magnetization, the same entropy mechanism yields
quantitative stability under $L^1$ perturbations of the field,
even for an explicit solenoidal singular field outside the main mixed
region. The magnetic-field coupling is the cross-product precession term
already present in the Bloch--Torrey diffusion model
\cite{Torrey1956}; it has the algebraic skew symmetry required by
our analysis. We deliberately work with prescribed fields,
isotropic diffusion, homogeneous Dirichlet data, and at most a
spatially uniform scalar relaxation term. More general spin-transport
models include additional transport, coupling and boundary
conditions \cite{GarciaCerveraWang2007,GarciaCerveraWang2015,PuWang2020};
we do not claim to treat those systems in full. 

The paper is organized as follows. After specifying the
solution class in Section~\ref{sec:framework}, we establish the
conditional chain rule in Section~\ref{sec:chain} and the bounded
comparison estimate in Section~\ref{sec:bounded}. Sections
\ref{sec:existence}--\ref{sec:strong} construct the singular solution,
prove contraction for arbitrary competitors and upgrade the
approximation to strong energy convergence. Section~\ref{sec:mixed}
then treats the full mixed range and the complementary arguments.
The singular example and diffusion--precession application complete
the discussion in Sections~\ref{sec:example} and
\ref{sec:application-precession}.

\section{Functional framework, Bochner spaces, and the main result}\label{sec:framework}
Let $\Omega\subset\mathbb R^3$ be a bounded Lipschitz open set,
$0<T<\infty$, $\nu>0$, and $m\in\mathbb N$. We set
\begin{displaymath}
Q_T=(0,T)\times\Omega,\qquad
H=L^2(\Omega;\mathbb R^m),\qquad
V=H_0^1(\Omega;\mathbb R^m),\qquad V'=(H_0^1(\Omega;\mathbb R^m))'.
\end{displaymath}
For $H$ we take the standard real Hilbert-space inner product, and for $V$
the gradient norm $\|v\|_V=\|\nabla v\|_{L^2(\Omega;\mathbb R^{m\times3})}$,
which is equivalent to the full Sobolev norm by Poincar\'e's inequality.
The inclusion $V\hookrightarrow H$ is continuous and dense; identifying
$H$ with its dual gives the Gelfand triple $V\hookrightarrow H
\hookrightarrow V'$. The duality pairing $\langle f,v\rangle_{V',V}$
coincides with the $H$ inner product when $f\in H$.
Throughout the proofs, $\|v(t)\|_H$ means the spatial $L^2$ norm
at a fixed physical time; $\|v\|_{L^2(0,T;V)}$ is a Bochner norm;
and $\|\nabla v\|_{L^2((0,T)\times\Omega;\mathbb R^{m\times3})}$
integrates in both time and space. For scalar or matrix functions,
$\|\cdot\|_\infty$ denotes the essential or uniform supremum over the
explicitly indicated natural domain. We keep $H,V,V'$ in long proofs
only after identifying their spatial domains and value spaces above.

\subsection{Function spaces and temporal regularity}
For a Banach space $X$, $L^p(0,T;X)$ denotes the Bochner space of
strongly measurable functions $t\mapsto f(t)\in X$ with finite norm
$\bigl(\int_0^T\|f(t)\|_X^p\,\mathrm dt\bigr)^{1/p}$ for $1\le p<\infty$.
The space $C([0,T];X)$ consists of functions that are continuous with
respect to the norm of $X$ at every time, including the one-sided limits at
$0$ and $T$. We use $L^\infty(0,T;X)$ only for essential boundedness in
time; its elements do not automatically possess time traces.

In particular, saying
\begin{displaymath}
C\in L^2\bigl(0,T;L^{6/5}(\Omega;\mathbb R^{m\times m})\bigr)
\end{displaymath}
means that for almost every $t$ the matrix field $C(t,\cdot)$ belongs to
$L^{6/5}(\Omega;\mathbb R^{m\times m})$, that the map into this Lebesgue
space is strongly measurable, and that
\begin{equation}\label{eq:coefficient-norm}
\|C\|_{L^2(0,T;L^{6/5}(\Omega;\mathbb R^{m\times m}))}
=\left[\int_0^T\left(\int_\Omega |C(t,x)|_{\mathrm F}^{6/5}
\,\mathrm dx\right)^{5/3}\,\mathrm dt\right]^{1/2}<\infty.
\end{equation}
For a matrix $M=(M_{ij})_{1\le i,j\le m}$ we use its Frobenius norm
$|M|_{\mathrm F}=(\sum_{i,j=1}^m|M_{ij}|^2)^{1/2}$. The induced
Euclidean operator norm is $|M|_{\mathrm{op}}=\sup_{|\xi|=1}|M\xi|$;
in finite dimensions $|M|_{\mathrm{op}}\le |M|_{\mathrm F}\le
\sqrt m\,|M|_{\mathrm{op}}$. These pointwise norms are distinct from
the operator norm of a linear map between function spaces.
A vector $z\in\mathbb R^m$ has Euclidean norm
$|z|=(\sum_i|z_i|^2)^{1/2}$, and $Mz$ is the usual matrix--vector
product. Cauchy--Schwarz yields $|Mz|\le |M|_{\mathrm F}|z|$.
For $u:\Omega\to\mathbb R^m$, its Jacobian is the $m\times3$ matrix
$\nabla u=(\partial_j u_i)_{i,j}$. For matrices of the same shape,
$P:Q=\sum_{i,j}P_{ij}Q_{ij}$ is the Frobenius inner product of $P$ and $Q$.
For vectors $\xi,\eta\in\mathbb R^m$, their outer product
$\xi\otimes\eta\in\mathbb R^{m\times m}$ has entries
$(\xi\otimes\eta)_{ij}=\xi_i\eta_j$. We use this outer-product notation
for Hessians of radial entropies. In particular, the matrix coefficient
$C$ and the spatial Jacobian $\nabla U$ retain distinct meanings. Moreover, $\|v\|_{L^1(\Omega;\mathbb R^m)}=
\int_\Omega|v(x)|\,\mathrm dx$.

The regularity required of a solution is the intersection
\begin{displaymath}
C([0,T];L^2(\Omega;\mathbb R^m))\cap
L^2(0,T;H_0^1(\Omega;\mathbb R^m)),
\end{displaymath}
which expresses {two different properties}: the map
$t\mapsto U(t)$ is continuous from $[0,T]$ into
$L^2(\Omega;\mathbb R^m)$ in its norm, and $U$ belongs to the Bochner
space $L^2(0,T;H_0^1(\Omega;\mathbb R^m))$.
Homogeneous Dirichlet boundary conditions are imposed through
$H_0^1(\Omega;\mathbb R^m)$, not through a pointwise boundary value.
In contrast, the product $CU$ lives in
$L^1((0,T)\times\Omega;\mathbb R^m)$, the ordinary Lebesgue space on the
space--time cylinder.

The Dirichlet Laplacian is defined by the form identity
\begin{displaymath}
\langle\Delta_D v,\varphi\rangle_{V',V}
=-\int_\Omega\nabla v:\nabla\varphi\,\mathrm dx,
\qquad v,\varphi\in V.
\end{displaymath}
Thus $\Delta_D v\in V'$ for $v\in V$, and
$\Delta_D U\in L^2(0,T;V')$ whenever
$U\in L^2(0,T;V)$. This statement does not assert that $\Delta_D U$ is
an element of $L^2((0,T)\times\Omega;\mathbb R^m)$.

We occasionally abbreviate $C([0,T];X)$ by $C_tX$,
$L^r(0,T;X)$ by $L^r_tX$, and
$L^r(0,T;L^q(\Omega;E))$ by $L^r_tL^q_x$,
where the spatial value space $E=\mathbb R^m$ for solutions and
$E=\mathbb R^{m\times m}$ for coefficients is determined by the
specified argument. The symbols $L^p(Q_T;E)$ and $L^p(\Omega;E)$
refer to different domains; the domain is written out whenever
both are possible. For gradients of $\mathbb R^m$-valued functions,
the range is $\mathbb R^{m\times3}$.

\subsection{The distributional solution class}
We call $U$ an {energy--distributional solution} if it belongs to
$C([0,T];H)\cap L^2(0,T;V)$, its product $CU$ belongs to
$L^1((0,T)\times\Omega;\mathbb R^m)$, and the following identity holds for every $\varphi\in C_c^\infty((0,T)\times\Omega;\mathbb R^m)$:
\begin{equation}\label{eq:weak-definition}
-\int_0^T\!\int_\Omega U\cdot\partial_t\varphi\,\mathrm dx\,\mathrm dt
+\nu\int_0^T\!\int_\Omega\nabla U:\nabla\varphi\,\mathrm dx\,\mathrm dt
+\int_0^T\!\int_\Omega(CU)\cdot\varphi\,\mathrm dx\,\mathrm dt=0.
\end{equation}
The terms of \eqref{eq:weak-definition} are well defined: $U\in L^2(Q_T;\mathbb R^m)$,
$\nabla U\in L^2(Q_T;\mathbb R^{m\times3})$, and
$CU\in L^1(Q_T;\mathbb R^m)$, while $\varphi$ and its derivatives are bounded.
Equivalently, the differential equation holds in
$\mathcal D'(Q_T;\mathbb R^m)$, meaning that it holds when paired with
all such smooth compactly supported tests. We do not assert equality in
$L^2(0,T;V')$: that would require additional dual-space integrability
of $CU$ which the assumptions do not provide.

Neither the initial condition nor the homogeneous Dirichlet condition is
encoded by these interior tests alone. We impose $U(0)=U_0$ using the
representative in $C([0,T];H)$, while $U\in L^2(0,T;V)$ imposes the zero
Sobolev trace on the spatial boundary at almost every time.
The uniqueness claim applies to {any two} functions satisfying this
definition, even if neither was obtained by the approximation method used
later to prove existence.

\begin{theorem}[Well-posedness, energy equality, and contraction]\label{thm:main}
Assume $U_0\in L^2(\Om;\Rn)$ and
\begin{equation}\label{eq:C}
C\in L^2(0,T;L^{6/5}(\Om;\mathbb R^{m\times m})),\qquad C(t,x)^\top=-C(t,x)\quad\text{a.e.}
\end{equation}
Then there exists exactly one $U$ satisfying
\begin{equation}\label{eq:energy-class}
U\in C([0,T];L^2(\Om;\Rn))\cap L^2(0,T;H_0^1(\Om;\Rn)),\qquad CU\in L^1((0,T)\times\Om;\Rn),
\end{equation}
$U(0)=U_0$, and
\begin{equation}\label{eq:weak}
-\int_0^T\!\int_\Om U\cdot\partial_t\varphi
+\nu\int_0^T\!\int_\Om\nabla U:\nabla\varphi
+\int_0^T\!\int_\Om(CU)\cdot\varphi=0
\end{equation}
for every $\varphi\in C_c^\infty((0,T)\times\Om;\Rn)$. The initial datum and the homogeneous boundary condition are meant respectively in $C([0,T];\HH)$ and in $L^2(0,T;\VV)$. For every $0\le s\le t\le T$,
\begin{equation}\label{eq:energy}
\frac12\|U(t)\|_{H}^2+\nu\int_s^t\|\nabla U(\tau)\|_{L^2(\Omega;\mathbb R^{m\times3})}^2\dd\tau
=\frac12\|U(s)\|_{H}^2.
\end{equation}
If $U$ and $V$ are any two energy--distributional solutions satisfying \eqref{eq:energy-class}--\eqref{eq:weak} for the same $C$ (possibly with different initial data), then
\begin{equation}\label{eq:contract}
\|U(t)-V(t)\|_{L^1(\Om;\Rn)}\le\|U(s)-V(s)\|_{L^1(\Om;\Rn)}
\qquad (0\le s\le t\le T).
\end{equation}
Furthermore, let each $C_n\in L^\infty(Q_T;\mathbb R^{m\times m})$
be measurable and skew-symmetric, and assume that $C_n\to C$ in
$L^2(0,T;L^{6/5}(\Omega;\mathbb R^{m\times m}))$.
Let $U_n$ solve
$\partial_tU_n-\nu\Delta U_n+C_nU_n=0$ in $Q_T$, with homogeneous
Dirichlet boundary condition and the {same} initial value $U_0$.
Then
\begin{equation}\label{eq:strongconv}
\begin{aligned}
U_n&\longrightarrow U &&\text{in }C([0,T];L^2(\Om;\Rn))\cap L^2(0,T;H_0^1(\Om;\Rn)),\\
C_nU_n&\longrightarrow CU &&\text{in }L^1((0,T)\times\Om;\Rn).
\end{aligned}
\end{equation}
\end{theorem}

\paragraph{How the proof is organized.}
The first step is to solve problems with bounded skew coefficients.
Their standard energy identities provide uniform bounds, while a
quantitative vector $L^1(\Omega;\mathbb R^m)$  comparison makes the approximate solutions
Cauchy as continuous trajectories. This is essential: weak compactness
in the energy space alone would not give the time continuity needed
by our nonlinear chain rule.

After identifying the singular product in $L^1(Q_T)$, we may apply
the conditional chain rule to the {limit equation}. Radial
entropies justify the cancellation and recover an exact energy
identity, rather than merely passing an energy inequality to the
limit. The identity in turn promotes the limit to strong $L^2(\Omega;\mathbb R^m)$  
continuity. Applying the same argument to the difference of
{arbitrary} solutions in the declared class proves contraction
and uniqueness. Finally, equality of energy norms turns weak
gradient convergence into strong convergence. Each stage supplies a
hypothesis needed by the next, so no time trace or nonlinear test is
used before it becomes available.

\begin{remark}[The exact integrability supplied by the hypotheses]\label{rem:integrability}
For $v\in H_0^1(\Omega;\mathbb R^m)$, the three-dimensional Sobolev
inequality (see \cite[Theorem~4.31, p.~102]{AdamsFournier2003}) gives
$\|v\|_{L^6(\Omega;\mathbb R^m)}\le S_\Omega\|\nabla v\|_{L^2(\Omega;\mathbb R^{m\times3})}$.
The vector-valued form follows by zero extension, density, and the
finite-component norm equivalence. At almost every time, H\"older's inequality with conjugate exponents
$6/5$ and $6$ implies
$\|C(t)U(t)\|_{L^1(\Omega;\mathbb R^m)}\le
\|C(t)\|_{L^{6/5}(\Omega;\mathbb R^{m\times m})}
\|U(t)\|_{L^6(\Omega;\mathbb R^m)}$.
A second application of H\"older with exponents $2$ and $2$ in time yields
\begin{equation}\label{eq:holder}
\begin{split}
\|CU\|_{L^1((0,T)\times\Omega;\mathbb R^m)}
&\le
\|C\|_{L^2(0,T;L^{6/5}(\Omega;\mathbb R^{m\times m}))}
\|U\|_{L^2(0,T;L^6(\Omega;\mathbb R^m))}\\
&\le S_\Omega
\|C\|_{L^2(0,T;L^{6/5}(\Omega;\mathbb R^{m\times m}))}
\|\nabla U\|_{L^2((0,T)\times\Omega;\mathbb R^{m\times3})}.
\end{split}
\end{equation}
This estimate proves $CU\in L^1(Q_T;\mathbb R^m)$. In general it does
{not} place $CU$ in $L^2(0,T;V')$. The latter conclusion must not be
used to test the equation directly against $U$. Instead we will apply a
chain rule to bounded gradients of smooth convex radial functions of the vector variable.
\end{remark}

\begin{remark}[Beyond the base coefficient class]
The preceding estimate extends to coefficients in the mixed region
$1\le r\le\infty$, $6/5\le q\le\infty$ and
$2/r+3/q\le7/2$, but integrability of $CU$ is only the
starting point. The present base case isolates the nonlinear
chain rule, the temporal-trace issue and the uniqueness argument.
Section~\ref{sec:mixed} then supplies the energy-space multiplier
for the full range and verifies the approximation and stability
steps at its endpoints. In particular, $r=1$ uses temporal
integrability rather than uniform-in-time form bounds, whereas
$r=\infty$ cannot be handled by assuming that bounded
truncations converge strongly in an $L^\infty$ coefficient norm.
The separate presentation makes these additional arguments visible
without obscuring the mechanism already proved here.
\end{remark}

\section{A conditional chain rule for sum-space time derivatives}\label{sec:chain}

The ordinary Hilbert-space chain rule is well suited to a time
derivative taking values in $L^2(0,T;V')$. Our weak equation has an
additional source which is only integrable. We must therefore use
tests that are regular enough for the dual pairing and bounded enough
for the $L^1(Q_T;\mathbb R^m)$  pairing. This is the purpose of the following chain
rule.

More precisely, its time derivative has the form
\begin{displaymath}
\partial_t w=A+B,\qquad
A\in L^2(0,T;V'),\quad B\in L^1(Q_T;\mathbb R^m).
\end{displaymath}
We can pair $A$ with $V$ functions and integrate $B$ against bounded
functions, but neither observation allows us simply to test with
$w$. The entropy derivative will supply one admissible test for both
terms.

Instead, consider a smooth scalar function $\Phi:\mathbb R^m\to\mathbb R$.
Here $D\Phi(z)=(\partial_{z_i}\Phi(z))_{i=1}^m$ is its gradient in the
{vector variable} $z$, and
$D^2\Phi(z)=(\partial_{z_i}\partial_{z_j}\Phi(z))_{i,j=1}^m$
is the Hessian. If $D\Phi$ is bounded, then
$B\cdot D\Phi(w)\in L^1(Q_T;\mathbb R)$.
If also $D\Phi(0)=0$ and $D^2\Phi$ is bounded, the Sobolev composition
rule gives, almost everywhere,
\begin{displaymath}
\begin{aligned}
	|D\Phi(w)|\le & \|D^2\Phi\|_{L^\infty(\mathbb R^m;\mathcal L(\mathbb R^m,\mathbb R^m))} |w|,\\
	\nabla_x[D\Phi(w)]=&D^2\Phi(w)\nabla_xw,
\end{aligned}
\end{displaymath}
and therefore
$|\nabla_x[D\Phi(w)]|\le\|D^2\Phi\|_{L^\infty(\mathbb R^m;\mathcal L(\mathbb R^m,\mathbb R^m))}|\nabla_xw|$.
The zero value $D\Phi(0)=0$ preserves the homogeneous Sobolev boundary
trace. Thus $D\Phi(w)\in L^2(0,T;V)$ whenever $w\in L^2(0,T;V)$;
the vector-valued chain rule is obtained componentwise from
\cite[Exercise~11.52, p.~340]{Leoni2017}; preservation of the zero
trace follows by approximation, consistently with
\cite[Theorem~18.7, p.~595]{Leoni2017}.
Both terms in the chain-rule formula are now integrable in physical time:
the $A$ pairing by Cauchy--Schwarz in $L^2(0,T)$ and the $B$ product
by boundedness of $D\Phi$.

To evaluate the integrated identity at specified times, we additionally
assume that $w$ is continuous from $[0,T]$ into
$L^1(\Omega;\mathbb R^m)$. In our parabolic application this
continuity will be proved first by a stability estimate; it is not a
consequence of the displayed decomposition alone.

The proof regularizes in space by the Dirichlet heat operator $S(h)$,
where $h>0$ is an auxiliary parameter unrelated to physical time $t$.
For each fixed $h$, $S(h)$ smooths the dual-space term $A(t)\in V'$
and the integrable term $B(t)\in L^1(\Omega;\mathbb R^m)$ separately.
The two actions are compatible on common inputs, as justified below
using bounded Sobolev tests and transposition; no action on arbitrary
distributions is presumed.

\paragraph{The Dirichlet heat semigroup.}
We first define the regularization operator on
$H=L^2(\Omega;\mathbb R^m)$.
Given $f\in H$, consider the homogeneous Dirichlet heat problem
\begin{displaymath}
\begin{cases}
\partial_h y(h,x)-\nu\Delta y(h,x)=0,
  & (h,x)\in(0,\infty)\times\Omega,\\
y(h,x)=0,
  & (h,x)\in(0,\infty)\times\partial\Omega,\\
y(0,\cdot)=f,
  & \text{in }L^2(\Omega;\mathbb R^m).
\end{cases}
\end{displaymath}
For every $a>0$, this problem has a unique weak solution satisfying
\begin{displaymath}
y\in C([0,a];L^2(\Omega;\mathbb R^m))
\cap L^2(0,a;H_0^1(\Omega;\mathbb R^m)),
\end{displaymath}
with
\begin{displaymath}
\partial_h y\in
L^2\bigl(0,a;(H_0^1(\Omega;\mathbb R^m))'\bigr).
\end{displaymath}
The boundary condition is understood in the Sobolev trace sense,
and the initial condition holds strongly in $H$.
Existence, uniqueness and $H$-continuity follow from the Dirichlet
semigroup construction
\cite[Theorem~3.1.4, p.~34, and the discussion of the heat equation,
p.~35]{ArendtHeatKernels2005}. Its spectral energy identity,
\begin{displaymath}
\frac12\|y(h)\|_H^2+
\nu\int_0^h\|\nabla y(s)\|_{L^2(\Omega;\mathbb R^{m\times3})}^2\,\mathrm ds
=\frac12\|f\|_H^2\qquad(0\le h\le a),
\end{displaymath}
provides the stated $L^2(0,a;H_0^1(\Omega;\mathbb R^m))$ bound;
moreover $\partial_hy=\nu\Delta_Dy\in
L^2(0,a;(H_0^1(\Omega;\mathbb R^m))')$.
Uniqueness in this weak class also follows from the same energy identity.

For each $h\geq0$, we define
\begin{displaymath}
S(h):H\longrightarrow H,\qquad S(h)f=y(h,\cdot).
\end{displaymath}
Here $h$ is an auxiliary time variable, independent of the physical
time $t\in(0,T)$ in the original system. When applied to a
time-dependent function $w(t,\cdot)$, the operator $S(h)$ acts only
on its spatial variable, with the physical time $t$ held fixed.

The semigroup also acts as a strongly continuous contraction on
$L^1(\Omega;\mathbb R^m)$ and, by duality through $V$, on
$V'=(H_0^1(\Omega;\mathbb R^m))'$. On a common input the two
extensions yield the same spatial distribution. We shall prove this
compatibility by testing against $V\cap L^\infty(\Omega;\mathbb R^m)$
and using the symmetric heat kernel; it is not an extension of
$S(h)$ to all of $\mathcal D'(\Omega;\mathbb R^m)$.

The positive closed Dirichlet form on $H$ is
\begin{displaymath}
\mathfrak a(u,v)=\int_\Omega\nabla u:\nabla v\,\mathrm dx,
\qquad D(\mathfrak a)=V.
\end{displaymath}
Its associated nonnegative self-adjoint operator is $L=-\Delta_D$,
acting componentwise; the scalar Dirichlet Laplacian and its
self-adjoint semigroup are identified in
\cite[Theorem~3.1.4, p.~34]{ArendtHeatKernels2005}.
Consequently $S(h)=e^{-\nu hL}$.
To make this exponential concrete, let $(e_k)$ be an orthonormal family
of Dirichlet eigenfunctions, $Le_k=\lambda_k e_k$, with
$\lambda_k>0$. In a scalar component,
\begin{displaymath}
S(h)f=\sum_k e^{-\nu h\lambda_k}(f,e_k)_{L^2(\Omega)}e_k.
\end{displaymath}
Thus ``spectral calculus'' simply means multiplying each eigenmode by
$e^{-\nu h\lambda_k}$. The form domain satisfies
$D(L^{1/2})=H_0^1(\Omega)$ componentwise. Our estimates use this
first-order Sobolev space, not the second-order space $H^2(\Omega;\mathbb R^m)\cap H_0^1(\Omega;\mathbb R^m)$; on a nonsmooth Lipschitz domain one cannot generally identify that space with $D(L)$.
For Banach spaces $X$ and $Y$, we write $\mathcal L(X,Y)$ for the space
of bounded linear operators from $X$ to $Y$, equipped with the operator norm.

\begin{lemma}[Smoothing by the Dirichlet heat semigroup]\label{lem:heat-smoothing}
For every $h>0$, the compatible extensions of $S(h)$ satisfy
\begin{equation}\label{eq:semigroup-spectral}
\begin{aligned}
\|S(h)\|_{\mathcal L(V',H)}
    &\le (2e\nu h)^{-1/2},\\
\|S(h)\|_{\mathcal L(H,V)}
    &\le (2e\nu h)^{-1/2},\\
\|\Delta_D S(h)\|_{\mathcal L(V,H)}
    &\le (2e\nu h)^{-1/2},\\
\|S(h)\|_{\mathcal L(L^1(\Omega;\mathbb R^m),H)}
    &\le (8\pi\nu h)^{-3/4}.
\end{aligned}
\end{equation}
Moreover, $S(h)$ is strongly continuous on $L^1(\Omega;\mathbb R^m)$,
$H$, $V$, and $V'$, and maps $H$ into
$V\cap L^\infty(\Omega;\mathbb R^m)$ for each $h>0$.
\end{lemma}
\begin{proof} $ $\\ 
In the spectral representation, the dual norm is
$\|f\|_{V'}^2=\sum_k |f_k|^2/\lambda_k$ and
$\|v\|_V^2=\sum_k\lambda_k|v_k|^2$ (componentwise).
Therefore
\begin{displaymath}
\|S(h)f\|_H^2\le
\left(\sup_{\lambda>0}\lambda e^{-2\nu h\lambda}\right)
\sum_k\frac{|f_k|^2}{\lambda_k}
=(2e\nu h)^{-1}\|f\|_{V'}^2.
\end{displaymath}
Here $L^{1/2}$ denotes the nonnegative spectral square root of
$L=-\Delta_D$: if $v=\sum_k v_ke_k$, then
$L^{1/2}v=\sum_k\sqrt{\lambda_k}v_ke_k$, with
$D(L^{1/2})=V$ and $\|L^{1/2}v\|_H=\|\nabla v\|_{L^2(\Omega;\mathbb R^{m\times3})}$.
Consequently,
\begin{displaymath}
\begin{aligned}
 \|S(h)z\|_V^2
 &=\sum_k\lambda_k e^{-2\nu h\lambda_k}|z_k|^2
 \le (2e\nu h)^{-1}\|z\|_H^2,\\
 \|LS(h)v\|_H^2
 &=\sum_k\lambda_k^2e^{-2\nu h\lambda_k}|v_k|^2
 \le (2e\nu h)^{-1}\|v\|_V^2.
\end{aligned}
\end{displaymath}
These are exactly the $H\to V$ and $V\to H$ operator estimates,
with $\Delta_DS(h)=-LS(h)$ in the last line.
The free heat kernel in $\mathbb R^3$ is
$G_{\nu h}(z)=(4\pi\nu h)^{-3/2}\exp[-|z|^2/(4\nu h)]$.
The Dirichlet heat semigroup is dominated, for nonnegative data,
by the whole-space heat flow on the zero extension. One way to see
this is to compare the two nonnegative weak heat solutions, using
the positive-part energy inequality and the zero Dirichlet trace of
the domain solution; this yields the kernel domination
$0\le k_\Omega(h,x,y)\le G_{\nu h}(x-y)$ a.e. for $h>0$;
see \cite[Theorems~4.2.1 and~4.2.4, pp.~48 and~50]{ArendtHeatKernels2005}.
A direct Gaussian integral shows
\begin{displaymath}
\|G_{\nu h}\|_{L^2(\mathbb R^3)}^2
=(4\pi\nu h)^{-3}\int_{\mathbb R^3}
 e^{-|z|^2/(2\nu h)}\,\mathrm dz
=(8\pi\nu h)^{-3/2}.
\end{displaymath}
This is the kernel estimate needed for the mapping
$S(h):L^1(\Omega;\mathbb R^m)\to L^2(\Omega;\mathbb R^m)$.
It requires neither a bound into $H^2(\Omega;\mathbb R^m)$ nor an
$H^2(\Omega;\mathbb R^m)$  characterization of the operator domain $D(\Delta_D)$.
Young's inequality, componentwise positivity and the Euclidean norm yield
\begin{displaymath}
\|S(h)f\|_H\le\|G_{\nu h}\|_{L^2(\mathbb R^3)}
\|f\|_{L^1(\Omega;\mathbb R^m)}
=(8\pi\nu h)^{-3/4}\|f\|_{L^1(\Omega;\mathbb R^m)}.
\end{displaymath}
Self-adjoint duality gives $S(h):H\to L^\infty(\Omega;\mathbb R^m)$,
while the second spectral bound gives $S(h):H\to V$.
Strong continuity on $H,V,V'$ follows spectrally; on $L^1(\Omega;\mathbb R^m)$  it follows
by density of $L^1(\Omega;\mathbb R^m)$  and contractivity of the scalar Dirichlet
semigroup, extended componentwise using the nonnegative kernel.
The scalar extrapolation result
\cite[Theorem~4.4.1, p.~54]{ArendtHeatKernels2005} establishes strong
continuity for $1\le p<\infty$ (not, in general, for $p=\infty$).
The relevant contractivity estimates here are
\begin{displaymath}
\|S(h)\|_{\mathcal L(L^1(\Omega;\mathbb R^m),
                          L^1(\Omega;\mathbb R^m))}\le1,
\qquad
\|S(h)\|_{\mathcal L(L^\infty(\Omega;\mathbb R^m),
                          L^\infty(\Omega;\mathbb R^m))}\le1,
\end{displaymath}
which follow from positivity and Gaussian domination of the scalar
Dirichlet kernel, also for the Euclidean norm of the vector-valued
componentwise extension.
\end{proof}
The extensions on $L^1(\Omega;\mathbb R^m)$ and $V'$ are compatible
on their intersection, as verified by transposition in the proof of
Lemma~\ref{lem:chain}. The displayed bounds are estimates for
{fixed} $h>0$. Their right-hand sides diverge when $h\downarrow0$;
this is harmless because the limit uses strong continuity on the
unsmoothed base spaces $L^1(\Omega;\mathbb R^m)$, $V$, and $V'$.
There is no claim of a uniform $L^1\to H$ or $V'\to H$ estimate.

\begin{lemma}[Conditional vector chain rule in the sum space]\label{lem:chain}
Assume
\begin{equation}\label{eq:sumspace}
w\in C([0,T];L^1(\Om;\Rn))\cap L^2(0,T;\VV),\qquad
\partial_tw=A+B\quad\text{in }\mathcal D'(Q_T;\mathbb R^m),
\end{equation}
where $A\in L^2(0,T;\VV')$ and $B\in L^1(\QT;\Rn)$. Let $\Phi\in C^2(\Rn;\mathbb R)$ satisfy
\begin{equation}\label{eq:phi}
\Phi(0)=0,\quad D\Phi(0)=0,\quad
\|D\Phi\|_{L^\infty(\mathbb R^m;\mathbb R^m)}<\infty,\quad
\|D^2\Phi\|_{L^\infty(\mathbb R^m;\mathcal L(\mathbb R^m,\mathbb R^m))}<\infty.
\end{equation}

Define the real-valued function $E:[0,T]\to\mathbb R$ by

$$E(t)=\int_\Omega\Phi(w(t,x))\,\mathrm dx.$$
Then $E$ is absolutely continuous on $[0,T]$, and for every
$0\le s\le t\le T$,
\begin{equation}\label{eq:sumchain}
\begin{split}
\int_\Om\Phi(w(t)) -\int_\Om\Phi(w(s))
={}&\int_s^t\dual{A(\tau)}{D\Phi(w(\tau))}\dd\tau\\
&+\int_s^t\!\int_\Om B(\tau,x)\cdot D\Phi(w(\tau,x))\dd x\dd\tau.
\end{split}
\end{equation}
The derivative $D\Phi:\mathbb R^m\to\mathbb R^m$ is globally
Lipschitz because its Jacobian
$D^2\Phi:\mathbb R^m\to\mathcal L(\mathbb R^m,\mathbb R^m)$
is bounded in the induced operator norm. Since $D\Phi(0)=0$,
Sobolev composition gives $D\Phi(w(t,\cdot))\in V$ for almost
every $t$, with
\begin{displaymath}
\begin{aligned}
 \|D\Phi(w(t))\|_{L^2(\Omega;\mathbb R^m)}
 &\le M\|w(t)\|_{L^2(\Omega;\mathbb R^m)},\\
 \|\nabla[D\Phi(w(t))]\|_{L^2(\Omega;\mathbb R^{m\times3})}
 &\le M\|\nabla w(t)\|_{L^2(\Omega;\mathbb R^{m\times3})},
\end{aligned}
\end{displaymath}
where $M=\sup_{\xi\in\mathbb R^m}|D^2\Phi(\xi)|_{\mathrm{op}}$.
Thus $D\Phi(w)\in L^2(0,T;V)$, so the duality bracket is legitimate.
A proof of the composition identity by smooth approximation is given
in Corollary~\ref{cor:diffusion}.
\end{lemma}
\begin{proof} $ $\\
The proof has six stages. We first regularize the three quantities, then
justify extending interior tests to a space on which the semigroup can
act by transposition. The regularized field belongs to
$W^{1,1}(0,T;L^2(\Omega;\mathbb R^m))$, with time derivative in
$L^1(0,T;L^2(\Omega;\mathbb R^m))$, so we can apply an ordinary
Hilbert-valued chain rule. We
finally remove the smoothing separately in the $V'$ pairing, the $L^1((0,T)\times\Omega;\mathbb R^m)$ product, and the endpoint entropy values.

\medskip\noindent{Step 1: Regularize while keeping the two times distinct.}
Fix $h>0$ and set $w_h(t)=S(h)w(t)$, $A_h(t)=S(h)A(t)$, and
$B_h(t)=S(h)B(t)$. The smoothing estimates \eqref{eq:semigroup-spectral} imply
\begin{displaymath}
\begin{aligned}
w_h&\in C([0,T];H),\\
A_h&\in L^2(0,T;H), &
\|A_h\|_{L^2(0,T;H)}&\le(2e\nu h)^{-1/2}\|A\|_{L^2(0,T;V')},\\
B_h&\in L^1(0,T;H), &
\|B_h\|_{L^1(0,T;H)}&\le(8\pi\nu h)^{-3/4}
\|B\|_{L^1((0,T)\times\Omega;\mathbb R^m)}.
\end{aligned}
\end{displaymath}
The two smoothings refer to the spatial variable only. Their estimates
need not be uniform as $h\downarrow0$.

\medskip\noindent{Step 2: Extend the distributional identity to bounded Sobolev tests.}
Put $X=V\cap L^\infty(\Omega;\mathbb R^m)$. Given $\psi\in X$, choose
$f_j\in C_c^\infty(\Omega;\mathbb R^m)$ approaching $\psi$ strongly in
$V$ and almost everywhere (after extraction). Choose $M>\|\psi\|_{L^\infty(\Omega;\mathbb R^m)}$. Select a scalar
$\theta\in C^\infty(\mathbb R)$ such that
$\theta(0)=0$, $\theta(r)=r$ for $|r|\le M$, and both
$\theta$ and $\theta'$ are bounded For example, take $\chi\in C_c^\infty(\mathbb R)$ with
$0\le\chi\le1$ and $\chi=1$ on $[-M,M]$, and set
\begin{equation}\label{eq:author-scalar-cutoff}
 \theta(r)=r\chi(r),\qquad
 \theta'(r)=\chi(r)+r\chi'(r).
\end{equation}
Then $\theta\in C_c^\infty(\mathbb R)$,
$\theta(r)=r$ and $\theta'(r)=1$ on $[-M,M]$,
while both $\theta$ and $\theta'$ are bounded.
Apply $\theta$ to each component of $f_j$ and set
$\psi_j=(\theta(f_{j,1}),\ldots,\theta(f_{j,m}))$.
Because $f_j$ is smooth with compact support and $\theta(0)=0$,
$\psi_j\in C_c^\infty(\Omega;\mathbb R^m)$; moreover
$\sup_j\|\psi_j\|_{L^\infty(\Omega;\mathbb R^m)}
\le\sqrt m\,\|\theta\|_{L^\infty(\mathbb R)}$.

To prove the required strong convergence, fix a component $i$ and a
spatial derivative $\partial_k$, $k=1,2,3$. Since
$\theta(\psi_i)=\psi_i$ and $\theta'(\psi_i)=1$ almost everywhere,
\begin{displaymath}
\partial_k\theta(f_{j,i})-\partial_k\psi_i
=\theta'(f_{j,i})[\partial_k f_{j,i}-\partial_k\psi_i]
 +[\theta'(f_{j,i})-1]\partial_k\psi_i.
\end{displaymath}
The first term tends to zero in $L^2(\Omega)$ by the uniform bound
on $\theta'$ and $f_j\to\psi$ in $V$. On a subsequence converging
pointwise almost everywhere, the second term tends to zero pointwise
and is dominated by $2\|\theta'\|_\infty|\partial_k\psi_i|$ in
$L^2(\Omega)$. Dominated convergence gives its $L^2(\Omega)$ limit;
the subsequence criterion then yields convergence for the original
sequence. Since $\theta$ is Lipschitz and fixes $\psi$, its function
values also converge in $L^2(\Omega;\mathbb R^m)$. Summing over
$i,k$ proves $\psi_j\to\psi$ in $H_0^1(\Omega;\mathbb R^m)$.
The full gradient error is a field in
$L^2(\Omega;\mathbb R^{m\times3})$, not a scalar.

For $\eta\in C_c^\infty(0,T)$, the original distributional equation
with test $\eta(t)\psi_j(x)$ reads
\begin{displaymath}
-\int_0^T\eta'(t)(w(t),\psi_j)_H\,\mathrm dt
=\int_0^T\eta(t)\langle A(t),\psi_j\rangle_{V',V}\,\mathrm dt
 +\int_0^T\!\int_\Omega\eta(t)B(t,x)\cdot\psi_j(x)
\,\mathrm dx\,\mathrm dt.
\end{displaymath}
We now pass to the limit in each term separately. For the first,
$|\int_0^T\eta'(t)(w(t),\psi_j-\psi)_H\,\mathrm dt|
\le\|\eta'\|_{L^2(0,T)}\|w\|_{L^2(0,T;H)}
\|\psi_j-\psi\|_H\to0$.
For the dual term,
\begin{displaymath}
\left|\int_0^T\eta(t)\langle A(t),\psi_j-\psi\rangle\,\mathrm dt\right|
\le\|\eta\|_{L^2(0,T)}\|A\|_{L^2(0,T;V')}
\|\psi_j-\psi\|_V\to0.
\end{displaymath}
Finally $\psi_j\to\psi$ almost everywhere, with a uniform bound
in $L^\infty(\Omega;\mathbb R^m)$. The absolute value of the
remaining integrand is bounded by
$K|\eta(t)|\,|B(t,x)|$, which is integrable on
$(0,T)\times\Omega$, so dominated convergence applies. The
identity therefore holds for every $\psi\in X$. It is
{not} legitimate to test against an arbitrary unbounded
$V$ function in the $B$ term.

\medskip\noindent{Step 3: Transpose the semigroup and identify the temporal derivative.}
By the preceding smoothing estimates, $S(h)z\in X$ for every $z\in H$.
Using it as a test in Step~2 and the symmetric Dirichlet kernel gives
\begin{displaymath}
(S(h)w,z)_H=(w,S(h)z)_H,\quad
(S(h)A,z)_H=\langle A,S(h)z\rangle_{V',V},\quad
(S(h)B,z)_H=\int_\Omega B\cdot S(h)z.
\end{displaymath}
For a fixed time where $B(t)\in L^1(\Omega;\mathbb R^m)$, take
$B_j(t)\in L^1(\Omega;\mathbb R^m)\cap H$ with
$B_j(t)\to B(t)$ in $L^1(\Omega;\mathbb R^m)$.
The self-adjoint identity in $H$ gives
$(S(h)B_j,z)_H=(B_j,S(h)z)_H$. The left sides converge since
\begin{displaymath}
\|S(h)(B_j-B)\|_H\le(8\pi\nu h)^{-3/4}
\|B_j-B\|_{L^1(\Omega;\mathbb R^m)},
\end{displaymath}
and the right sides converge since
$S(h)z\in L^\infty(\Omega;\mathbb R^m)$.
This proves the $L^1$ transposition formula. For time-dependent $B$,
the $L^1\to H$ bound gives a strongly measurable element of
$L^1(0,T;H)$ by bounded linearity and Bochner approximation.
If an input belongs to both $L^1(\Omega;\mathbb R^m)$ and $V'$
and defines the same interior distribution, the Step~2 extension to
$X$ shows that the two resulting transposed pairings are equal on
every $z\in H$. This proves compatibility without invoking
nonlocal operators on arbitrary distributions. Thus, for each $z\in H$ and each temporal test
$\eta$, Step~2 yields
\begin{displaymath}
-\int_0^T\eta'(t)(w_h(t),z)_H\,\mathrm dt
=\int_0^T\eta(t)(A_h(t)+B_h(t),z)_H\,\mathrm dt.
\end{displaymath}
We have established $\partial_t w_h=A_h+B_h$ in
$\mathcal D'(0,T;H)$, with $A_h+B_h\in L^1(0,T;H)$.
Since $w_h\in C([0,T];H)$, define
$F_h(t)=w_h(0)+\int_0^t(A_h+B_h)(r)\,\mathrm dr$.
This belongs to $W^{1,1}(0,T;H)$ and has derivative $A_h+B_h$
almost everywhere by the fundamental theorem of calculus for
Bochner integrals: if $g\in L^1(0,T;H)$, then
$t\mapsto\int_0^t g(r)\,\mathrm dr$ is absolutely continuous as an
$H$-valued function and has derivative $g(t)$ almost everywhere
(see, for example, the Banach-space evolution framework in
\cite{LionsMagenes1972Vol1}). The
distributional derivative of $w_h-F_h$ is zero, hence this difference
is constant on $(0,T)$; continuity at $t=0$ identifies that constant
with zero. Thus $w_h=F_h$ for every $t\in[0,T]$, including the
endpoints. This argument proves precisely the temporal regularity
needed here rather than relying on an unspecified theorem number.

\medskip\noindent{Step 4: Apply the Fr\'echet chain rule to an $H$-valued path.}
We use the Banach-space chain rule for an absolutely continuous
curve and a continuously Fr\'echet differentiable scalar functional:
$(\mathcal F\circ w_h)'(t)=D\mathcal F(w_h(t))[w_h'(t)]$
almost everywhere. This is the standard chain-rule mechanism of
Bochner--Sobolev calculus (cf.~\cite{LionsMagenes1972Vol1});
the specific derivative and remainder estimate needed here are
verified directly below.
On $H$ define $\mathcal F(v)=\int_\Omega\Phi(v(x))\,\mathrm dx$.
The conditions on $\Phi$ and finite measure of $\Omega$ give
\begin{displaymath}
\begin{aligned}
D\mathcal F(v)[z]=&\int_\Omega D\Phi(v)\cdot z\,\mathrm dx,
\\
|\mathcal F(v+z)-\mathcal F(v)-D\mathcal F(v)[z]|
\le & \frac12\|D^2\Phi\|_{L^\infty(\mathbb R^m;\mathcal L(\mathbb R^m,\mathbb R^m))}\|z\|_H^2.
\end{aligned}
\end{displaymath}
Consequently $\mathcal F\in C^1(H;\mathbb R)$, with
$D\mathcal F:H\to\mathcal L(H,\mathbb R)$ globally Lipschitz
and bounded. The chain rule in this particular setting can be proved
directly: at almost every $t$, $w_h$ is strongly differentiable as
an $H$-valued function. Fr\'echet differentiability implies
$\frac{\mathrm d}{\mathrm dt}\mathcal F(w_h(t))
=D\mathcal F(w_h(t))[w_h'(t)]$ there; indeed the Taylor remainder
is bounded by $(M/2)\|w_h(t+r)-w_h(t)\|_H^2$, which divided by
$|r|$ tends to zero at each strong differentiability time.
The composition is absolutely continuous because $\mathcal F$ is
Lipschitz and $w_h\in W^{1,1}(0,T;H)$, while its derivative is
integrable since
$\|D\mathcal F(w_h)\|_{\mathcal L(H,\mathbb R)}
\le|\Omega|^{1/2}\|D\Phi\|_{L^\infty(\mathbb R^m;\mathbb R^m)}$.
Integrating its almost-everywhere identity between arbitrary
$s,t\in[0,T]$ yields
\begin{equation}\label{eq:chainh}
\int_\Omega\Phi(w_h(t))-\int_\Omega\Phi(w_h(s))
=\int_s^t (A_h,D\Phi(w_h))_H\,\mathrm d\tau
+\int_s^t (B_h,D\Phi(w_h))_H\,\mathrm d\tau
\end{equation}
for every pair of endpoints. No false pointwise domination of temporal
difference quotients by $|\partial_tw_h(t)|$ is used.

\medskip\noindent{Step 5: Remove spatial smoothing from both source pairings.}
Strong continuity of $S(h)$ on $V$ and $V'$ implies
$w_h\to w$ in $L^2(0,T;V)$ and $A_h\to A$ in
$L^2(0,T;V')$. Its $L^1(\Omega;\mathbb R^m)$ contraction and strong continuity give
$B_h\to B$ in $L^1((0,T)\times\Omega;\mathbb R^m)$.
The globally Lipschitz map $D\Phi$ satisfies
\begin{displaymath}
\nabla[D\Phi(w_h)] = D^2\Phi(w_h)\nabla w_h.
\end{displaymath}
To justify convergence in $L^2(0,T;V)$, subtract the analogous
expression for $w$, and split the gradient difference into
\begin{displaymath}
D^2\Phi(w_h)(\nabla w_h-\nabla w)
+[D^2\Phi(w_h)-D^2\Phi(w)]\nabla w.
\end{displaymath}
The first term tends to zero by the uniform Hessian bound and strong
gradient convergence. For the second, let
$Q_T=(0,T)\times\Omega$ and denote the difference by
$R_h=[D^2\Phi(w_h)-D^2\Phi(w)]\nabla w$ in
$L^2(Q_T;\mathbb R^{m\times3})$.
Every sequence $h_n\downarrow0$ has a subsequence on which
$w_{h_n}\to w$ almost everywhere, since $w_h\to w$ strongly
in $L^2(Q_T;\mathbb R^m)$. Along that subsequence,
$R_{h_n}\to0$ pointwise by continuity of $D^2\Phi$, while
$|R_{h_n}|^2\le4M^2|\nabla w|^2\in L^1(Q_T;\mathbb R)$.
Dominated convergence proves $R_{h_n}\to0$ strongly in
$L^2(Q_T;\mathbb R^{m\times3})$. If the full limit failed,
there would exist a sequence $h_n\downarrow0$ for which the norms
stay bounded away from zero; the preceding subsequence argument
contradicts this. Thus the full family converges strongly.
The $L^2$ norms of the functions themselves converge because $D\Phi$
is Lipschitz, while $D\Phi(0)=0$ preserves $H_0^1$ traces. Therefore
$D\Phi(w_h)\to D\Phi(w)$ in $L^2(0,T;V)$, and
\begin{displaymath}
\int_s^t(A_h,D\Phi(w_h))_H\,\mathrm d\tau
\longrightarrow
\int_s^t\langle A,D\Phi(w)\rangle_{V',V}\,\mathrm d\tau.
\end{displaymath}
For the $B$ term, use the exact decomposition
\begin{displaymath}
(B_h-B)\cdot D\Phi(w_h)
+B\cdot[D\Phi(w_h)-D\Phi(w)].
\end{displaymath}
The first term is controlled in the exact scalar-integrand norm by
\begin{displaymath}
\| (B_h-B)\cdot D\Phi(w_h)\|_{L^1(Q_T;\mathbb R)}
\le \|D\Phi\|_{L^\infty(\mathbb R^m;\mathbb R^m)}
\|B_h-B\|_{L^1(Q_T;\mathbb R^m)}\longrightarrow0.
\end{displaymath}
The second term converges in measure on $Q_T$ and is bounded in
absolute value by $2\|D\Phi\|_{L^\infty(\mathbb R^m;\mathbb R^m)}|B|$.
To make the weighted convergence explicit, given $\delta>0$,
\begin{displaymath}
\int_{Q_T}|B|\,|D\Phi(w_h)-D\Phi(w)|
\le\delta\|B\|_{L^1(Q_T;\mathbb R^m)}
+2\|D\Phi\|_{L^\infty(\mathbb R^m;\mathbb R^m)}\int_{\{|D\Phi(w_h)-D\Phi(w)|>\delta\}}|B|.
\end{displaymath}
The last integral tends to zero by convergence in measure and
absolute continuity of the integral of $|B|$.
Thus the scalar source integrands converge in $L^1(Q_T;\mathbb R)$;
their spatial integrals converge in $L^1(0,T;\mathbb R)$.

\medskip\noindent{Step 6: Recover all entropy endpoints.}
The map $t\mapsto w(t)$ is continuous from the compact interval
$[0,T]$ into the metric space $L^1(\Omega;\mathbb R^m)$, hence its
image $K=\{w(t):t\in[0,T]\}$ is compact. For any $\delta>0$,
choose finitely many $f_1,\ldots,f_N\in K$ such that every
$w(t)$ is within $\delta$ of some $f_i$. Since $S(h)$ is contractive
on $L^1(\Omega;\mathbb R^m)$,
\begin{displaymath}
\sup_{t\in[0,T]}\|S(h)w(t)-w(t)\|_{L^1(\Omega;\mathbb R^m)}
\le 2\delta+\max_{1\le i\le N}
\|S(h)f_i-f_i\|_{L^1(\Omega;\mathbb R^m)}.
\end{displaymath}
For fixed $\delta$ the finite maximum tends to zero by strong
continuity; then let $\delta\downarrow0$. We obtain
\begin{equation}\label{eq:uniformh}
\sup_{0\le t\le T}
\|S(h)w(t)-w(t)\|_{L^1(\Omega;\mathbb R^m)}\longrightarrow0.
\end{equation}
Since $\Phi$ is globally Lipschitz, the entropy values at $s$ and $t$
converge uniformly as $h\downarrow0$. Combining all limits in
\eqref{eq:chainh} proves \eqref{eq:sumchain} for every
$0\le s\le t\le T$. More precisely, define
\begin{displaymath}
g(t)=\langle A(t),D\Phi(w(t))\rangle_{V',V}
+\int_\Omega B(t,x)\cdot D\Phi(w(t,x))\,\mathrm dx.
\end{displaymath}
By Cauchy--Schwarz in time and the $L^2(0,T;V)$ bound on
$D\Phi(w)$, the first summand belongs to $L^1(0,T;\mathbb R)$;
the second also belongs to that space because $D\Phi$ is bounded
and $B\in L^1(Q_T;\mathbb R^m)$. The established identity reads
$E(t)=E(0)+\int_0^t g(\tau)\,\mathrm d\tau$ for every $t\in[0,T]$,
which is the absolutely continuous representative of $E$.
This also proves the formula for every pair $0\le s\le t\le T$.\qedhere
\end{proof}

\begin{remark}[A complementary temporal-convolution derivation of Lemma~\ref{lem:chain}] 
\label{rem:temporal-convolution}
A second proof of the full identity \eqref{eq:sumchain} uses temporal
convolution instead of spatial semigroup regularization. It requires
the same $C([0,T];L^1(\Omega;\mathbb R^m))$  hypothesis for exact endpoints and a justification
of bounded Sobolev tests for the mollified equation. Let
$w\in C([0,T];L^1)\cap L^2(0,T;V)$,
$w_t=A+B$ with $A\in L^2(0,T;V')$ and $B\in L^1(Q_T)$,
and let $\Phi$ satisfy the hypotheses of Lemma~\ref{lem:chain}.
On a compact interior time interval convolve $w,A,B$ with the same
smooth time mollifier $\rho_h$. Then $w_h\in C^1([s,t];H)$   and
$w_h'=A_h+B_h$ as spatial distributions. The identity may be
tested against $D\Phi(w_h)\in V\cap L^\infty(\Omega;\mathbb R^m)$: density of
uniformly bounded smooth approximants in $V\cap L^\infty$
passes the $V'$ term by strong convergence in $V$ and the $L^1(Q_T;\mathbb R^m)$  
term by dominated convergence. Applying the elementary $H$-valued
chain rule to $v\mapsto\int_\Omega\Phi(v)$ gives
\begin{displaymath}
 \int_\Omega\!\bigl[\Phi(w_h(t))-\Phi(w_h(s))\bigr]
 =\int_s^t\!\langle A_h,D\Phi(w_h)\rangle\,\mathrm d\tau
 +\int_s^t\!\int_\Omega B_h\cdot D\Phi(w_h).
\end{displaymath}
Here $w_h\to w$ in $L^2(s,t;V)$, $A_h\to A$ in $L^2(s,t;V')$,
and $B_h\to B$ in $L^1(Q_T)$ locally in time.
The bounded continuous Hessian gives strong $L^2(s,t;V)$  
convergence of $D\Phi(w_h)$, by the Sobolev composition formula,
while the $B$ products pass by boundedness of $D\Phi$ and
weighted convergence in measure against $L^1(Q_T)$.
The $C([0,T];L^1(\Omega;\mathbb R^m))$   hypothesis yields convergence of the entropy at every
interior endpoint, and the resulting $L^1(0,T;\mathbb R)$  right-hand side allows
$s\downarrow0$, $t\uparrow T$. This gives the all-time identity.
Scalar renormalization and temporal approximation in
\cite[Definition~2.2 and Eq.~(3.22)]{BlanchardMurat1997} and
\cite[Definition~1 and Lemma~1]{BlanchardMuratRedwane2001}, together
with the scalar sum-space rules of
\cite[Lemma~9 and Corollary~11]{AlvarezVazquezJDE2008} and
\cite[Lemma~6 and Corollary~9]{AlvarezVazquez2009}, are important
antecedents. The vectorial dual-derivative formulations in
\cite[Lemmas~A.19--A.20]{Jungel2025} impose different temporal
hypotheses; the present argument verifies the additional $L^1$-term
pairing and the exact-time traces explicitly.
\end{remark}

\begin{remark}[Why the chain rule is conditional]\label{rem:conditional}
There are two logically distinct hypotheses in Lemma~\ref{lem:chain}.
Its derivative decomposition $\partial_tw=A+B$, with
$A\in L^2(0,T;V')$ and
$B\in L^1((0,T)\times\Omega;\mathbb R^m)$, makes the entropy
calculation possible once the chain rule is established. Its additional
requirement $w\in C([0,T];L^1(\Omega;\mathbb R^m))$ supplies actual
values at every time and permits the uniform endpoint limit
\eqref{eq:uniformh}. The first set of conditions is {not} being used
here to deduce the second.

In the existence proof continuity in $L^1(\Omega;\mathbb R^m)$ of the {constructed}
limit will be proved independently by estimate~\eqref{eq:cauchyL1};
only afterwards do we apply the lemma and establish continuity in
$L^2(\Omega;\mathbb R^m)$. On the other hand, if $U$ and $V$ are two {arbitrary}
competing energy solutions in the theorem's prescribed class, then
$U-V\in C([0,T];L^2(\Omega;\mathbb R^m))$, and
\begin{displaymath}
\|U(t)-U(s)\|_{L^1(\Omega;\mathbb R^m)}
\le |\Omega|^{1/2}\|U(t)-U(s)\|_H.
\end{displaymath}
Thus their difference has the continuity required by the lemma without
being assumed to arise from any approximation. This is why no circular
use of the yet-to-be-proved $C([0,T];H)$ regularity occurs.
\end{remark}

\begin{remark}[Which regularization is needed for which entropy?]
The following comparison explains why the preceding proof was given
in full generality, even though certain convex entropies also admit
a direct argument. The spatial Dirichlet heat-semigroup proof of
Lemma~\ref{lem:chain} establishes the vector chain rule for every
$\Phi$ satisfying \eqref{eq:phi}, without any convexity assumption.
There are two further routes, with different scopes.

For convex $\Phi$, one can work with temporal differences and the
supporting-hyperplane inequalities
$D\Phi(a)\cdot(b-a)\leq\Phi(b)-\Phi(a)
\leq D\Phi(b)\cdot(b-a)$.
Bounded Sobolev approximation first extends the distributional
identity to tests in $V\cap L^\infty(\Omega;\mathbb R^m)$; Steklov averages then handle
separately the $L^2(0,T;V')$ and $L^1(Q_T)$ terms. The latter requires
convergence against an integrable weight, while $C([0,T];L^1(\Omega;\mathbb R^m))$   continuity
recovers all temporal endpoints. This argument applies to the convex
radial entropies $\Phi_\varepsilon$ and $\Psi_R$ used below.
Related scalar weak-entropy and temporal comparison methods occur in
\cite[Eq.~(14), p.~387; Section~4]{KnabnerOtto2000}; see also the
scalar sum-space chain rules in
\cite[Corollary~11, p.~145]{AlvarezVazquezJDE2008} and
\cite[Corollary~9, p.~143]{AlvarezVazquez2009}.

The general, possibly nonconvex, vector identity also follows by
temporal convolution, with the bounded-test argument and the two
different derivative spaces made explicit in
Remark~\ref{rem:temporal-convolution}. The proof by the Dirichlet
semigroup remains useful because it handles the two summands of
$w_t$ in their natural spatial spaces.
\end{remark}

\begin{corollary}[Diffusion and $L^1(Q_T;\mathbb R^m)$  sources]\label{cor:diffusion}
If in Lemma~\ref{lem:chain} $A=\nu\Delta_D w$ and
$B=F\in L^1((0,T)\times\Omega;\mathbb R^m)$, then
\begin{equation}\label{eq:entropy}
\begin{split}
\int_\Om\Phi(w(t))&+\nu\int_s^t\!\int_\Om
\sum_{j=1}^3 D^2\Phi(w)[\partial_j w,\partial_j w]\dd x\dd\tau\\
&=\int_\Om\Phi(w(s))+\int_s^t\!\int_\Om F\cdot D\Phi(w)\dd x\dd\tau.
\end{split}
\end{equation}
\end{corollary}
\begin{proof} $ $\\
For $v\in V$, the Dirichlet Laplacian is the element of $V'$ defined by
\begin{displaymath}
\langle\Delta_Dv,\psi\rangle_{V',V}
=-\int_\Omega\nabla v:\nabla\psi\,\mathrm dx
\qquad(\psi\in V).
\end{displaymath}
Here the admissible test is $\psi=D\Phi(w)$. By $D\Phi(0)=0$ and
the uniform Hessian bound, Sobolev composition places this test in
$L^2(0,T;H_0^1(\Omega;\mathbb R^m))$ and gives
$\partial_j[D\Phi(w)]=D^2\Phi(w)\,\partial_jw$ in
$L^2(Q_T;\mathbb R^m)$ for $j=1,2,3$. Here is a direct
justification. Approximate $w$ in $L^2(0,T;V)$ by smooth
compactly supported space--time functions $w_k$. The classical
chain rule gives the identity for each $w_k$. The function values
$D\Phi(w_k)$ converge strongly in $L^2(Q_T;\mathbb R^m)$
by global Lipschitz continuity; for gradients, split
$D^2\Phi(w_k)\nabla w_k-D^2\Phi(w)\nabla w$ into a bounded
Hessian times $\nabla(w_k-w)$ and the Hessian difference times
$\nabla w$. The former tends to zero in $L^2$, and the latter
by the subsequence/dominated-convergence argument of Step~5.
Since $D\Phi(0)=0$, composition of zero-trace approximants
has zero trace, and closure in $V$ establishes the assertion.
Consequently,
\begin{displaymath}
\begin{aligned}
\langle\nu\Delta_Dw,D\Phi(w)\rangle_{V',V}
&=-\nu\int_\Omega\nabla w:\nabla[D\Phi(w)]\,\mathrm dx\\
&=-\nu\int_\Omega\sum_{j=1}^3
D^2\Phi(w)[\partial_jw,\partial_jw]\,\mathrm dx.
\end{aligned}
\end{displaymath}
The last integrand is absolutely integrable in space--time because its
absolute value is at most
$\|D^2\Phi\|_{L^\infty(\mathbb R^m;\mathcal L(\mathbb R^m,\mathbb R^m))}|\nabla w|^2$.
Substituting this identity and $B=F$ into \eqref{eq:sumchain}, and
moving the negative diffusion contribution to the left-hand side,
gives exactly \eqref{eq:entropy} for every $0\le s\le t\le T$.
\end{proof}

\section{Bounded coefficients and quantitative vectorial stability}\label{sec:bounded}

We now need a source of the time continuity required by the chain
rule. For bounded coefficients, the usual Hilbert-space theory gives
energy bounds without any dependence on the size of the skew
potential. More importantly, comparing solutions with different
bounded coefficients produces an $L^1(\Omega;\mathbb R^m)$  estimate controlled by the
difference of those coefficients. It is this {comparison}, rather
than compactness alone, that will select a single continuous limit
when the coefficient becomes singular.

In the bounded problem the product of coefficient and solution has
enough dual-space regularity for the familiar energy calculation.
After passage to the rough limit, only an $L^1$ product is available;
we will then return to the conditional entropy identity instead of
reusing the classical energy test without justification.

\begin{lemma}[Classical auxiliary problem]\label{lem:bounded}
If $D\in L^\infty(\QT;\mathbb R^{m\times m})$ is skew-symmetric a.e. and $U_0\in\HH$, the Dirichlet problem $v_t-\nu\Delta v+Dv=0$, $v(0)=U_0$, has a unique solution
\begin{displaymath}
v\in C([0,T];\HH)\cap L^2(0,T;\VV),\qquad v_t\in L^2(0,T;\VV'),
\end{displaymath}
with
\begin{equation}\label{eq:boundedenergy}
\frac12\|v(t)\|_{H}^2+\nu\int_s^t\|\nabla v\|_{L^2(\Omega;\mathbb R^{m\times3})}^2
=\frac12\|v(s)\|_{H}^2\qquad(0\le s\le t\le T).
\end{equation}
In particular,
\begin{equation}\label{eq:uniformgradient}
\|v\|_{L^\infty(0,T;\HH)}\le\|U_0\|_{H},
\qquad\|v\|_{L^2(0,T;L^6(\Om;\Rn))}\le
\frac{S_\Om}{\sqrt{2\nu}}\|U_0\|_{H}.
\end{equation}
\end{lemma} 
\begin{proof} $ $\\
We include the classical Galerkin argument for completeness, in order
to make clear which estimates are independent of
$\|D\|_{L^\infty((0,T)\times\Omega)}$.

\medskip\noindent{Step 1: Finite-dimensional problem.}
Let $(e_k)$ be an $H$-orthonormal basis of Dirichlet eigenvectors of
$L=-\Delta_D$; its span is dense in $V$ and in $H$. Write
$v_N(t)=\sum_{k=1}^N d_k(t)e_k$, and set
$v_N(0)=P_NU_0$, with $P_N$ the corresponding spectral projection.
For every $1\le j\le N$, solve
\begin{displaymath}
(\partial_tv_N,e_j)_H+\nu(\nabla v_N,\nabla e_j)_{L^2(\Omega)}
+(D(t)v_N,e_j)_H=0.
\end{displaymath}
The resulting coefficient matrix $M_N(t)$ is measurable and bounded
on $(0,T)$. The finite-dimensional system
$d'(t)=M_N(t)d(t)$ is equivalent to the Volterra equation
$d(t)=d(0)+\int_0^tM_N(s)d(s)\,\mathrm ds$.
Its successive-approximation series is absolutely and uniformly
convergent on $[0,T]$: the $k$th iterated integral has norm bounded
by $\|d(0)\|(T\|M_N\|_\infty)^k/k!$.
The resulting solution is absolutely continuous, and uniqueness
follows from the integral inequality and Gronwall's lemma. Thus
no continuity of $D$ in time is required. The comparable Galerkin
construction for measurable rough coefficients is carried out in
\cite[Lemma~11, proof, pp.~144--146]{AlvarezVazquez2009}; our ODE
argument is given here in full.

\medskip\noindent{Step 2: Uniform energy bounds.}
Multiply each equation by $d_j(t)$, sum over $j$, and use
$v_N\cdot Dv_N=0$ almost everywhere, because $D^\top=-D$. Then
\begin{displaymath}
\frac12\frac{\mathrm d}{\mathrm dt}\|v_N(t)\|_H^2
+\nu\|\nabla v_N(t)\|_{L^2(\Omega;\mathbb R^{m\times3})}^2=0.
\end{displaymath}
Integrating from $s$ to $t$ and using $\|P_NU_0\|_H\le\|U_0\|_H$
gives
\begin{displaymath}
\sup_{0\le t\le T}\|v_N(t)\|_H\le\|U_0\|_H,
\qquad
\|\nabla v_N\|_{L^2((0,T)\times\Omega;\mathbb R^{m\times3})}
\le\frac{\|U_0\|_H}{\sqrt{2\nu}}.
\end{displaymath}
By the Sobolev inequality on $V$, applied at almost every time,
\begin{displaymath}
\|v_N\|_{L^2(0,T;L^6(\Omega;\mathbb R^m))}
\le S_\Omega\|\nabla v_N\|_{L^2((0,T)\times\Omega;\mathbb R^{m\times3})}
\le \frac{S_\Omega}{\sqrt{2\nu}}\|U_0\|_H.
\end{displaymath}
Notice that none of these constants involves $\|D\|_\infty$.

\medskip\noindent{Step 3: Weak limits and the initial condition.}
The preceding estimates yield, upon extraction,
$v_N\rightharpoonup v$ in $L^2(0,T;V)$ and
$v_N\rightharpoonup^*v$ in $L^\infty(0,T;H)$.
For fixed bounded $D$, the maps $v\mapsto Dv$ are bounded on
$L^2((0,T)\times\Omega;\mathbb R^m)$, so
$Dv_N\rightharpoonup Dv$ in this space. Passing to the Galerkin
identities first against finite linear combinations of eigenvectors
and then by density produces
$\partial_tv=\nu\Delta_Dv-Dv$ in $L^2(0,T;V')$.
The same integrated Galerkin identity contains the initial term
$(P_NU_0,\psi)_H$; letting $N\to\infty$ gives $(U_0,\psi)_H$ for all
$\psi\in V$. Hence the $V'$-continuous representative of $v$ has
initial value $U_0$, and the Hilbert-triple theorem below identifies
this equality in $H$.

\medskip\noindent{Step 4: Strong traces, energy and uniqueness.}
The Hilbert-triple time-continuity result
\cite[Vol.~I, Ch.~1, Proposition~2.1, p.~18, and Theorem~3.1,
p.~19]{LionsMagenes1972Vol1} states that
$u\in L^2(0,T;V)$ and $u'\in L^2(0,T;V')$ imply
$u\in C([0,T];H)$. The accompanying identity
$\frac12\frac{\mathrm d}{\mathrm dt}\|u(t)\|_H^2
=\langle u'(t),u(t)\rangle_{V',V}$ almost everywhere
follows by the Hilbert-space product rule and density in the evolution
space \cite[Vol.~I, Ch.~1, Theorem~2.1, p.~11]{LionsMagenes1972Vol1}.
Applying these facts to $v\in L^2(0,T;V)$ and
$\partial_tv\in L^2(0,T;V')$ gives
$v\in C([0,T];H)$ and
\begin{displaymath}
\frac12\frac{\mathrm d}{\mathrm dt}\|v(t)\|_H^2
=\langle\partial_tv(t),v(t)\rangle_{V',V}
\quad\text{for a.e. }t.
\end{displaymath}
This theorem is applicable {here} because $Dv\in L^2(0,T;H)$:
for bounded $D$ there is no problem of the rough-product kind present
in the main theorem. Substituting the equation, the skew term vanishes and yields
\eqref{eq:boundedenergy} at all endpoints. The bounds in
\eqref{eq:uniformgradient} follow as follows. Weak lower
semicontinuity gives
$\|\nabla v\|_{L^2(Q_T;\mathbb R^{m\times3})}
\le\liminf_N\|\nabla v_N\|_{L^2(Q_T;\mathbb R^{m\times3})}$;
the trace identity, now justified for $v$, yields the exact
gradient bound from $\frac12\|U_0\|_H^2$.
Sobolev--Poincar\'e then gives the stated $L^2(0,T;L^6)$ bound.
The initial trace is recovered by testing the integrated Galerkin
formula with $\eta(t)\psi(x)$ where $\eta(T)=0$ and
$\eta(0)\ne0$, passing to the limit using
$P_NU_0\to U_0$ in $H$, and matching the boundary term at $0$.
Applying the identity to
the difference of two bounded-coefficient solutions with equal initial
data yields uniqueness. The construction and proof are complete.
\end{proof}

\begin{remark}[The classical form-bounded contribution to domination]
Before deriving quantitative stability, it is useful to separate
what follows from an existing invariance theorem from the new
singular-coefficient arguments. In the uniformly form-bounded regime, the scalar Dirichlet heat
semigroup dominates the {Euclidean modulus} of the vectorial
evolution. Proposition~\ref{prop:product-cone} checks the product-cone
conditions in the published nonautonomous convex-invariance criterion
\cite[Theorem~2.2]{ArendtDierOuhabaz2014}; the calculation is a
specific application of that theorem, not a new general invariance
principle. Autonomous precursors include
\cite[Theorem~4 and Proposition~7]{Ouhabaz1999} and
\cite{Shigekawa1997}. Uniform form-boundedness is not guaranteed for
the full singular mixed region.
\end{remark}

The next proposition makes the classical comparison explicit. It
verifies the hypotheses of the nonautonomous product-cone criterion
in a bounded-form subcase; it is not being used as a substitute for
the later singular-range entropy argument.

\begin{proposition}[Bounded-form heat domination from convex invariance]
\label{prop:product-cone}
Suppose $C\in L^\infty(0,T;L^{3/2}(\Omega;\mathbb R^{m\times m}))$
is real and skew-symmetric. Let $U_C(t,s):H\to H$ denote the evolution family obtained from
the nonautonomous form theorem \cite[Chapter~IV, \S1,
Theorem~1.1]{Lions1961}: for $f\in H$, $U_C(\cdot,s)f$
is the unique solution on $[s,T]$ with initial value $f$ at $s$
for $\partial_tu-\nu\Delta_Du+Cu=0$ with zero Dirichlet
boundary data. Then for every $f\in H$ and $0\le s\le t\le T$,
\begin{equation}\label{eq:ado-cone-domination}
  |U_C(t,s)f|\le S(t-s)|f|\quad\text{a.e. on }\Omega,
  \qquad \|U_C(t,s)f\|_{L^1}\le\|f\|_{L^1},
\end{equation}
where $S$ is the scalar Dirichlet heat semigroup of diffusivity $\nu$.
\end{proposition}
\begin{proof}
Use the product Hilbert triple
$\widehat H=H\times L^2(\Omega)$ and
$\widehat V=V\times H_0^1(\Omega)$. Define the block form
\begin{displaymath}
 \widehat a_\tau((u,z),(v,h))=
 \nu\!\int_\Omega\nabla u:\nabla v+
 \nu\!\int_\Omega\nabla z\cdot\nabla h+
 \int_\Omega(C(\tau)u)\cdot v.
\end{displaymath}
H\"older and $H_0^1(\Omega;\mathbb R^m)\hookrightarrow L^6(\Omega;\mathbb R^m)$   make this form uniformly
bounded on $\widehat V\times\widehat V$, while skew symmetry
preserves coercivity. Consider the closed convex product cone
\begin{displaymath}
 \mathcal K=\{(u,z)\in\widehat H:|u(x)|\le z(x)
   \text{ for almost every }x\in\Omega\}.
\end{displaymath}
The orthogonal projection $P$ is the pointwise projection onto the
second-order cone in $\mathbb R^{m+1}$. Writing $\rho=|u|$, it is
\begin{displaymath}
 P(u,z)=\begin{cases}
 (u,z),&z\ge\rho,\\
 (0,0),&z\le-\rho,\\
 ((\rho+z)u/(2\rho),(\rho+z)/2),&-\rho<z<\rho,
 \end{cases}
\end{displaymath}
with the usual continuous interpretation when $\rho=0$.
A metric projection onto a closed convex set is nonexpansive and
fixes the origin here, so $P(\widehat V)\subset\widehat V$ by
Sobolev composition. Moreover, firm nonexpansiveness of the
finite-dimensional projection gives, for every spatial direction $j$,
\begin{displaymath}
 \partial_j(PX)\cdot\partial_j(X-PX)\ge0
 \quad\text{a.e.},\qquad X=(u,z)\in\widehat V.
\end{displaymath}
To justify this assertion at nondifferentiability points of $P$,
apply $(P(a)-P(b))\cdot((a-Pa)-(b-Pb))\ge0$ to spatial difference
quotients and pass to their strong local $L^2_{\mathrm{loc}}(\Omega;\mathbb R^{m+1})$   limits. If
$PX=(p,p_0)$, then $p$ and $u-p$ are collinear almost everywhere,
so $(Cp)\cdot(u-p)=0$. Consequently
$\widehat a_\tau(PX,X-PX)\ge0$ for almost every $\tau$.
The invariance theorem
\cite[Theorem~2.2]{ArendtDierOuhabaz2014} implies that a
product evolution starting in $\mathcal K$ stays there. Start at
time $s$ from $(f,|f|)\in\widehat H$, whose scalar component
solves the heat equation, and obtain the first inequality in
\eqref{eq:ado-cone-domination}; the scalar heat $L^1(\Omega)$  contraction
gives the second. This deduction uses the previously published
nonautonomous convex-invariance theorem and should not be
presented as a new bounded-coefficient contraction principle.
\end{proof}

\noindent This classical domination also explains the $L^1(\Omega;\mathbb R^m)$  
contraction behind the next quantitative stability estimate. The latter
adds a forcing term generated by changing the coefficient; its proof
retains the complete entropy calculation rather than relying on an
unverified passage from an autonomous semigroup result.

\begin{lemma}[Stability in the vector $L^1(\Omega;\mathbb R^m)$  norm]\label{lem:stability}
Let $D,E\in L^\infty((0,T)\times\Omega;\mathbb R^{m\times m})$
be skew-symmetric almost everywhere, and let $v,z$ denote the unique
solutions of Lemma~\ref{lem:bounded} for the respective initial data
$v_0,z_0\in H$. Then for every $0\le s\le t\le T$, 
\begin{equation}\label{eq:stability}
\|v(t)-z(t)\|_{L^1(\Omega;\Rn)}\le\|v(s)-z(s)\|_{L^1(\Omega;\Rn)}+
\int_s^t\!\int_\Om |D-E|\,|z|\dd x\dd\tau.
\end{equation}
In particular, for equal initial data $U_0$,
\begin{equation}\label{eq:cauchyL1}
\sup_{t\in[0,T]}\|v(t)-z(t)\|_{L^1(\Omega;\Rn)}\le
\frac{S_\Om}{\sqrt{2\nu}}\|U_0\|_{H}
\|D-E\|_{L^2(0,T;L^{6/5}(\Om;\mathbb R^{m\times m}))}.
\end{equation}
\end{lemma} 
\begin{proof} $ $\\
Write $w=v-z$ and $G=(E-D)z$. Then $w_t=\nu\Delta_D w-Dw+G$ and $G\in L^1(Q_T;\mathbb R^m)$ because $D-E$ is bounded and
$z\in L^2(Q_T;\mathbb R^m)$ on a finite cylinder. Also
$w\in C([0,T];H)$ and, for any $a,b\in H$,
$\|a-b\|_{L^1(\Omega;\mathbb R^m)}
\le|\Omega|^{1/2}\|a-b\|_H$ by Cauchy--Schwarz.
Consequently $w\in C([0,T];L^1(\Omega;\mathbb R^m))$  and $w\in L^2(0,T;\VV)$, so Corollary \ref{cor:diffusion} applies to
\begin{displaymath}
\Phi_\varepsilon(\xi)=\sqrt{|\xi|^2+\varepsilon^2}-\varepsilon,
\quad D\Phi_\varepsilon(\xi)=\frac\xi{\sqrt{|\xi|^2+\varepsilon^2}},
\quad
D^2\Phi_\varepsilon(\xi)=\frac I{\sqrt{|\xi|^2+\varepsilon^2}}-
\frac{\xi\otimes\xi}{(|\xi|^2+\varepsilon^2)^{3/2}}.
\end{displaymath}

This regularized Euclidean modulus is chosen because it converges to
$|\xi|$ while its gradient remains bounded by one; unlike the
nondifferentiable modulus it satisfies the hypotheses of
Corollary~\ref{cor:diffusion}. The earlier scalar truncated-modulus
method appears, with a different entropy, in
\cite[Lemma~9 and the uniqueness proof, pp.~140--148]{AlvarezVazquezJDE2008}
and \cite[Corollary~9, pp.~143--144]{AlvarezVazquez2009}; neither
source is being credited with our particular vector formula.
For $\xi\ne0$, the Hessian has $(m-1)$ tangential eigenvalues
$(|\xi|^2+\varepsilon^2)^{-1/2}$ and the radial eigenvalue
$\varepsilon^2(|\xi|^2+\varepsilon^2)^{-3/2}$; at $\xi=0$
all eigenvalues equal $1/\varepsilon$. Hence it is positive
semidefinite and $D\Phi_\varepsilon(0)=0$.

Apply Corollary~\ref{cor:diffusion} to $w$, with the source
$F=-Dw+G\in L^1(Q_T;\mathbb R^m)$. For every $0\le s\le t\le T$,
\begin{displaymath}
\begin{aligned}
\int_\Omega\Phi_\varepsilon(w(t))\,\mathrm dx
&+\nu\int_s^t\!\int_\Omega
 \sum_{j=1}^3 D^2\Phi_\varepsilon(w)[\partial_jw,\partial_jw]
 \,\mathrm dx\,\mathrm d\tau\\
&=\int_\Omega\Phi_\varepsilon(w(s))\,\mathrm dx
 +\int_s^t\!\int_\Omega[-Dw+G]\cdot D\Phi_\varepsilon(w)
 \,\mathrm dx\,\mathrm d\tau.
\end{aligned}
\end{displaymath}
Because $D^\top=-D$ and $D\Phi_\varepsilon(w)$ is parallel to
$w$, one has $(Dw)\cdot D\Phi_\varepsilon(w)=0$ almost everywhere.
The diffusion term is nonnegative, while
$|G\cdot D\Phi_\varepsilon(w)|\le |G|$.
Therefore
$\int_\Omega\Phi_\varepsilon(w(t))\le
\int_\Omega\Phi_\varepsilon(w(s))+
\int_s^t\int_\Omega|D-E|_{\mathrm F}|z|$.
At fixed $s,t$, the bounds $0\le\Phi_\varepsilon(\xi)\le|\xi|$
and convergence $\Phi_\varepsilon(\xi)\to|\xi|$ permit
Lebesgue dominated convergence using $w(s),w(t)\in L^1(\Omega;\mathbb R^m)$;
this proves \eqref{eq:stability}.
For equal initial data take $s=0$. Mixed H\"older and Sobolev--Poincar\'e yield
\begin{displaymath}
\begin{aligned}
\int_{Q_T}|D-E|_{\mathrm F}|z|
&\le\|D-E\|_{L^2(0,T;L^{6/5}(\Omega;\mathbb R^{m\times m}))}
\|z\|_{L^2(0,T;L^6(\Omega;\mathbb R^m))}\\
&\le \frac{S_\Omega}{\sqrt{2\nu}}\|U_0\|_H
\|D-E\|_{L^2(0,T;L^{6/5}(\Omega;\mathbb R^{m\times m}))},
\end{aligned}
\end{displaymath}
which is exactly \eqref{eq:cauchyL1}.
\end{proof}

\section{Construction of the singular-coefficient solution}\label{sec:existence}

The approximate equations provide uniform energy bounds, but a weak
limit by itself would leave two questions open: whether the singular
product has the expected value, and whether the candidate solution
has a trace at every time. The coefficient-comparison estimate
answers the second question first, producing a limit in
$C([0,T];L^1(\Omega;\mathbb R^m))$. We then identify the product,
pass to the distributional equation, and only afterward apply the
conditional chain rule. This ordering is what turns the formal
cancellation into an exact identity for the limiting solution.
\begin{proof}[Existence part of Theorem~\ref{thm:main}] $ $\\
We divide the proof into five steps. The first produces a {single}
trajectory in $C([0,T];L^1(\Omega;\mathbb R^m))$ from bounded skew approximations. The second
identifies the singular product in the strong
$L^1((0,T)\times\Omega;\mathbb R^m)$ topology, while the
third establishes the limiting equation. Only after these three steps
can the conditional chain rule be used. It then gives the exact energy
identity (Step~4), from which the stronger topology of $C([0,T];L^2(\Omega;\mathbb R^m))$ follows
(Step~5).

\medskip\noindent{Step 1: Bounded skew approximations, energy bounds, and the $L^1$ limit.}
Set $C_n=C\mathbf1_{\{|C|_{\mathrm F}\le n\}}$. Because the indicator
is scalar, $C_n^\top=-C_n$ and
$\|C_n\|_{L^\infty((0,T)\times\Omega;\mathbb R^{m\times m})}\le n$.
For almost every $t$, the scalar indicator
$\mathbf1_{\{|C(t,x)|_{\mathrm F}\le n\}}$ converges to one
for almost every $x$. Moreover
$$|C_n(t,x)-C(t,x)|_{\mathrm F}^{6/5}
\le |C(t,x)|_{\mathrm F}^{6/5}.$$
Dominated convergence in space therefore yields
$\|C_n(t)-C(t)\|_{L^{6/5}(\Omega;\mathbb R^{m\times m})}\to0$
for almost every $t$. The scalar squared function
$t\mapsto\|C_n(t)-C(t)\|_{L^{6/5}(\Omega;\mathbb R^{m\times m})}^2$
is bounded by the integrable majorant
$t\mapsto\|C(t)\|_{L^{6/5}(\Omega;\mathbb R^{m\times m})}^2$,
since $C_n-C=-C\mathbf1_{\{|C|_{\mathrm F}>n\}}$.
Dominated convergence in time now gives
\begin{displaymath}
C_n\longrightarrow C\qquad\text{in }
L^2(0,T;L^{6/5}(\Omega;\mathbb R^{m\times m})).
\end{displaymath}

The truncation is merely one particular admissible sequence: the
later stability estimate also applies to any sequence of bounded,
skew-symmetric matrix fields converging in the {norm}
$L^2(0,T;L^{6/5}(\Omega;\mathbb R^{m\times m}))$.
The proof of convergence of the particular truncations above is
what ensures such sequences exist. Let $U_n$ be the
solutions given by Lemma~\ref{lem:bounded}, all with the same initial
datum $U_0$. Their exact energy balances yield
\begin{equation}\label{eq:uniform}
\sup_n\left(\|U_n\|_{L^\infty(0,T;H)}
+\|U_n\|_{L^2(0,T;V)}\right)<\infty.
\end{equation}
For a fixed $n,m$, Lemma~\ref{lem:stability} estimates their difference,
with constant $S_\Omega\|U_0\|_H/\sqrt{2\nu}$, which
does not contain either
$\|C_n\|_{L^\infty(Q_T;\mathbb R^{m\times m})}$ or
$\|C_m\|_{L^\infty(Q_T;\mathbb R^{m\times m})}$.
This independence is essential because each truncation is bounded
but the sequence need not be uniformly bounded. Since the coefficients are Cauchy in the Bochner
norm, the approximating solutions are Cauchy in the complete space
$C([0,T];L^1(\Omega;\mathbb R^m))$. Denote their limit by $U$:
\begin{equation}\label{eq:cl1limit}
U_n\longrightarrow U\quad\text{in }C([0,T];L^1(\Omega;\mathbb R^m)),
\qquad U(0)=U_0\quad\text{in }L^1(\Omega;\mathbb R^m).
\end{equation}
This is the continuity needed later by Lemma~\ref{lem:chain}; we have
proved it without using that lemma on the singular limiting equation.

Weak compactness from \eqref{eq:uniform} provides a subsequence weakly
convergent in $L^2(0,T;V)$ and weak-star convergent in
$L^\infty(0,T;H)$. By \eqref{eq:cl1limit}, both distributional limits
are $U$, so $U$ belongs to both spaces, with
$\|U\|_{L^\infty(0,T;H)}\le\|U_0\|_H$.
This essential-in-time bound is not yet a pointwise trace assertion.
For {each fixed} $t\in[0,T]$, however, $U_n(t)\to U(t)$ in
$L^1(\Omega;\mathbb R^m)$ and
$\|U_n(t)\|_H\le\|U_0\|_H$. Fix $t\in[0,T]$. Since $U_n(t)\to U(t)$ strongly in
$L^1(\Omega;\mathbb R^m)$, some subsequence $U_{n_j}(t,x)$
converges to $U(t,x)$ for almost every $x$. Applying Fatou's lemma
to $|U_{n_j}(t,x)|^2$ and using the approximate energy bounds,
\begin{displaymath}
\int_\Omega|U(t,x)|^2\,\mathrm dx
\le\liminf_{j\to\infty}\int_\Omega|U_{n_j}(t,x)|^2\,\mathrm dx
\le\|U_0\|_H^2.
\end{displaymath}
Thus
\begin{displaymath}
U(t)\in H,\qquad \|U(t)\|_H\le\|U_0\|_H
\qquad\text{for every }t\in[0,T].
\end{displaymath}
The spatial subsequence is allowed to depend on $t$: the path $U$ is
already fixed by the uniform convergence in $C([0,T];L^1(\Omega;\mathbb R^m))$. In particular, the
limiting entropy will have a value in $L^2(\Omega;\mathbb R^m)$ at {every} endpoint.

\medskip\noindent{Step 2: Strong identification of $C_nU_n$.}
The zero-order term must converge as an ordinary integrable function,
not merely as a distribution. Split
\begin{equation}\label{eq:split}
\begin{aligned}
\|C_nU_n-CU\|_{L^1((0,T)\times\Omega;\mathbb R^m)}
&\le\|(C_n-C)U_n\|_{L^1((0,T)\times\Omega;\mathbb R^m)}\\
&\quad+\|C(U_n-U)\|_{L^1((0,T)\times\Omega;\mathbb R^m)}.
\end{aligned}
\end{equation}
For almost every $t$, Sobolev--Poincar\'e in dimension three gives
\begin{displaymath}
\|U_n(t)\|_{L^6(\Omega;\mathbb R^m)}
\le S_\Omega\|\nabla U_n(t)\|_{L^2(\Omega;\mathbb R^{m\times3})}.
\end{displaymath}
Squaring, integrating in $t$, and taking the square root yields
\begin{displaymath}
\|U_n\|_{L^2(0,T;L^6(\Omega;\mathbb R^m))}
\le \frac{S_\Omega}{\sqrt{2\nu}}\|U_0\|_H.
\end{displaymath}
Mixed H\"older, with spatial exponents $6/5$ and $6$ and temporal
exponents $2$ and $2$, therefore bounds the first summand in
\eqref{eq:split} by
\begin{displaymath}
\|(C_n-C)U_n\|_{L^1(Q_T;\mathbb R^m)}
\le\|C_n-C\|_{L^2(0,T;L^{6/5}(\Omega;\mathbb R^{m\times m}))}
\frac{S_\Omega}{\sqrt{2\nu}}\|U_0\|_H\longrightarrow0.
\end{displaymath}
For the second term, it would be wrong to
argue that $L^1$ convergence of $U_n$ alone absorbs an arbitrary
singular multiplier. We use the {fixed} coefficient $C$.

For any measurable $E\subset(0,T)\times\Omega$, H\"older gives
\begin{equation}\label{eq:uniformintegrability}
\int_E |C|_{\mathrm F}|U_n-U|\,\mathrm dx\,\mathrm dt
\le \|C\mathbf1_E\|_{L^2(0,T;L^{6/5}(\Omega;\mathbb R^{m\times m}))}
\|U_n-U\|_{L^2(0,T;L^6(\Omega;\mathbb R^m))}.
\end{equation}
The last factor is uniformly bounded. To check absolute continuity of the first factor, fix $M>0$ and define
\begin{displaymath}
C^{(M)}(t,x):=C(t,x)\mathbf1_{\{|C(t,x)|_{\mathrm F}\le M\}}.
\end{displaymath}
This is a bounded, skew-symmetric truncation of the {fixed}
limiting coefficient; it is unrelated to the approximation index $n$.
Write $C=C^{(M)}+(C-C^{(M)})$.
The tail tends to zero in the mixed norm. Moreover,
\begin{displaymath}
\|C^{(M)}\mathbf1_E\|_{L^2(0,T;L^{6/5}(\Omega;\mathbb R^{m\times m}))}
\le M|\Omega|^{1/3}|E|^{1/2}.
\end{displaymath}
Thus the right-hand side of \eqref{eq:uniformintegrability} becomes
uniformly small over small-measure sets. Since $U_n\to U$ in $C([0,T];L^1(\Omega;\mathbb R^m))$,
we also have $U_n\to U$ in $L^1(Q_T;\mathbb R^m)$, hence in
measure on the finite-measure cylinder $Q_T$. Because $C$ is fixed
and finite almost everywhere, $C(U_n-U)\to0$ in measure.
The estimate \eqref{eq:uniformintegrability} and the preceding
truncation argument show that the family
$\{|C(U_n-U)|\}_{n\ge1}$ is uniformly integrable:
for every $\varepsilon>0$ some $\delta>0$ makes
$$\sup_n\int_E|C(U_n-U)|<\varepsilon$$ whenever $|E|<\delta$.
The finite-measure Vitali criterion
\cite[Theorem~12.44, Section~12.7]{DriverVitali2010}
(convergence in measure plus uniform integrability implies
convergence in $L^1(Q_T;\mathbb R^m)$) therefore
gives strong convergence in $L^1(Q_T;\mathbb R^m)$.
The next displayed estimate is also a direct proof of the same fact,
so the argument does not depend on an unverified external version
of Vitali's theorem.
Alternatively, the same truncation yields the direct estimate
\begin{displaymath}
\begin{aligned}
\|C(U_n-U)\|_{L^1(Q_T;\mathbb R^m)}
&\le K\|C-C^{(M)}\|_{L^2(0,T;L^{6/5}(\Omega;\mathbb R^{m\times m}))}\\
&\quad+M T\sup_{t\in[0,T]}\|U_n(t)-U(t)\|_{L^1(\Omega;\mathbb R^m)},
\end{aligned}
\end{displaymath}
where $K=\sup_n\|U_n-U\|_{L^2(0,T;L^6(\Omega;\mathbb R^m))}<\infty$;
we first take $n\to\infty$, then $M\to\infty$. Both arguments give the same conclusion: the first part of
$C_nU_n-CU$ vanishes because the coefficient converges strongly in its
Bochner norm and $U_n$ is bounded in the matching Sobolev space;
the second vanishes because the remaining multiplier $C$ is fixed,
so its high-amplitude tail can be chosen independently of $n$.
Neither mere convergence in measure of $C_n$ nor a bound on its
mixed norm would suffice. We obtain
\begin{equation}\label{eq:products}
C_nU_n\longrightarrow CU\qquad\text{in }
L^1((0,T)\times\Omega;\mathbb R^m).
\end{equation}
The important conclusion is that the limiting product has the precise
space--time integrability required by the distributional equation.

\medskip\noindent{Step 3: Obtain the limit equation before testing it by entropies.}
For each $\varphi\in C_c^\infty((0,T)\times\Omega;\mathbb R^m)$,
\begin{displaymath}
-\int_0^T\!\int_\Omega U_n\cdot\partial_t\varphi
+\nu\int_0^T\!\int_\Omega\nabla U_n:\nabla\varphi
+\int_0^T\!\int_\Omega(C_nU_n)\cdot\varphi=0.
\end{displaymath}
The first integral converges by \eqref{eq:cl1limit}; the second by
weak $L^2(0,T;V)$ convergence; the third by \eqref{eq:products}.
Thus we have exactly \eqref{eq:weak} and its equivalent
\begin{equation}\label{eq:derivative}
\partial_tU=\nu\Delta_DU-CU\quad\text{in }
\mathcal D'((0,T)\times\Omega;\mathbb R^m),
\end{equation}
where $\nu\Delta_DU\in L^2(0,T;V')$ and
$CU\in L^1((0,T)\times\Omega;\mathbb R^m)$. No stronger
membership of $CU$ in $L^2(0,T;V')$ is asserted. Weak closedness of
$L^2(0,T;V)$ retains the homogeneous Dirichlet condition. Crucially,
\eqref{eq:cl1limit}, \eqref{eq:products} and \eqref{eq:derivative}
verify {all} hypotheses of Lemma~\ref{lem:chain} for the limit.

\medskip\noindent{Step 4: Exact energy by radial entropies.}
Choose for $R>0$ the function
\begin{equation}\label{eq:psiR}
\Psi_R(\xi)=R^2\bigl(\sqrt{1+|\xi|^2/R^2}-1\bigr).
\end{equation}
For each fixed $R$, its gradient is bounded, making it admissible in
the pairing with $CU\in L^1(Q_T;\mathbb R^m)$; in the limit $R\to\infty$ it approximates quadratic
energy: for every fixed $\xi$, rationalizing the square root gives
$\Psi_R(\xi)=|\xi|^2/(\sqrt{1+|\xi|^2/R^2}+1)$,
which tends to $|\xi|^2/2$ as $R\to\infty$ and never exceeds it.
For fixed $R$ the gradient has bound $R$; thus the product
$CU\cdot D\Psi_R(U)$ is integrable since $CU\in L^1(Q_T;\mathbb R^m)$.
This is why we never test the singular equation directly with $U$.
If $a=(1+|\xi|^2/R^2)^{-1/2}$, then
\begin{equation}\label{eq:psiHess}
D\Psi_R(\xi)=a\xi,\qquad
D^2\Psi_R(\xi)=aI-\frac{a^3}{R^2}\xi\otimes\xi.
\end{equation}
To verify the Hessian bound, differentiate $a$:
$Da(\xi)[\eta]=-a^3(\xi\cdot\eta)/R^2$. Hence
$D(a\xi)[\eta]=a\eta-a^3(\xi\cdot\eta)\xi/R^2$,
which is the asserted formula. If $\eta\perp\xi$, then
$D^2\Psi_R(\xi)\eta=a\eta$. In the radial direction
$\eta=\xi$, use $a^{-2}=1+|\xi|^2/R^2$ to find
$D^2\Psi_R(\xi)\xi=a^3\xi$.
At $\xi=0$ the formula gives $a=1$ and $\xi\otimes\xi=0$, hence $D^2\Psi_R(0)=I$ and all its $m$ eigenvalues equal one. Because $0<a\le1$,
all eigenvalues belong to $[0,1]$: as quadratic forms
$0\le D^2\Psi_R(\xi)\le I$.
Furthermore
$|D\Psi_R(\xi)|=a|\xi|\le R$ and $D\Psi_R(0)=0$.
Because $C^\top=-C$, the integrable product
$(CU)\cdot D\Psi_R(U)=a(CU)\cdot U$ vanishes pointwise. Corollary
\ref{cor:diffusion} now gives, for every $0\le s\le t\le T$,
\begin{equation}\label{eq:energyR}
\begin{aligned}
\int_\Omega\Psi_R(U(t))\,\mathrm dx
&+\nu\int_s^t\!\int_\Omega\sum_{j=1}^3
D^2\Psi_R(U)[\partial_jU,\partial_jU]\\
&=\int_\Omega\Psi_R(U(s))\,\mathrm dx.
\end{aligned}
\end{equation}
For each fixed vector $\xi$,
$0\le\Psi_R(\xi)=|\xi|^2/(\sqrt{1+|\xi|^2/R^2}+1)
\uparrow |\xi|^2/2$. Since $U(s)$ and $U(t)$ belong to $H$ at
every time by Step~1, dominated (or monotone) convergence applies to
both endpoint terms: for each fixed $r\in\{s,t\}$,
$U(r)\in L^2(\Omega;\mathbb R^m)$, and the pointwise convergence
$\Psi_R(U(r,x))\to|U(r,x)|^2/2$ is dominated by the integrable
function $|U(r,x)|^2/2$. Thus
$\int_\Omega\Psi_R(U(r))\to\frac12\|U(r)\|_H^2$.
The diffusion integrand converges pointwise to
$|\nabla U|^2$, while its absolute value is bounded by
$|\nabla U|^2\in L^1((0,T)\times\Omega)$; hence dominated convergence
gives \eqref{eq:energy} after passage to the limit in the
{fixed limiting solution} $U$:
\begin{displaymath}
\lim_{R\to\infty}
\int_s^t\!\int_\Omega\sum_{j=1}^3
 D^2\Psi_R(U)[\partial_jU,\partial_jU]
\,\mathrm dx\,\mathrm d\tau
=\int_s^t\!\int_\Omega|\nabla U|^2\,\mathrm dx\,\mathrm d\tau.
\end{displaymath}
The integrable majorant is $|\nabla U|^2$, not the gradients of
approximate solutions. This matters: weak convergence of
$\nabla U_n$ alone would yield only a lower-semicontinuity
inequality and might leave a dissipation defect. We avoid that
inference altogether by proving the entropy identity directly for
$U$, then letting $R\to\infty$. Since the entropy identity was valid
for every fixed $0\le s\le t\le T$ and the limits hold at those
endpoints, the resulting exact equality holds for every such pair. 

\medskip\noindent{Step 5: Recover strong temporal $L^2$ continuity.}
We have $U\in C([0,T];L^1(\Omega;\mathbb R^m))$ and a uniform $H=L^2(\Omega;\mathbb R^m)$ bound for the value
$U(t,\cdot)$ at each fixed $t\in[0,T]$. For a fixed $g\in H$, choose
$g_k\in L^\infty(\Omega;\mathbb R^m)$ with
$\|g_k-g\|_H\to0$. Then
\begin{displaymath}
|(U(t)-U(s),g)_H|
\le\|g_k\|_{L^\infty(\Omega;\mathbb R^m)}
\|U(t)-U(s)\|_{L^1(\Omega;\mathbb R^m)}
+2\|U_0\|_H\|g_k-g\|_H.
\end{displaymath}
For each $k$, the first term tends to zero as $t\to s$, after which
$k\to\infty$ removes the second. This proves weak continuity
$U\in C_w([0,T];H)$. With $s=0$, the exact energy law is
\begin{displaymath}
\|U(t)\|_H^2=\|U_0\|_H^2
-2\nu\int_0^t\|\nabla U(\tau)\|_{L^2(\Omega;\mathbb R^{m\times3})}^2
\,\mathrm d\tau,
\end{displaymath}
a continuous function of $t$. To prove strong continuity without invoking a theorem by name,
fix any $s\in[0,T]$ and let $t_n\to s$. Weak continuity gives
$(U(t_n),U(s))_H\to\|U(s)\|_H^2$, while the exact energy
identity gives $\|U(t_n)\|_H^2\to\|U(s)\|_H^2$ because
$t\mapsto\int_0^t\|\nabla U(\tau)\|_2^2\,\mathrm d\tau$
is continuous. Consequently,
\begin{displaymath}
\|U(t_n)-U(s)\|_H^2=
\|U(t_n)\|_H^2+\|U(s)\|_H^2
-2(U(t_n),U(s))_H\longrightarrow0.
\end{displaymath}
This proves $U\in C([0,T];L^2(\Omega;\mathbb R^m))$,
including one-sided continuity at both endpoints. At $t=0$, the representative already
satisfies $U(0)=U_0$ in $L^1$ by Step~1, and both are in $H$, so their
equality holds in $H$. Existence with every assertion of
Theorem~\ref{thm:main} is proved.\qedhere
\end{proof}

\section{Contraction and uniqueness among arbitrary energy solutions}\label{sec:uniqueness}

Existence gives us one solution, but it would not be enough to show
that two {approximating sequences} have the same limit. The
uniqueness assertion concerns any competitor satisfying the
distributional equation and the stated energy regularity. Its
difference from another solution has precisely the continuity and
integrability required by the chain rule. We can therefore establish
the Euclidean modulus estimate directly for that difference, with
no assumption about how either solution was constructed.
\begin{proof}[Contraction and uniqueness in Theorem~\ref{thm:main}] $ $\\
{Step 1: Start with two solutions in the stated class.}
Let $U$ and $V$ be any two energy--distributional solutions for the
same $C$ (with possibly different initial data), and put $W=U-V$.
Then $W\in C([0,T];H)\cap L^2(0,T;V)$, and the spatial embedding
$H\hookrightarrow L^1(\Omega;\mathbb R^m)$ yields
$W\in C([0,T];L^1(\Omega;\mathbb R^m))$. Since $W\in L^2(0,T;H_0^1(\Omega;\mathbb R^m))$,
Sobolev--Poincar\'e and H\"older yield
\begin{displaymath}
\begin{aligned}
\|CW\|_{L^1(Q_T;\mathbb R^m)}
&\le\|C\|_{L^2(0,T;L^{6/5}(\Omega;\mathbb R^{m\times m}))}
\|W\|_{L^2(0,T;L^6(\Omega;\mathbb R^m))}\\
&\le S_\Omega\|C\|_{L^2(0,T;L^{6/5}(\Omega;\mathbb R^{m\times m}))}
\|\nabla W\|_{L^2(Q_T;\mathbb R^{m\times3})}<\infty.
\end{aligned}
\end{displaymath}
Thus $CW\in L^1(Q_T;\mathbb R^m)$.
Subtracting their distributional equations gives
\begin{displaymath}
\partial_tW=\nu\Delta_DW-CW\quad\text{in }
\mathcal D'((0,T)\times\Omega;\mathbb R^m).
\end{displaymath}
Thus all the hypotheses of Corollary~\ref{cor:diffusion} are satisfied
{without} requiring either competitor to be a limit of approximate
problems.

\medskip\noindent{Step 2: Use the Euclidean radial entropy.}
For $\varepsilon>0$ choose the same entropy as in
Lemma~\ref{lem:stability},
\begin{displaymath}
\Phi_\varepsilon(\xi)=\sqrt{|\xi|^2+\varepsilon^2}-\varepsilon,
\qquad D\Phi_\varepsilon(\xi)
=\frac{\xi}{\sqrt{|\xi|^2+\varepsilon^2}}.
\end{displaymath}
Its Hessian has tangential eigenvalues
$(|\xi|^2+\varepsilon^2)^{-1/2}$ and radial eigenvalue
$\varepsilon^2(|\xi|^2+\varepsilon^2)^{-3/2}$, both nonnegative.
Therefore the diffusion term in \eqref{eq:entropy} is nonnegative.
Because $D\Phi_\varepsilon(W)$ is a scalar multiple of $W$ and
$C^\top=-C$, the integrable zero-order term obeys
$(-CW)\cdot D\Phi_\varepsilon(W)=0$ almost everywhere.
Dropping only the nonnegative diffusion contribution yields
\begin{displaymath}
\int_\Omega\Phi_\varepsilon(W(t))\,\mathrm dx
\le\int_\Omega\Phi_\varepsilon(W(s))\,\mathrm dx
\quad(0\le s\le t\le T).
\end{displaymath}

\medskip\noindent{Step 3: Take the entropy limit at both times.}
For each fixed $\xi$, $0\le\Phi_\varepsilon(\xi)\le|\xi|$ and
$\Phi_\varepsilon(\xi)\to|\xi|$. At every fixed $r\in[0,T]$,
$W(r,\cdot)\in L^1(\Omega;\mathbb R^m)$, and
$\Phi_\varepsilon(W(r,x))\to|W(r,x)|$ for almost every $x$,
with domination $0\le\Phi_\varepsilon(W(r,x))\le|W(r,x)|$.
The dominated convergence theorem therefore gives
$\int_\Omega\Phi_\varepsilon(W(r))\,\mathrm dx
\to\|W(r)\|_{L^1(\Omega;\mathbb R^m)}$ for each fixed $r$.
Choose any $0\le s\le t\le T$ and apply this conclusion at
$r=s$ and $r=t$ in the same established inequality; because that
inequality held for every pair of times before taking the limit,
we obtain
\eqref{eq:contract} for every $0\le s\le t\le T$, with the norm
of the {whole vector}, not componentwise absolute values.
If $U(0)=V(0)$, set $s=0$; nonnegativity forces
$W(t)=0$ almost everywhere on $\Omega$ at every $t$. This is uniqueness
in the entire declared energy--distributional class, and no larger
class lacking the stipulated time traces is claimed.\qedhere
\end{proof}

\noindent{Comparison of proof methods.} The preceding entropy proof
establishes contraction for arbitrary energy solutions in the base class.
An independent argument using a bounded backward adjoint and a linear
Green identity for sum-space derivatives is given after the mixed-norm
theorem in Proposition~\ref{prop:alternative-contraction}. It proves
contraction and uniqueness without replacing the entropy calculation
needed for the exact energy identity.

\section{Strong stability of the approximation}\label{sec:strong}

The construction already gives uniform-in-time convergence in spatial
$L^1$ and weak control of the gradients. The remaining task is to
upgrade this information to the natural energy topology. The crucial
point is that weak convergence does not imply convergence of gradient
norms. Here the exact energy identities for both the approximations
and their limit provide the missing equality of norms. We can then
use the Hilbert-space principle that weak convergence together with
convergence of norms implies strong convergence.
\begin{proof}[Strong approximation statement \eqref{eq:strongconv}]$ $\\
We proceed in eight steps, because none of the following upgrades is an
automatic consequence of weak $L^2(0,T;V)$ convergence. The essential
new input is the exact energy identity already proved for the limit;
we do not use the strong approximation result to derive that identity.

\medskip\noindent{Step 1: Keep track of the full sequence and its weak limits.}
We know from \eqref{eq:cl1limit} and \eqref{eq:products} that
$U_n\to U$ in $C([0,T];L^1(\Omega;\mathbb R^m))$ and
$C_nU_n\to CU$ in $L^1((0,T)\times\Omega;\mathbb R^m)$.
Every weakly convergent subsequence of the bounded family $(U_n)$ in
$L^2(0,T;V)$ must have the distributional limit $U$. Reflexivity then
implies that the {whole} sequence converges weakly to $U$ in this
Bochner space. At each fixed time $t$, convergence in spatial $L^1$
and the uniform $H$ bound similarly imply $U_n(t)\rightharpoonup U(t)$
in $H$ (first test against bounded functions, then approximate any
$H$ test by bounded functions).

\medskip\noindent{Step 2: Apply lower semicontinuity at fixed times.}
Write
\begin{displaymath}
\begin{aligned}
a_n(t)&=\int_0^t\|\nabla U_n(\tau)\|_{L^2(\Omega;\mathbb R^{m\times3})}^2\,\mathrm d\tau,\\
a(t)&=\int_0^t\|\nabla U(\tau)\|_{L^2(\Omega;\mathbb R^{m\times3})}^2\,\mathrm d\tau.
\end{aligned}
\end{displaymath}
Restriction of weak $L^2(0,T;V)$ convergence to $(0,t)$ gives
$\liminf_n a_n(t)\ge a(t)$. Weak convergence of the values at time
$t$ gives $\liminf_n\|U_n(t)\|_H^2\ge\|U(t)\|_H^2$.
Both statements hold for every fixed $t$, including the endpoints.

\medskip\noindent{Step 3: Use the common exact energy to force norm convergence.}
Both energies have the identical right-hand side, since all approximate
problems start at $U_0$:
\begin{displaymath}
\|U_n(t)\|_H^2+2\nu a_n(t)
=\|U_0\|_H^2
=\|U(t)\|_H^2+2\nu a(t).
\end{displaymath}
The lower bounds from Step~2 already add up to this same sum. Neither
nonnegative term can retain a strict excess. Consequently
\begin{equation}\label{eq:pointwise-norms}
a_n(t)\to a(t),\qquad
\|U_n(t)\|_H^2\to\|U(t)\|_H^2
\qquad\text{for every }t\in[0,T].
\end{equation}
This is {pointwise-in-time} convergence; uniform convergence needs
its own argument.

\medskip\noindent{Step 4: Uniformize the cumulative dissipation.}
Every $a_n$ is nondecreasing as a function of the physical time $t$,
whereas $a$ is continuous and nondecreasing because its integrand is
$L^1$ in time. Given $\varepsilon>0$, choose
$0=t_0<\cdots<t_N=T$ such that
$a(t_{j+1})-a(t_j)<\varepsilon$ for all $j$. For large $n$,
$|a_n(t_j)-a(t_j)|<\varepsilon$ at each of the finitely many nodes.
If $t_j\le t\le t_{j+1}$, then
\begin{displaymath}
\begin{aligned}
a_n(t)-a(t)
&\le a_n(t_{j+1})-a(t_j)\\
&\le |a_n(t_{j+1})-a(t_{j+1})|
+a(t_{j+1})-a(t_j)<2\varepsilon.
\end{aligned}
\end{displaymath}
The reverse inequality follows from $a_n(t)\ge a_n(t_j)$ and
$a(t)\le a(t_{j+1})$. Hence
$\sup_{t\in[0,T]}|a_n(t)-a(t)|\to0$, and the energy identity implies
uniform convergence of $\|U_n(t)\|_H^2$ to $\|U(t)\|_H^2$.
Only monotonicity in {time} is used, not monotonicity in $n$.

\medskip\noindent{Step 5: Obtain uniform weak convergence for fixed tests.}
For $g\in H$, choose bounded $g_k\to g$ in $H$. Uniform boundedness of
$U_n(t)$ and $U(t)$ in $H$ gives
\begin{displaymath}
\sup_{t\in[0,T]}|(U_n(t)-U(t),g)_H|
\le\|g_k\|_{L^\infty(\Omega;\mathbb R^m)}
\sup_t\|U_n(t)-U(t)\|_{L^1(\Omega;\mathbb R^m)}
+2\|U_0\|_H\|g_k-g\|_H.
\end{displaymath}
First let $n\to\infty$ for fixed $k$, and then let $k\to\infty$.
This proves uniform-in-time weak convergence against every {fixed}
$g$, but not yet against the moving test $U(t)$.

\medskip\noindent{Step 6: Control the moving test using a finite net.}
The range $K=U([0,T])$ is compact in $H$ because strong temporal
continuity was proved in Step~5 of the {existence} proof.
For $\delta>0$, select a finite $\delta$-net $g_1,\ldots,g_J$ of
$K$. For each $t$, choose an index $j(t)$ with
$\|U(t)-g_{j(t)}\|_H<\delta$. Then
\begin{displaymath}
\begin{aligned}
\sup_t|(U_n(t)-U(t),U(t))_H|
&\le\max_{1\le j\le J}\sup_t
|(U_n(t)-U(t),g_j)_H|\\
&\quad+2\|U_0\|_H\delta.
\end{aligned}
\end{displaymath}
The maximum tends to zero for each fixed net by Step~5; then
$\delta\downarrow0$ proves uniform convergence of the cross term.
The net depends on $U$ and $\delta$, not on $n$.

\medskip\noindent{Step 7: Obtain strong $C([0,T];H)$ convergence.}
The Hilbert identity
\begin{displaymath}
\|U_n(t)-U(t)\|_H^2
=\|U_n(t)\|_H^2-\|U(t)\|_H^2
-2(U_n(t)-U(t),U(t))_H
\end{displaymath}
combined with Steps~4 and~6 gives
$\sup_t\|U_n(t)-U(t)\|_H\to0$.
This is the strong temporal convergence claimed in \eqref{eq:strongconv}.

\medskip\noindent{Step 8: Obtain strong $L^2(0,T;V)$ convergence.}
Set $\mathcal G=L^2((0,T)\times\Omega;\mathbb R^{m\times3})$.
Taking $t=T$ in \eqref{eq:pointwise-norms} yields
$\|\nabla U_n\|_{\mathcal G}\to\|\nabla U\|_{\mathcal G}$.
Their gradients already converge weakly in this Hilbert space; weak
convergence together with convergence of norms implies strong
gradient convergence. Step~7 yields strong convergence in
$L^2((0,T)\times\Omega;\mathbb R^m)$. Therefore
$U_n\to U$ strongly in $L^2(0,T;V)$, completing the proof.\qedhere
\end{proof}

\section{The full mixed-norm coefficient range}\label{sec:mixed}

The base theorem was presented first for a mathematical reason:
it isolates the loss of admissibility of the energy test and
establishes the conditional chain rule, all-time identity and
unrestricted comparison in a single critical coefficient class.
These are the same analytic mechanisms needed for the general case,
not merely a special example. What remains is to show that the
mixed exponents still control the coefficient--solution product and
permit approximation without losing the endpoint bounds.
We now extend the argument, beginning with comparison in
$L^1(\Omega;\mathbb R^m)$, exact energy at the limit and strong
convergence. We now ask how far this argument can be
pushed when time and space integrability compensate for one another.
The answer is the full coefficient region in
Theorem~\ref{thm:mixed}. In this region, an energy-space multiplier
estimate controls $CU$, but the endpoint cases also require careful
approximation of the coefficient. In particular, an $L^\infty$ time
norm cannot be treated as though truncation converged strongly in
that norm. We give the proof rather than inferring the theorem from
product integrability alone.

\begin{theorem}[Mixed-norm well-posedness and coefficient stability]
\label{thm:mixed}
Let $\Omega\subset\mathbb R^3$ be a bounded Lipschitz domain,
$T,\nu>0$, and $m\ge1$. Set $H=L^2(\Omega;\mathbb R^m)$ and
$V=H_0^1(\Omega;\mathbb R^m)$. Suppose
\begin{equation}\label{eq:mixed-region}
 1\le r\le\infty,\qquad \frac65\le q\le\infty,\qquad
 \frac2r+\frac3q\le\frac72,\qquad \frac1\infty:=0,
\end{equation}
and let $C\in L^r(0,T;L^q(\Omega;\mathbb R^{m\times m}))$
be real and skew-symmetric almost everywhere. For each $U_0\in H$
there is exactly one energy--distributional solution
$U\in C([0,T];H)\cap L^2(0,T;V)$ with $CU\in
L^1((0,T)\times\Omega;\mathbb R^m)$, homogeneous Dirichlet
boundary condition and $U(0)=U_0$. The distributional identity
is \eqref{eq:weak-definition}. For $0\le s\le t\le T$,
\begin{equation}\label{eq:mixed-energy}
 \frac12\|U(t)\|_H^2+\nu\int_s^t
 \|\nabla U(\tau)\|_{L^2(\Omega;\mathbb R^{m\times3})}^2\dd\tau
 =\frac12\|U(s)\|_H^2.
\end{equation}
For any two such solutions $U,W$ with the same coefficient,
possibly different initial data, one has
\begin{equation}\label{eq:mixed-contraction}
 \|U(t)-W(t)\|_{L^1(\Omega;\mathbb R^m)}
 \le\|U(s)-W(s)\|_{L^1(\Omega;\mathbb R^m)}
 \quad(0\le s\le t\le T).
\end{equation}
Moreover, if $C_n$ are skew-symmetric and
$C_n\to C$ strongly in
$L^r(0,T;L^q(\Omega;\mathbb R^{m\times m}))$, and $U_n$ have
{the same} initial datum $U_0$, then
\begin{equation}\label{eq:mixed-stability}
 \begin{aligned}
 U_n&\longrightarrow U&&\text{in }C([0,T];H)\cap L^2(0,T;V),\\
 C_nU_n&\longrightarrow CU&&\text{in }L^1((0,T)\times\Omega;\mathbb R^m).
 \end{aligned}
\end{equation}
These conclusions include all exponent endpoints allowed by
\eqref{eq:mixed-region}; no approximability restriction is imposed on
the class in which uniqueness is claimed.
\end{theorem}

\begin{lemma}[The mixed energy multiplier]\label{lem:mixed-mult}
Set $\mathcal E=L^\infty(0,T;H)\cap L^2(0,T;V)$, with
$\|v\|_{\mathcal E}=\|v\|_{L^\infty(0,T;H)}+
\|v\|_{L^2(0,T;V)}$. If $(r,q)$ satisfies
\eqref{eq:mixed-region}, multiplication defines a bilinear map
\begin{displaymath}
L^r(0,T;L^q(\Omega;\mathbb R^{m\times m}))\times\mathcal E
\longrightarrow L^1((0,T)\times\Omega;\mathbb R^m)
\end{displaymath}
with
\begin{equation}\label{eq:mixed-mult}
\|Dv\|_{L^1(Q_T;\mathbb R^m)}\le K_{\Omega,T,r,q}
\|D\|_{L^r(0,T;L^q(\Omega;\mathbb R^{m\times m}))}
\|v\|_{\mathcal E}.
\end{equation}
\end{lemma}
\begin{proof}
For $6/5\le q\le2$ write $q'=q/(q-1)$ and
$\theta=3/q-3/2\in[0,1]$. The Sobolev interpolation inequality
applied to $v(t)\in V$ gives
\begin{displaymath}
\|v(t)\|_{L^{q'}(\Omega;\mathbb R^m)}\le K
 \|v(t)\|_H^{1-\theta}\|v(t)\|_V^\theta.
\end{displaymath}
For $\theta>0$, integration with $p_*=2/\theta$ yields
\begin{equation}\label{eq:mixed-interpolation}
\|v\|_{L^{p_*}(0,T;L^{q'}(\Omega;\mathbb R^m))}
\le K\|v\|_{L^\infty(0,T;H)}^{1-\theta}
\|v\|_{L^2(0,T;V)}^\theta.
\end{equation}
At $\theta=0$ this is the $L^\infty(0,T;H)$ bound and we
put $p_*=\infty$. The restriction on $(r,q)$ is equivalent to
$1/r+\theta/2\le1$, or $r'\le p_*$. H\"older in space and
on the finite time interval now proves \eqref{eq:mixed-mult}.
For $q\ge2$, $q'\le2$, and
$L^2(\Omega;\mathbb R^m)\hookrightarrow
L^{q'}(\Omega;\mathbb R^m)$ since $|\Omega|<\infty$.
Pair the resulting $L^\infty$ time bound with
$D\in L^1(0,T;L^q(\Omega))$, which follows from $r\ge1$ and
$T<\infty$. Both $q=\infty$ and $r=1$ are permitted.
\end{proof}

\begin{lemma}[Skew-symmetric bounded truncations]
\label{lem:mixed-trunc}
Assume $1\le r\le\infty$, $6/5\le q\le\infty$ and
$2/r+3/q\le7/2$, as in \eqref{eq:mixed-region}.
For $D\in L^r(0,T;L^q(\Omega;\mathbb R^{m\times m}))$
real and skew-symmetric, let
\begin{displaymath}
D_k(t,x)=D(t,x)\min\{1,k/|D(t,x)|_{\mathrm F}\},
\end{displaymath}
with the multiplier equal to $1$ at $D=0$. Then $D_k$ is
measurable, bounded and skew-symmetric. If $r<\infty$,
$D_k\to D$ strongly in $L^r(0,T;L^q(\Omega))$.
If $r=\infty$, one instead has
$D_k\to D$ strongly in
$L^2(0,T;L^{6/5}(\Omega;\mathbb R^{m\times m}))$.
In either case
\begin{equation}\label{eq:mixed-trunc-op}
\sup_{\|v\|_{\mathcal E}\le1}
\|(D_k-D)v\|_{L^1(Q_T;\mathbb R^m)}\longrightarrow0.
\end{equation}
\end{lemma}
\begin{proof}
Scalar truncation preserves $D^\top=-D$. For finite $r$ and
$q<\infty$, dominated convergence first in $x$ and then in $t$
applies. When $q=\infty$, for almost every $t$ the spatial
essential supremum of $|D(t)|$ is finite, and the pointwise spatial
truncation error has $L^\infty(\Omega)$ norm tending to zero;
its norm is bounded by $\|D(t)\|_{L^\infty(\Omega)}$, an
$L^r(0,T)$ function. Dominated convergence in time still applies.
For $r=\infty$ we make no false assertion of strong truncation in
the $L^\infty$ temporal norm: the finite measure of $Q_T$ and the
inequality $q\ge6/5$ imply $D\in L^2(0,T;L^{6/5}(\Omega))$.
Dominated convergence proves convergence there.
Finally use Lemma~\ref{lem:mixed-mult} for finite $r$ and its
$(2,6/5)$ instance for $r=\infty$.
\end{proof}

\begin{proof}[Proof of Theorem~\ref{thm:mixed}] $ $\\
We separate construction from the exact energy balance, since
using the unknown singular solution as a test would not be legitimate.

\smallskip\noindent{Step 1. Bounded-coefficient approximations.}
Let $C_k$ be as in Lemma~\ref{lem:mixed-trunc}.
For bounded $C_k$, Lemma~\ref{lem:bounded} supplies the unique
$U_k\in C([0,T];H)\cap L^2(0,T;V)$ with datum $U_0$ and
$\partial_tU_k\in L^2(0,T;V')$. Its energy identity is
\begin{equation}\label{eq:mixed-Ek}
\frac12\|U_k(t)\|_H^2+\nu\int_0^t\|\nabla U_k(\tau)\|_2^2\dd\tau
=\frac12\|U_0\|_H^2\quad(0\le t\le T).
\end{equation}
In particular, $(U_k)$ is bounded in $\mathcal E$.

\smallskip\noindent{Step 2. Strong temporal $L^1$ convergence.}
For $Z_{k\ell}=U_k-U_\ell$, distributionally
\begin{displaymath}
\partial_tZ_{k\ell}=\nu\Delta_DZ_{k\ell}-C_kZ_{k\ell}
 +(C_\ell-C_k)U_\ell,\qquad Z_{k\ell}(0)=0.
\end{displaymath}
The source lies in $L^1(Q_T;\mathbb R^m)$, and its norm is bounded by
$\varepsilon_{k\ell}\|U_\ell\|_{\mathcal E}$ with
$\varepsilon_{k\ell}\to0$, by \eqref{eq:mixed-trunc-op}.
The radial consequence of Lemma~\ref{lem:chain} (derived in
Lemma~\ref{lem:stability}) gives
\begin{equation}\label{eq:mixed-Cauchy}
\sup_{t\in[0,T]}\|U_k(t)-U_\ell(t)\|_{L^1(\Omega;\mathbb R^m)}
\le\|(C_\ell-C_k)U_\ell\|_{L^1(Q_T;\mathbb R^m)}\to0.
\end{equation}
Let $U$ denote the limit in $C([0,T];L^1(\Omega;\mathbb R^m))$.
The energy bound yields $U\in L^2(0,T;V)$ by weak compactness.
For each $t\in[0,T]$, $U_k(t)\to U(t)$ in $L^1(\Omega)$ and
$\|U_k(t)\|_H\le\|U_0\|_H$, so $U(t)\in H$ and
$U_k(t)\rightharpoonup U(t)$ in $H$. Approximate arbitrary elements
of $H$ by bounded functions to deduce that $U:[0,T]\to H$ is
weakly continuous, including its initial value $U(0)=U_0$.

\smallskip\noindent{Step 3. Identify the singular product.}
Lemma~\ref{lem:mixed-mult} shows that $CU\in L^1(Q_T;\mathbb R^m)$.
Moreover $(C_k-C)U_k\to0$ in $L^1(Q_T)$, by
Lemma~\ref{lem:mixed-trunc}. To handle $C(U_k-U)$ we do
{not} assume a product of weakly convergent sequences converges
strongly. Instead fix $\varphi\in C_c^\infty(Q_T;\mathbb R^m)$ and
consider the following exhaustive cases.

If $r<\infty$ and $q<2$, then $r>1$ and $q'>2$. The interpolation bound makes $(U_k)$ bounded in $L^{r'}(0,T;L^{q'}(\Omega;\mathbb R^m))$, with $1<r',q'<\infty$. Since $U_k\to U$ in $L^1(Q_T;\mathbb R^m)$, reflexivity identifies its weak limit in this mixed-norm space as $U$. Consequently, $\int_{Q_T}CU_k\cdot\varphi\to\int_{Q_T}CU\cdot\varphi$ for every $\varphi\in C_c^\infty(Q_T;\mathbb R^m)$, because $C^\top\varphi\in L^r(0,T;L^q(\Omega;\mathbb R^m))$ and the mixed-norm spaces are dual under the corresponding H\"older pairing.

If $r<\infty$ and $q\ge2$, then
$C^\top\varphi\in L^1(0,T;H)$ by the finite-measure
embedding of $L^q(\Omega)$ into $L^2(\Omega)$.
Strong $L^1(Q_T)$ convergence together with the uniform
$L^\infty(0,T;H)$ bound identifies the weak-star limit of
$U_k$ in $L^\infty(0,T;H)$ as $U$.
This includes $(r,q)=(1,2)$ and $q=\infty$.

Finally, if $r=\infty$, then
$C\in L^2(0,T;L^{6/5}(\Omega;\mathbb R^{m\times m}))$;
$U_k\rightharpoonup U$ in
$L^2(0,T;L^6(\Omega;\mathbb R^m))$, using the continuous
Sobolev embedding of $V$ into $L^6(\Omega;\mathbb R^m)$.
This provides the required pairing with $C^\top\varphi$.
All cases establish $C_kU_k\to CU$ in distributions.
Passing to the limit in the bounded equations yields
$\partial_tU=\nu\Delta_DU-CU$ in $\mathcal D'(Q_T;\mathbb R^m)$;
$U\in L^2(0,T;V)$ supplies the homogeneous Dirichlet trace.

\smallskip\noindent{Step 4. Exact energy and strong temporal traces.}
We may now apply Lemma~\ref{lem:chain}, and we do so only
{after} identifying the singular equation and its $C_tL^1$
representative. Its hypotheses hold with
$A=\nu\Delta_DU\in L^2(0,T;V')$ and
$B=-CU\in L^1(Q_T;\mathbb R^m)$.
Choose a smooth function $\chi:[0,\infty)\to[0,1]$ which is
$1$ on $[0,1]$ and $0$ on $[2,\infty)$, put
$h''=\chi$ and $h(0)=h'(0)=0$, and set
$\Phi_R(z)=R^2h(|z|/R)$.
Each $\Phi_R\in C^2(\mathbb R^m)$ has
$\Phi_R(0)=D\Phi_R(0)=0$ and bounded gradient and Hessian;
its Hessian eigenvalues are $\chi(|z|/R)$ and
$h'(|z|/R)/(|z|/R)$, both in $[0,1]$ (with the continuous
value $1$ at the origin). In particular,
\begin{displaymath}
	0\le\Phi_R(z)\le\frac12|z|^2,\quad
	0\preceq D^2\Phi_R(z)\preceq I,\quad
	\Phi_R(z)\to\frac12|z|^2,\quad D^2\Phi_R(z)\to I.
\end{displaymath}
Here $A\preceq B$ means that $B-A$ is positive semidefinite
for symmetric matrices $A,B\in\mathbb R^{m\times m}$.
Because $D\Phi_R(z)$ is collinear with $z$,
$CU\cdot D\Phi_R(U)=0$ almost everywhere by skew symmetry.
The convexity of $\Phi_R$ and the Sobolev chain rule yield
$\langle\nu\Delta_DU,D\Phi_R(U)\rangle
=-\nu\int_\Omega\sum_jD^2\Phi_R(U)
\partial_jU\cdot\partial_jU$ for almost every time.
Dominated convergence applies in space--time, with majorant
$|\nabla U|^2$, and at {both} endpoints because
$U(s),U(t)\in H$. It proves \eqref{eq:mixed-energy} for every
$0\le s\le t\le T$. The identity makes $t\mapsto\|U(t)\|_H$
continuous. As $U$ was already weakly continuous in $H$, it is
strongly continuous there, including at $0,T$.

\smallskip\noindent{Step 5. Uniqueness in the entire stated class.}
Let $U,W$ be {arbitrary} energy--distributional solutions
for $C$. Their difference $Z=U-W$ belongs to $C_tL^1\cap L^2_tV$,
and $CZ\in L^1(Q_T)$ by Lemma~\ref{lem:mixed-mult}.
Apply Lemma~\ref{lem:chain} to
$Z_t=\nu\Delta_DZ-CZ$ with radial entropies
$\Phi_\varepsilon(z)=\sqrt{\varepsilon^2+|z|^2}-\varepsilon$.
Since $D^2\Phi_\varepsilon\succeq0$,
$|D\Phi_\varepsilon|\le1$ and
$CZ\cdot D\Phi_\varepsilon(Z)=0$, the entropy does not
increase between any two times. Its endpoint limit as
$\varepsilon\downarrow0$ is the $L^1(\Omega;\mathbb R^m)$ norm,
which proves \eqref{eq:mixed-contraction} and uniqueness from
$Z(0)=0$. Neither argument assumes that $U$ or $W$ is an
approximating limit.

\smallskip\noindent{Step 6. Strong stability under coefficient changes.}
For $C_n\to C$ in the indicated mixed norm and the same $U_0$,
put $Z_n=U_n-U$. Its equation is
\begin{displaymath}
\partial_t Z_n=\nu\Delta_DZ_n-C_nZ_n+(C-C_n)U,
\qquad Z_n(0)=0.
\end{displaymath}
The forced radial estimate and Lemma~\ref{lem:mixed-mult} give
\begin{equation}\label{eq:mixed-L1-stability}
\|Z_n\|_{C([0,T];L^1(\Omega;\mathbb R^m))}
\le K_{\Omega,T,r,q}\|C_n-C\|_{L^r(0,T;L^q(\Omega))}
\|U\|_{\mathcal E}\longrightarrow0.
\end{equation}
The common energy identity and datum bound $U_n$ in $\mathcal E$.
Strong $L^1(Q_T)$ convergence forces every weak $L^2(0,T;V)$
accumulation point to be $U$, and $U_n(T)\rightharpoonup U(T)$ in $H$.
Weak lower semicontinuity supplies
\begin{displaymath}
\begin{aligned}
 \|\nabla U\|_{L^2(Q_T)}^2&\le
 \liminf_n\|\nabla U_n\|_{L^2(Q_T)}^2,\\
 \|U(T)\|_H^2&\le\liminf_n\|U_n(T)\|_H^2.
\end{aligned}
\end{displaymath}
Both energy sums equal $\|U_0\|_H^2/2$; consequently
$\|\nabla U_n\|_{L^2(Q_T)}\to\|\nabla U\|_{L^2(Q_T)}$,
and weak convergence implies $U_n\to U$ strongly in $L^2(0,T;V)$.
Using \eqref{eq:mixed-energy} for each time,
\begin{displaymath}
\sup_{t\in[0,T]}\big|\|U_n(t)\|_H^2-\|U(t)\|_H^2\big|
\le 2\nu\big(\|\nabla U_n\|_{L^2(Q_T)}+
\|\nabla U\|_{L^2(Q_T)}\big)
\|\nabla(U_n-U)\|_{L^2(Q_T)}\to0.
\end{displaymath}
If $U_n$ failed to converge in $C([0,T];H)$, there would be
times $t_n\to t_*$ where $\|U_n(t_n)-U(t_n)\|_H$
stays away from zero. Yet \eqref{eq:mixed-L1-stability} and the
uniform $H$ bound imply
$U_n(t_n)\rightharpoonup U(t_*)$ in $H$; the uniform
convergence of squared norms and continuity of $U:[0,T]\to H$
then force strong convergence in $H$, a contradiction.
Hence $U_n\to U$ in $C([0,T];H)$.

For the nonlinear products, write
\begin{displaymath}
 C_nU_n-CU=(C_n-C)U_n+C(U_n-U).
\end{displaymath}
The first term tends to zero in $L^1(Q_T)$ by the multiplier
estimate and the uniform energy bound; the second tends to zero
by the same estimate with the strong convergence in both factors
of the energy norm. Thus \eqref{eq:mixed-stability} holds.
\end{proof}

\begin{remark}[Why a second proof of contraction is useful]
The proof of Theorem~\ref{thm:mixed} uses the vectorial entropy identity
to control the Euclidean length of differences of solutions. There is
another route that treats contraction and uniqueness directly by an
adjoint calculation: a linear Green identity permits bounded backward
adjoints to act as tests; duality then yields forced $L^1$ contraction.
This alternative still needs a nontrivial step:
the weak equation may have an $L^1(Q_T)$ right-hand side, and the
backward adjoint must therefore be a {bounded} test. We spell out
the bounded-test density and the resulting linear Green identity rather
than invoking the ordinary Hilbert-triple product rule outside its
hypotheses. The argument covers all energy--distributional solutions in
the mixed range, not only limits of bounded-coefficient solutions.
It does {not} establish the general nonlinear chain rule, the exact
energy identity, or the strong coefficient/product convergence; those
conclusions still rest on the preceding proof.
\end{remark}

\begin{lemma}[Uniformly bounded Sobolev test approximation]
\label{lem:bounded-test-approx}
Every $h\in H_0^1(\Omega;\mathbb R^m)\cap
L^\infty(\Omega;\mathbb R^m)$ admits a sequence
$h_j\in C_c^\infty(\Omega;\mathbb R^m)$ such that
$h_j\to h$ strongly in $H_0^1$, $h_j\to h$ almost everywhere
along a subsequence, and
$\sup_j\|h_j\|_{L^\infty}\le K_m\|h\|_{L^\infty}$
with a constant depending only on $m$.
\end{lemma}
\begin{proof} $ $\\
Choose $g_j\in C_c^\infty(\Omega;\mathbb R^m)$ converging to
$h$ strongly in $H_0^1$ and, after a subsequence, almost
everywhere. Clip each component of $g_j$ at the fixed level
$M=\|h\|_{L^\infty}$ using the scalar 1-Lipschitz truncation
$T_M$. Standard Sobolev composition continuity gives
$T_M(g_j^{(i)})\to h^{(i)}$ strongly in $H_0^1(\Omega)$  for each $i$;
its zero trace and compact support are preserved, and its
absolute value is at most $M$. Mollify these finitely many
compactly supported $H_0^1(\Omega)$  functions in space with radii small
enough to keep supports inside $\Omega$, choosing radii so that
the additional $H_0^1(\Omega)$  errors tend to zero. Nonnegative
mollification preserves the componentwise $M$ bound, hence
$\|h_j\|_{L^\infty}\le\sqrt m\,M$. The convergent sequence
has an almost-everywhere convergent subsequence, as required.
\end{proof}

\begin{lemma}[Linear Green identity with a sum-space derivative]
\label{lem:linear-green}
Suppose $z\in C([0,T];H)\cap L^2(0,T;V)$ and
$\partial_tz=A+B$ in space--time distributions, where
$A\in L^2(0,T;V')$ and $B\in L^1(Q_T;\mathbb R^m)$.
Suppose also $\psi\in C([0,T];H)\cap L^2(0,T;V)
\cap L^\infty(Q_T;\mathbb R^m)$ and
$\partial_t\psi\in L^2(0,T;V')$.
Then, for every $0\le s\le t\le T$,
\begin{align}\label{eq:linear-green}
 (z(t),\psi(t))_H-(z(s),\psi(s))_H
 &=\int_s^t\langle A,\psi\rangle_{V',V}\,\mathrm d\tau
 +\int_s^t\!\int_\Omega B\cdot\psi\,\mathrm dx\,\mathrm d\tau
 +\int_s^t\langle\partial_t\psi,z\rangle_{V',V}\,\mathrm d\tau.
\end{align}
\end{lemma}
\begin{proof}
First, an equality of spatial distributions of the form
$f=a+b$, with $f\in H$, $a\in V'$ and $b\in L^1(\Omega)$,
extends to every test in $V\cap L^\infty$. Indeed, use
Lemma~\ref{lem:bounded-test-approx}; the $H$ and $V'$ pairings
pass by strong $V$ convergence and the $L^1(Q_T;\mathbb R^m)$  term by dominated
convergence along the bounded almost-everywhere sequence.

Fix $0<s<t<T$ and convolve $z,A,B$ in time on a slightly
larger interval in $(0,T)$ with a common smooth mollifier:
$z_h=\rho_h*_tz$, $A_h=\rho_h*_tA$ and $B_h=\rho_h*_tB$.
Then $z_h\in C^1([s,t];H)\cap L^2(s,t;V)$ and
$\partial_t z_h=A_h+B_h$ in spatial distributions for every
interior time. By the preceding test extension, for almost
every $\tau$ this identity may be paired with
$\psi(\tau)\in V\cap L^\infty(\Omega;\mathbb R^m)$   to obtain
\begin{displaymath}
 (\partial_tz_h,\psi)_H=\langle A_h,\psi\rangle_{V',V}
 +\int_\Omega B_h\cdot\psi\,\mathrm dx.
\end{displaymath}
The ordinary linear Hilbert-triple product identity for the
$C^1([s,t];H)$   path $z_h$ and the Lions-class path $\psi$ is
applicable; in particular $\partial_tz_h\in L^2(s,t;H)$
for each fixed $h$ and $\partial_t\psi\in L^2(s,t;V')$.
Hence
\begin{displaymath}
 (z_h(t),\psi(t))_H-(z_h(s),\psi(s))_H
 =\int_s^t(\partial_tz_h,\psi)_H\,\mathrm d\tau
 +\int_s^t\langle\partial_t\psi,z_h\rangle\,\mathrm d\tau.
\end{displaymath}
Insert the previous spatial identity and let $h\downarrow0$.
The convergences $z_h\to z$ in $C([s,t];H)\cap L^2(s,t;V)$,
$A_h\to A$ in $L^2(s,t;V')$, and $B_h\to B$ in
$L^1((s,t)\times\Omega)$ justify each term; for the $B_h$
pairing use the uniform $L^\infty$ bound of $\psi$.
The left-hand side has the prescribed $C([0,T];H)$  traces.
Finally let interior times approach $s=0$ or $t=T$.
Every integral on the right is absolutely integrable over
$(0,T)$, so the equality holds at all endpoints.
No step assigns $B$ to $V'$.
\end{proof}

\begin{lemma}[Adjoint maximum bound for bounded skew coefficients]
\label{lem:backward-adjoint}
Let $D\in L^\infty(Q_T;\mathbb R^{m\times m})$ be real skew,
let $t\in(0,T]$ and let $\varphi\in H\cap L^\infty(\Omega;\mathbb R^m)$.
There exists a unique backward Lions-class solution
$\psi\in C([0,t];H)\cap L^2(0,t;V)$,
$\partial_\tau\psi\in L^2(0,t;V')$, satisfying
\begin{displaymath}
 -\partial_\tau\psi-\nu\Delta_D\psi-D(\tau)\psi=0,
 \qquad\psi(t)=\varphi,
\end{displaymath}
and $\|\psi\|_{L^\infty((0,t)\times\Omega)}\le\|\varphi\|_\infty$.
\end{lemma}
\begin{proof}
Reverse time $\sigma=t-\tau$. The forward potential is
$-D(t-\sigma)$, again bounded and skew, so Lions' theorem
\cite[Chapter~IV, Theorem~1.1]{Lions1961} gives the unique
Lions-class evolution. With $R=\|\varphi\|_\infty$, the set
$K_R=\{v\in H:|v(x)|\le R\text{ a.e.}\}$ is closed and convex.
Its orthogonal projection is the radial clipping
$P_Rv=\min(1,R/|v|)v$ with $P_R(0)=0$. It is 1-Lipschitz
and fixes zero, hence $P_R(V)\subset V$.
The pointwise skew term
$(-D P_Rv)\cdot(v-P_Rv)$ vanishes. Writing $v=\rho e$
on $\{\rho>R\}$, the Sobolev chain rule yields
\begin{displaymath}
 \nabla(P_Rv):\nabla(v-P_Rv)
 =\frac R\rho\Bigl(1-\frac R\rho\Bigr)
   \sum_j|\partial_jv-e(e\cdot\partial_jv)|^2\ge0;
\end{displaymath}
outside this set the product is zero almost everywhere.
Thus \cite[Theorem~2.2]{ArendtDierOuhabaz2014} preserves $K_R$
throughout the forward evolution, including its endpoint times.
Returning to backward time gives the asserted bound.
\end{proof}

\begin{proposition}[Alternative forced $L^1$ contraction and uniqueness]
\label{prop:alternative-contraction}
Assume the mixed region \eqref{eq:mixed-region} and
$C\in L^r(0,T;L^q(\Omega;\mathbb R^{m\times m}))$ is real skew.
If $Z\in C([0,T];H)\cap L^2(0,T;V)$ satisfies
\begin{displaymath}
 \partial_tZ-\nu\Delta_DZ+CZ=F
 \quad\text{in }\mathcal D'(Q_T;\mathbb R^m),
 \qquad F\in L^1(Q_T;\mathbb R^m),
\end{displaymath}
then, for every $0\le s\le t\le T$,
\begin{equation}\label{eq:alternative-contraction}
 \|Z(t)\|_{L^1(\Omega;\mathbb R^m)}\le
 \|Z(s)\|_{L^1(\Omega;\mathbb R^m)}+
 \int_s^t\!\int_\Omega|F(\tau,x)|\,\mathrm dx\,\mathrm d\tau.
\end{equation}
Consequently, any two energy--distributional solutions for the
same coefficient and initial datum coincide, whether or not they
were obtained by approximation.
\end{proposition}
\begin{proof}
Let $C_k$ be the bounded skew truncations of
Lemma~\ref{lem:mixed-trunc}. Their essential property is
\begin{displaymath}
 \sup_{\|W\|_{\mathcal E}\le1}
  \|(C_k-C)W\|_{L^1(Q_T)}\longrightarrow0,
\end{displaymath}
where $\mathcal E$ is the energy space of
Lemma~\ref{lem:mixed-mult}. For finite $r$ the truncations
converge in the mixed norm; at $r=\infty$ the required
operator convergence instead follows from
$C_k\to C$ in $L^2(0,T;L^{6/5}(\Omega;\mathbb R^{m\times m}))$  by the $(2,6/5)$
multiplier estimate. No strong convergence in
$L^\infty(0,T;L^q(\Omega;\mathbb R^{m\times m}))$  is used.

Fix $s<t$ and a terminal $\varphi\in H\cap L^\infty$ with
$|\varphi|\le1$ almost everywhere. Let $\psi_k$ be the
backward adjoint from Lemma~\ref{lem:backward-adjoint} with
$D=C_k$. Apply Lemma~\ref{lem:linear-green} on $[s,t]$
with $z=Z$, $A=\nu\Delta_DZ\in L^2_tV'$, and
$B=-CZ+F\in L^1(Q_T)$, noting that the multiplier lemma
indeed gives $CZ\in L^1(Q_T)$. The Laplacian terms cancel,
and skew symmetry with the {negative} potential in the
backward equation gives exactly
\begin{displaymath}
 (Z(t),\varphi)_H-(Z(s),\psi_k(s))_H
 =\int_s^t\!\int_\Omega((C_k-C)Z)\cdot\psi_k
  +\int_s^t\!\int_\Omega F\cdot\psi_k.
\end{displaymath}
The adjoint is bounded by one independently of $k$ and of
$\varphi$, so
\begin{displaymath}
 (Z(t),\varphi)_H\le \|Z(s)\|_{L^1}
   +\|(C_k-C)Z\|_{L^1((s,t)\times\Omega)}
   +\|F\|_{L^1((s,t)\times\Omega)}.
\end{displaymath}
Let $k\to\infty$. The coefficient error tends to zero uniformly
in the terminal tests by the operator multiplier estimate.
Take the supremum over $|\varphi|\le1$: as $|\Omega|<\infty$,
all these tests lie in $H$, and the supremum equals
$\|Z(t)\|_{L^1}$. The interval $[s,t]$ is arbitrary and
the Green identity includes both endpoints; hence
\eqref{eq:alternative-contraction} follows for all
$0\le s\le t\le T$.
\end{proof}

\begin{remark}[Scope and dependence of the alternative proof]
The transposition argument uses classical bounded-form adjoint
existence, a convex-invariance maximum estimate, and the
{proved} auxiliary linear Green identity, together with the
mixed energy multiplier and its truncations. It is a distinct
proof of contraction and energy-class uniqueness, not an
invocation of a published theorem stating the entire singular
mixed-norm result. It does not by itself prove the nonlinear
vectorial entropy identity, the quadratic equality of energies,
or strong stability in $C([0,T];H)$, $L^2(0,T;V)$  and product $L^1(Q_T;\mathbb R^m)$.
\end{remark}

\begin{remark}[Heat domination at the singular limit]
\label{rem:mixed-domination}
For every solution $U$ of Theorem~\ref{thm:mixed},
$|U(t)|\le S(t)|U_0|$ almost everywhere in $\Omega$ for every
$t\in[0,T]$. Indeed, for the approximating solutions in the proof of
Theorem~\ref{thm:mixed}, Proposition~\ref{prop:product-cone} gives
$|U_k(t)|\le S(t)|U_0|$ almost everywhere for each $t$.
Since $U_k\to U$ in $C([0,T];L^1(\Omega;\mathbb R^m))$, for every
fixed $t$ the positive-part inequality
\begin{displaymath}
\int_\Omega\bigl(|U(t)|-S(t)|U_0|\bigr)_+\,\mathrm dx
\le\|U(t)-U_k(t)\|_{L^1(\Omega;\mathbb R^m)}
\longrightarrow0
\end{displaymath}
proves $|U(t)|\le S(t)|U_0|$ almost everywhere, including at $t=0$.
Uniqueness extends the bound from the constructed limit to every
solution of the stated energy--distributional class. The passage to
singular coefficients uses the approximation established here in
addition to the bounded-form invariance theorem
\cite{ArendtDierOuhabaz2014}.
\end{remark}

\subsection{An alternative derivation from scalar entropies}
\label{subsec:scalar-averaging}

The vectorial chain rule treats a nonlinear function of the whole
solution at once. There is another, geometrically appealing route:
first project a vector solution onto a direction $\omega$, use a
{scalar} entropy identity for that projection, and then average
over all directions. Rotational symmetry turns the average into a
radial entropy; its gradient is parallel to the original vector, so
the skew term again vanishes. The calculation recovers both the
quadratic energy and the Euclidean modulus, rather than just a
componentwise estimate.

We present this as a second derivation of the complete mixed-norm
theorem. Its entropy argument and limiting construction differ from
the main proof, but it uses the previously established
Lemma~\ref{lem:bounded}, Proposition~\ref{prop:product-cone} and
Lemma~\ref{lem:mixed-mult}. The reader should also keep one
distinction in view: the initial bounded-data argument works with
coefficients merely in $L^1(Q_T)$, whereas the passage to arbitrary
$L^2$ data requires the mixed-norm multiplier hypothesis. The
following scalar identity provides the starting point.

\paragraph{The scalar entropy identity and its domain of validity.}
Write $V_0=H_0^1(\Omega)$, and for $\delta>0$ define
\begin{equation}\label{eq:scalar-av-trunc}
 \gamma_\delta(a)=\max\{-1,\min\{1,a/\delta\}\},\qquad
 K_\delta(a)=\int_0^a\gamma_\delta(b)\,\mathrm db.
\end{equation}
For any real-valued $v\in C([0,T];L^1(\Omega))\cap
L^2(0,T;V_0)$ satisfying, in spatial distributions,
$v_t=a+b$, $a\in L^2(0,T;V_0')$, $b\in L^1(Q_T)$, one has
\begin{equation}\label{eq:scalar-av-chain}
 \int_\Omega K_\delta(v(t))-\int_\Omega K_\delta(v(s))
 =\int_s^t\langle a,\gamma_\delta(v)\rangle\,\mathrm d\tau
   +\int_s^t\!\int_\Omega b\,\gamma_\delta(v)
 \quad (0\le s\le t\le T).
\end{equation}
This is the scalar sum-space mechanism of
\cite[Lemma~9, pp.~140--144, and Corollary~11, p.~145]
{AlvarezVazquezJDE2008}. The published result concerns scalar
functions in its own boundary-space setting; the following argument
justifies precisely the Dirichlet--Lipschitz version used here,
without invoking its different domain hypotheses.

Here are the functional details on the Dirichlet--Lipschitz domain.
Fix $[s,t]\Subset(0,T)$, choose a temporal mollifier with support
shorter than the distance to $\{0,T\}$, and put
$v_h=\rho_h*_t v$, $a_h=\rho_h*_t a$ and $b_h=\rho_h*_t b$.
The distributional equation gives, for each interior time and
$\varphi\in C_c^\infty(\Omega)$,
\begin{displaymath}
 ((v_h)'(\tau),\varphi)_{L^2}
   =\langle a_h(\tau),\varphi\rangle
     +\int_\Omega b_h(\tau)\varphi.
\end{displaymath}
Indeed $v_h\in C^\infty([s,t];V_0)$ and
$(v_h)'\in C([s,t];V_0)\subset C([s,t];L^2)$.
For each such time, $\gamma_\delta(v_h)\in V_0\cap L^\infty$
and $|\gamma_\delta(v_h)|\le1$. To extend the displayed equality
to this test, use a sequence in $C_c^\infty(\Omega)$ with the same
uniform bound (up to an inessential constant), converging strongly
in $V_0$ and almost everywhere; truncate first and approximate by
interior cutoffs and smoothing. This uses only the zero Sobolev trace,
not $H^2$ elliptic regularity. The $a_h$ pairing converges by $V_0$
density and the $b_h$ pairing by dominated convergence. Thus
\begin{displaymath}
 ((v_h)',\gamma_\delta(v_h))_{L^2}
 =\langle a_h,\gamma_\delta(v_h)\rangle
    +\int_\Omega b_h\gamma_\delta(v_h).
\end{displaymath}
The left side is the ordinary Hilbert-space chain derivative of
$\int_\Omega K_\delta(v_h)$, since $v_h$ is time smooth in $L^2$
and $K_\delta$ is $C^1(\mathbb R;\mathbb R)$  with bounded Lipschitz derivative.
Integration over $[s,t]$ gives the regularized entropy identity.

As $h\downarrow0$, $v_h\to v$ in $L^2(s,t;V_0)$,
$a_h\to a$ in $L^2(s,t;V_0')$, and $b_h\to b$ in
$L^1((s,t)\times\Omega)$. For the {fixed} piecewise-linear
Lipschitz function $\gamma_\delta$ with $\gamma_\delta(0)=0$,
strong Sobolev composition continuity yields
$\gamma_\delta(v_h)\to\gamma_\delta(v)$ in $L^2(s,t;V_0)$.
The $V_0'$ term converges by duality. For the integrable term,
\begin{displaymath}
 \begin{split}
 \|b_h\gamma_\delta(v_h)-b\gamma_\delta(v)\|_{L^1}
 &\le \|b_h-b\|_{L^1}
 +\int_{(s,t)\times\Omega}|b|
       |\gamma_\delta(v_h)-\gamma_\delta(v)|\longrightarrow0,
 \end{split}
\end{displaymath}
where the last limit follows from convergence in measure,
the uniform bound $2$, and the absolute continuity of the integral
with weight $L^1((s,t)\times\Omega)$. The integrands
$\langle a,\gamma_\delta(v)\rangle$ and
$\int b\gamma_\delta(v)$ are in $L^1(0,T)$.
Finally, $v\in C([0,T];L^1)$ and the $1$-Lipschitz property of
$K_\delta$ make $t\mapsto\int K_\delta(v(t))$ continuous.
Letting interior endpoints tend to $0$ and $T$ and using the
absolute continuity of the right-hand integrals proves
\eqref{eq:scalar-av-chain} for {every} $0\le s\le t\le T$.
The $L^2(0,T;V_0')$  term is paired with a bounded $V_0$ test, the $L^1((s,t)\times\Omega)$  
term only with its $L^\infty((s,t)\times\Omega)$  bound, and time traces are supplied by
$C([0,T];L^1(\Omega))$. In particular, neither $bv\in L^1(Q_T)$  nor
$b\in L^2(0,T;V_0')$  has been assumed.

\begin{proof}[Alternative proof of Theorem~\ref{thm:mixed} by scalar averaging]
We give the entire construction, not just a new proof of
\eqref{eq:mixed-contraction}. Throughout, sphere integration uses the
rotationally invariant probability measure $\sigma$ on
$\mathbb S^{m-1}$; for $m=1$ it is the equally weighted measure on
$\{-1,1\}$. In particular,
\begin{equation}\label{eq:scalar-av-sphere}
 \int_{\mathbb S^{m-1}}\omega\otimes\omega\,\mathrm d\sigma(\omega)
 =\frac1m I_m.
\end{equation}
Consequently, for every $z\in\mathbb R^m$,
\begin{displaymath}
 m\int_{\mathbb S^{m-1}}\frac12(\omega\cdot z)^2
      \mathrm d\sigma=\frac12|z|^2,
 \qquad
 \frac1{a_m}\int_{\mathbb S^{m-1}}|\omega\cdot z|
      \mathrm d\sigma=|z|,
\end{displaymath}
where $a_m$ is defined in \eqref{eq:scalar-av-mod-entropy}.
Both formulas include the convention $m=1$.

\smallskip\noindent{Step 1. Scalar projections of an arbitrary solution.}
Suppose $Z\in C([0,T];H)\cap L^2(0,T;V)$  solves
$Z_t-\nu\Delta_D Z+CZ=F$ in distributions, where $C$ is real
skew-symmetric and $CZ,F\in L^1(Q_T)$. For fixed $\omega$ the
scalar projection $v_\omega=\omega\cdot Z$ satisfies the
hypotheses of \eqref{eq:scalar-av-chain} with
$a=\nu\Delta_D v_\omega$ and
$b=\omega\cdot(F-CZ)$; in particular $v_\omega\in C([0,T];L^1(\Omega))$.
Multiplying \eqref{eq:scalar-av-chain} by any fixed positive constant
and integrating in $\omega$ is justified by
$\nabla Z\in L^2(Q_T)$, $CZ,F\in L^1(Q_T)$ and the uniform
bounds on the scalar truncated gradients. The diffusion pairing
is obtained by the Sobolev chain rule in space.

\smallskip\noindent{Step 2. Exact energy, with no illegal quadratic test.}
Set $j_R(a)=R K_R(a)$ and
\begin{equation}\label{eq:scalar-av-energy-entropy}
 \Phi_R(z)=m\int_{\mathbb S^{m-1}}j_R(\omega\cdot z)
 \,\mathrm d\sigma(\omega),\qquad
 j_R(a)=\begin{cases}a^2/2,&|a|\le R,\\
                   R|a|-R^2/2,&|a|>R.
          \end{cases}
\end{equation}
This $C^1(\mathbb R^m;\mathbb R)$  convex entropy has bounded gradient. Rotational invariance
implies $D\Phi_R(z)=\lambda_R(|z|)z$ with the natural value at
$z=0$. Hence $D\Phi_R(z)\cdot Cz=0$ pointwise for skew $C$,
although $(CZ)\cdot Z$ need not be integrable. At almost every $z$
its Hessian satisfies
\begin{equation}\label{eq:scalar-av-hessian}
 0\le D^2\Phi_R(z)=m\int_{\{|\omega\cdot z|<R\}}
   \omega\otimes\omega\,\mathrm d\sigma(\omega)\le I_m,
 \qquad D^2\Phi_R(z)\longrightarrow I_m\quad(R\to\infty).
\end{equation}
The Hessian is understood almost everywhere in the argument $z$.
At a truncation level, the one-dimensional Sobolev gradient of
$\omega\cdot Z$ vanishes almost everywhere on its level set;
therefore the undefined values of $j_R''$ do not affect the
integrated diffusion expression. This observation also handles
$m=1$, when $\mathbb S^0=\{-1,1\}$. For fixed $z$, as
$R\uparrow\infty$, the sets $\{|\omega\cdot z|<R\}$ increase
to the full sphere, so the nonnegative quadratic diffusion
integrand tends to $|\nabla Z|^2$ and is bounded by it.
We {never} test the singular equation by $Z$: the only
coefficient pairing is
$\int (CZ)\cdot D\Phi_R(Z)$, which is integrable because
$D\Phi_R$ is bounded and $CZ\in L^1$, and is identically zero
by radiality and skewness. Integrating the projected scalar
identities gives, for $F=0$,
\begin{equation}\label{eq:scalar-av-energy-R}
 \int_\Omega\Phi_R(Z(t))
 +\nu\int_s^t\!\int_\Omega\sum_{j=1}^3
   D^2\Phi_R(Z)[\partial_jZ,\partial_jZ]
 =\int_\Omega\Phi_R(Z(s)).
\end{equation}
Indeed, the integrable potential was paired only with the
{bounded} vector $D\Phi_R(Z)$. Since
$0\le\Phi_R(z)\uparrow|z|^2/2$ and the diffusion integrand
is bounded by $|\nabla Z|^2$, monotone and dominated convergence,
respectively, yield \eqref{eq:mixed-energy} for every
$0\le s\le t\le T$. This holds for {every} admissible
solution and hence also for the difference of any two of them.
In particular, the difference satisfies the exact energy identity,
which gives uniqueness and energy-norm dependence on the initial
condition without any approximability hypothesis.

\smallskip\noindent{Step 3. Euclidean $L^1$ contraction with forcing.}
Put
\begin{equation}\label{eq:scalar-av-mod-entropy}
 a_m=\int_{\mathbb S^{m-1}}|\omega\cdot e_1|\,
   \mathrm d\sigma(\omega)>0,\qquad
 \Psi_\delta(z)=\frac1{a_m}\int_{\mathbb S^{m-1}}
      K_\delta(\omega\cdot z)\,\mathrm d\sigma(\omega).
\end{equation}
These entropies are radial and convex;
$0\le |z|-\Psi_\delta(z)\le\delta/(2a_m)$ and
$|D\Psi_\delta(z)|\le1$. For the last bound, radiality
reduces the question to the radial derivative:
$0\le\gamma_\delta(|z|\omega\cdot e)(\omega\cdot e)
\le|\omega\cdot e|$, and its spherical mean is $a_m$.
The potential again disappears; the diffusion term is
nonnegative because $\Psi_\delta$ is convex. The forcing is
bounded in absolute value by $|F|$. Taking $\delta\downarrow0$
yields, for arbitrary $0\le s\le t\le T$,
\begin{equation}\label{eq:scalar-av-forced}
 \|Z(t)\|_{L^1(\Omega;\mathbb R^m)}\le
 \|Z(s)\|_{L^1(\Omega;\mathbb R^m)}
 +\|F\|_{L^1((s,t)\times\Omega;\mathbb R^m)}.
\end{equation}
For $F=0$ this gives \eqref{eq:mixed-contraction} directly,
including every temporal endpoint and all competing energy solutions.

\smallskip\noindent{Step 4. Build solutions for bounded data using
only $C\in L^1(Q_T)$.}
Let $g\in H\cap L^\infty(\Omega;\mathbb R^m)$,
$M=\|g\|_{L^\infty}$, and truncate $C$ pointwise by
$C_k=C\min\{1,k/|C|_{\mathrm F}\}$, with the factor $1$ at
$C=0$. Then $C_k\to C$ in $L^1(Q_T)$ and every $C_k$
is bounded and skew-symmetric. Lemma~\ref{lem:bounded} gives the
solutions $U_k$ with the common datum $g$. The cone computation of
Proposition~\ref{prop:product-cone} yields
\begin{equation}\label{eq:scalar-av-bounded-dom}
 |U_k(t)|\le S(t)|g|\le M\quad\text{a.e. in }\Omega,
 \qquad 0\le t\le T.
\end{equation}
Apply \eqref{eq:scalar-av-forced} to $U_k-U_\ell$ with
coefficient $C_k$ and forcing $(C_\ell-C_k)U_\ell$.
Thus
\begin{equation}\label{eq:scalar-av-cauchy}
 \|U_k-U_\ell\|_{C_tL^1}\le
 M\|C_k-C_\ell\|_{L^1(Q_T)}\longrightarrow0.
\end{equation}
The common bound $|U_k|\le M$ implies
$\|U_k(t)-U_\ell(t)\|_H^2\le
2M\|U_k(t)-U_\ell(t)\|_{L^1}$, so $U_k\to U$
strongly in $C([0,T];H)$. Its spacetime limit satisfies
$|U|\le M$ almost everywhere. The bounded-problem energy
identities give bounded gradients in $L^2(Q_T)$; each weak
$L^2_tV$ cluster point equals $U$ because of the strong
$L^2(Q_T)$ convergence. Thus the {whole} sequence converges
weakly in $L^2_tV$ to $U$, retaining the zero Dirichlet trace.
In the identity
$C_kU_k-CU=(C_k-C)U_k+C(U_k-U)$, the first term has $L^1$
norm at most $M\|C_k-C\|_1$. For the second, convergence in
$C_tH$ implies convergence in measure and
$|C(U_k-U)|\le2M|C|$. Weighted Vitali convergence, using the
integrable dominating weight and uniqueness of limits in measure,
gives convergence in $L^1(Q_T)$ for the {whole} sequence.
Consequently $U$ satisfies the distributional equation with
coefficient $C$, the zero lateral trace, and $U(0)=g$.
The exact energy identity from Step~2 applies to {every}
admissible solution independently of this construction, hence to
this limit. Comparing it at $T$ with the bounded-problem energy
identities gives convergence of the squared $L^2(0,T;V)$  norms;
weak convergence then upgrades to $U_k\to U$ strongly in
$L^2(0,T;V)$. This bounded-data construction needs only
$C\in L^1(Q_T)$; it cannot by itself identify $CU$ for arbitrary
unbounded $L^2(\Omega;\mathbb R^m)$  initial data.

\smallskip\noindent{Step 5. General $H$ data and every mixed endpoint.}
For the stated mixed class, finite measure gives $C\in L^1(Q_T)$.
Take bounded radial truncations $g_j$ of an arbitrary datum
$U_0\in H$, with $g_j\to U_0$ in $H$, and construct the bounded-data
solutions $U^j$ as in Step~4. The identity of Step~2, applied to $Z=U^j-U^k$ with
the {same} coefficient and datum $g_j-g_k$, gives the
{two separate bounds}
\begin{equation}\label{eq:scalar-av-initial-continuity}
 \begin{aligned}
 \|U^j-U^k\|_{C_tH}^{2}
   &\le\|g_j-g_k\|_H^{2},\\
 2\nu\|\nabla(U^j-U^k)\|_{L^2(Q_T)}^{2}
   &\le\|g_j-g_k\|_H^{2}.
 \end{aligned}
\end{equation}
They follow by taking the supremum in the pointwise energy identity,
and separately by evaluating that identity at $t=T$ and dropping
its nonnegative terminal energy. The supremum and total
dissipation must {not} be added with the same right-hand
side: if combined, the valid bound has $2\|g_j-g_k\|_H^2$.
Thus $U^j$ converges strongly in $C([0,T];H)\cap L^2(0,T;V)$.
Lemma~\ref{lem:mixed-mult} gives
$\|C(U^j-U)\|_{L^1}\to0$, so the limit is a solution with
datum $U_0$ and integrable product. Step~2 supplies the exact
energy equality and unrestricted uniqueness; Step~3 supplies
contraction for any two such solutions. Every endpoint of
\eqref{eq:mixed-region} is covered by the multiplier already
proved: in particular $(r,q)=(1,2)$ uses the $L^\infty(0,T;H)$  
factor, $(2,6/5)$ uses $L^2(0,T;L^6(\Omega;\mathbb R^m))$, and
$(\infty,6/5)$ follows through the finite-cylinder
$L^2(0,T;L^{6/5}(\Omega;\mathbb R^{m\times m}))$  embedding. The cases $q=\infty$ and $r=1$
use only the available $L^1(0,T;L^q(\Omega;\mathbb R^{m\times m}))$ -in-time coefficient bound.
{No} convergence of truncations in $L^\infty(0,T;L^q(\Omega;\mathbb R^{m\times m}))$ 
has been assumed. The domination of Step~4 passes to general data
at time zero by $H$ approximation and heat-semigroup contraction;
restarting the equation at any time $s$ and using unrestricted
uniqueness yields $|U(t)|\le S(t-s)|U(s)|$ for $s\le t$.

\smallskip\noindent{Step 6. Strong coefficient and product stability.}
Suppose $C_n\to C$ strongly in the same mixed norm; then
$C_n\to C$ in $L^1(Q_T)$. For a fixed bounded datum $g$, the
corresponding solutions satisfy $|U_n^g|,|U^g|\le\|g\|_\infty$.
The forced estimate \eqref{eq:scalar-av-forced}, now with
coefficient $C_n$ and source $(C_n-C)U^g$, gives
\begin{equation}\label{eq:scalar-av-coeff-L1}
 \|U_n^g-U^g\|_{C_tL^1}\le
 \|g\|_\infty\|C_n-C\|_{L^1(Q_T)}\longrightarrow0.
\end{equation}
By the common $L^\infty(Q_T;\mathbb R^m)$  bound this implies strong convergence in
$C([0,T];H)$. Uniform energy estimates give weak gradient convergence;
the all-time energy identities at $T$ show convergence of gradient
norms and hence strong convergence in $L^2(0,T;V)$.
For an arbitrary datum $U_0\in H$, use the same bounded $g_j$
as in Step~5. Both bounds in \eqref{eq:scalar-av-initial-continuity} are
independent of the coefficient. With the explicitly specified energy norm
\begin{displaymath}
 \|W\|_{\mathcal E_*}:=
 \|W\|_{C([0,T];L^2(\Omega;\mathbb R^m))}
 +\|W\|_{L^2(0,T;H_0^1(\Omega;\mathbb R^m))},
\end{displaymath} they yield a constant $c_\nu$,
independent of $n,j$, such that
\begin{displaymath}
 \|U_n^{U_0}-U^{U_0}\|_{\mathcal E_*}
 \le 2c_\nu\|U_0-g_j\|_H
      +\|U_n^{g_j}-U^{g_j}\|_{\mathcal E_*}.
\end{displaymath}
First fix $j$ and let $n\to\infty$ using the bounded-data
coefficient estimate and energy equality; then let $j\to\infty$.
This proves $U_n\to U$ in $C([0,T];H)\cap L^2(0,T;V)$.
The order is: bounded-data $L^1$ contraction, exact energy,
strong gradient convergence, uniform data approximation, and
finally mixed-multiplier product convergence. In particular, no
$L^\infty(Q_T;\mathbb R^m)$   bound is claimed for solutions with general $H$ data.
For the final product estimate, set
\begin{displaymath}
 \|D\|_{\mathcal C_{r,q}}:=
 \|D\|_{L^r(0,T;L^q(\Omega;\mathbb R^{m\times m}))}
 \qquad(D:Q_T\to\mathbb R^{m\times m}).
\end{displaymath}
Then
\begin{align}\label{eq:scalar-av-products}
 \|C_nU_n-CU\|_{L^1(Q_T;\mathbb R^m)}
 &\le\|(C_n-C)U_n\|_{L^1(Q_T;\mathbb R^m)}
      +\|C(U_n-U)\|_{L^1(Q_T;\mathbb R^m)}\\
 &\le K\|C_n-C\|_{\mathcal C_{r,q}}\|U_n\|_{\mathcal E}
      +K\|C\|_{\mathcal C_{r,q}}\|U_n-U\|_{\mathcal E}
 \longrightarrow0,
\end{align}
by Lemma~\ref{lem:mixed-mult}, the uniform energy bound and the
just-proved strong energy convergence. The $C_n$ here may be
unbounded; the conclusion is not restricted to their bounded
truncations. This completes every assertion of
Theorem~\ref{thm:mixed}.
\end{proof}

\begin{remark}[What the alternative approaches explain]
\label{rem:comparison-methods}
Scalar averaging uses the bounded radial compositions required for
energy and modulus estimates, whereas Lemma~\ref{lem:chain}
treats more general vector compositions, including admissible
nonconvex ones. Remark~\ref{rem:temporal-convolution} changes only
the regularization; Proposition~\ref{prop:product-cone} supplies
heat domination {within} the bounded-form regime and before
singular passage; and adjoint transposition
(Proposition~\ref{prop:alternative-contraction}) gives contraction
and uniqueness, but not by itself energy equality or strong product
stability. The approaches have distinct hypotheses and outputs.
None extends automatically to non-skew potentials, unbounded
geometries or coefficients outside the stated multiplication range.
\end{remark}

\begin{remark}[Relation to the classical system theorem and the scope of the range]
\label{rem:mixed-lsu}
The generalized-solution theorem for parabolic systems of
\cite[Chapter~VII, \S1, Theorem~1.1, pp.~571--574]
{LadyzhenskayaSolonnikovUraltseva1968} is a genuine precursor, not
merely a similarity of exponent notation. Its system (1.1) becomes
$U_t-\nu\Delta U+CU=0$ by taking $a_{ij}=\nu\delta_{ij}$,
$A_i=B_i=f_i=f=0$ and the matrix $A=C$; parabolicity (1.2) and
the homogeneous source condition (1.6) are then satisfied. In
three dimensions, condition (1.5), p.~573, admits
\begin{displaymath}
 C\in L^r(0,T;L^q(\Omega)),\qquad
 \frac1r+\frac{3}{2q}=1,\quad
 \frac32<q\le\infty,\quad1\le r<\infty.
\end{displaymath}
Thus $(r,q)=(1,\infty),(2,3),(4,2)$ fall within the published
hypotheses. On finite cylinders, the same theorem also covers
coefficient spaces embedded into one of these classes, notably
$1/r+3/(2q)\le1$ with $q>3/2$ and $1\le r\le\infty$,
whenever the underlying spatial domain satisfies the hypotheses
of that book. The remaining endpoint
$C\in L^\infty(0,T;L^{3/2}(\Omega;\mathbb R^{m\times m}))$
is also covered by Lions' theory of nonautonomous coercive
forms \cite[Chapter~IV, \S1, Theorem~1.1]{Lions1961}:
the embedding $H_0^1(\Omega)\hookrightarrow L^6(\Omega)$
ensures uniform form boundedness, while skew symmetry
preserves coercivity.

\medskip\noindent{The generalized solution class.}\enspace
The solution-class comparison can be made explicitly on domains satisfying
the hypotheses of both theorems. In Chapter~I,
\S1, p.~6, the book's $V_2$ energy norm controls the essential
$L^\infty(0,T;L^2(\Omega;\mathbb R^m))$  norm and the $L^2(Q_T;\mathbb R^{m\times3})$  norm of spatial gradients.
Chapter~VII, \S1, equation~(1.4), p.~572, defines generalized
solutions through an integral identity for test functions in the
zero-lateral-trace $W^{1,1}_2(Q_T)$ (in LSU's notation, for $\mathbb R^m$-valued functions)  class. Our solutions have the
stronger $C([0,T];H)$ trace and the same spatial energy bounds.
On the displayed coefficient curve, set
$\vartheta=3/(2q)\in[0,1)$.
The energy interpolation inequality gives both $U$ and an admissible
test $\eta$ in
$L^{2/\vartheta}(0,T;L^{2q/(q-1)}(\Omega;\mathbb R^m))$, with the usual
$L^\infty(0,T;L^2(\Omega))$  interpretation when $q=\infty$.
Since $1/r+\vartheta=1$, H\"older's inequality implies  
\begin{displaymath}
 \int_{Q_T}|C|\,|U|\,|\eta|\,\mathrm dx\,\mathrm dt
 \le K\|C\|_{L^r(0,T;L^q(\Omega;\mathbb R^{m\times m}))}\|U\|_{\mathcal E}
       \|\eta\|_{\mathcal E}<\infty.
\end{displaymath}
Here is the necessary test-density step. For each $t\in(0,T]$,
the book's zero-lateral-trace $W^{1,1}_2(Q_t)$ tests have both
$\eta\in L^2(0,t;H_0^1(\Omega;\mathbb R^m))$ and
$\eta_t\in L^2(0,t;L^2(\Omega;\mathbb R^m))$, hence
$\eta\in C([0,t];H)\cap L^2(0,t;V)$.
On a bounded Lipschitz domain, finite sums of smooth temporal functions
times $C_c^\infty(\Omega;\mathbb R^m)$ functions are dense in this intersection
with its $L^2(0,t;V)\cap H^1(0,t;H)$   norm: extend and mollify in time,
then approximate spatially in $H_0^1(\Omega;\mathbb R^m)$ (and therefore in $H$).
This approximation preserves zero lateral trace and converges in
$C([0,t];H)$ at both temporal endpoints. For the diffusion and
$U\eta_t$ terms, use $L^2$  duality; for the zero-order term use
the displayed energy-multiplier bound, which is continuous for the
convergent test sequence; for the initial and terminal terms use
the $C([0,t];H)$  traces of $U$ and $\eta$. The distributional identity
extends first to smooth tests not vanishing at $0$ or $t$ by its
specified initial trace and integration in time, and then by this
density to every such $W^{1,1}_2(Q_t)$ (LSU notation, for vector-valued tests)  test. Therefore, within the
common domain and coefficient assumptions, {every} solution
in our energy--distributional class satisfies LSU's identity (1.4)
and belongs to its generalized $V_2$ class. Its uniqueness estimate
(1.7), p.~573, and the conclusion on p.~574 apply in this overlap.
The assertion is restricted to this overlap; it does not identify
the larger distributional solution classes outside these hypotheses.

\medskip\noindent{Where the hypotheses stop.}\enspace
The LSU integrability assumptions do not directly cover all of
\eqref{eq:mixed-region}: for example, $(r,q)=(1,2)$,
$(2,6/5)$, $(4/3,3/2)$ and $(\infty,6/5)$ are outside its
stated curve and its finite-cylinder embeddings. The space
$L^1(0,T;L^2(\Omega;\mathbb R^{m\times m}))$, is genuinely not included in the LSU coefficient
union: on a ball, $C(t,x)=t^{-3/4}|x|^{-1}{\bf1}_{|x|<1}J$
belongs to $L^1(0,T;L^2(\Omega;\mathbb R^{m\times m}))$, whereas its temporal singularity requires
$r<4/3$ and its spatial singularity requires $q<3$, inconsistent
with $1/r+3/(2q)=1$ for $3/2<q<3$.
LSU does not state the full strong coefficient/product stability
\eqref{eq:mixed-stability} on our region. Its radial truncation and temporal averaging on pp.~575--576
are further methodological antecedents.
The original formulation and coefficient restrictions are given in
\cite[Chapter~VII, \S1, (1.4)--(1.7), pp.~572--574]
{LadyzhenskayaSolonnikovUraltseva1968}; the $V_2$ and
$W^{1,1}_2(Q_T)$ (LSU's test space), conventions are those of Chapter~I, \S1, pp.~5--6.
\end{remark}

\begin{remark}[Uniformly bounded forms and vector heat domination]
\label{rem:classical-form-comparison}
The classical Lions theorem for measurable, uniformly bounded and
coercive non-autonomous forms yields existence and uniqueness in
$L^2(0,T;V)\cap H^1(0,T;V')$ for initial data in $H$ and
right-hand side in $L^2(0,T;V')$; this is a standard form-theoretic
statement, not an attribution of any additional vectorial entropy
conclusion to a particular reference. It applies here when
$C\in L^\infty(0,T;L^q(\Omega))$ with $q\ge3/2$:
$H^1_0(\Omega)\hookrightarrow L^6(\Omega)$ makes the zero-order
form uniformly bounded on $V\times V$, while real skewness gives
$\int_\Omega (CU)\cdot U=0$ and preserves coercivity.
In the homogeneous equation, an arbitrary distributional energy
solution then satisfies $CU\in L^2(0,T;V')$, hence
$U_t\in L^2(0,T;V')$, and enters that classical uniqueness class.
The same derivative-class conclusion with a nonzero forcing term
requires, for example, $F\in L^2(0,T;V')$.
The form theorem alone does not assert Euclidean $L^1$ contraction;
the bounded-form case is covered by the product-cone deduction in
Proposition~\ref{prop:product-cone}. The comparison does not cover the
full mixed region, especially $q<3/2$ or general finite $r$.
For related modern work on convex invariance for non-autonomous
forms, see \cite{ArendtDierOuhabaz2014}; that paper is not cited
here as a direct source of all the preceding conclusions.
For the classical evolution theorem and its dual-space source version,
see \cite[Chapter~IV, \S1, Theorem~1.1 and Remark~1.2]{Lions1961}.
For {autonomous} vector semigroups, the form criteria of
\cite[Theorem~4]{Ouhabaz1999}, and the associated work
\cite{Shigekawa1997}, yield Euclidean-modulus domination once the
forms and domain ideals have been checked. Nonautonomous domination
in $L^\infty(0,T;L^{3/2}(\Omega;\mathbb R^{m\times m}))$, follows from the product-cone application
of \cite[Theorem~2.2]{ArendtDierOuhabaz2014}, verified in
Proposition~\ref{prop:product-cone}.
Neither assertion extends by direct substitution to the endpoint
$C\in L^\infty(0,T;L^{6/5}(\Omega;\mathbb R^{m\times m}))$,: the off-diagonal pairing
$\int (Cu)\cdot v$ need not be finite for arbitrary $u,v\in V$;
Remark~\ref{rem:unbounded-pairing} supplies explicit functions.
For finite $r$ the fixed-in-time form bound may likewise fail.
These are failures of hypotheses of the {cited tools}, not
counterexamples to solvability or to vector $L^1(\Omega;\mathbb R^m)$ contraction.

For comparison, \cite[Sections~2--3, in particular
Propositions~2.7--2.8]{DobrickGlueck2024} discusses contraction
and dissipative matrix couplings under a {bounded}, autonomous
potential and Neumann boundary conditions. Its theorem is not a
direct substitute for the present time-dependent singular setting.
\end{remark}

\begin{remark}[Sharpness of multiplication on the energy space]
\label{rem:mixed-mult-sharpness}
The restriction $2/r+3/q\le 7/2$, together with $q\ge6/5$,
is optimal for the {uniform multiplier estimate on the whole
energy space}, when $m\ge2$; this is not an optimality claim for
well-posedness of the PDE. Indeed, fix a nonzero real skew matrix
$J$ supported in a $2\times2$ block, a ball compactly contained in
$\Omega$, a time $t_0\in(0,T)$, and smooth compactly supported
scalar functions $u,c$ with $cu\not\equiv0$. For small $\varepsilon>0$,
with supports contained in $Q_T$, put
\begin{displaymath}
 U_\varepsilon(t,x)=\varepsilon^{-3/2}
 u\bigl((t-t_0)/\varepsilon^2,(x-x_0)/\varepsilon\bigr)e_1,
 \qquad
 C_\varepsilon(t,x)=\varepsilon^{-2/r-3/q}
 c\bigl((t-t_0)/\varepsilon^2,(x-x_0)/\varepsilon\bigr)J,
\end{displaymath}
where $1/\infty=0$. Direct changes of variables give
\begin{displaymath}
\begin{gathered}
 \sup_\varepsilon\|U_\varepsilon\|_{L^\infty_tL^2_x}<\infty,
 \qquad \sup_\varepsilon\|\nabla U_\varepsilon\|_{L^2(Q_T)}<\infty,\\
 \|C_\varepsilon\|_{L^r_tL^q_x}\asymp1,\qquad
 \|C_\varepsilon U_\varepsilon\|_{L^1(Q_T)}
 \asymp\varepsilon^{7/2-2/r-3/q}.
\end{gathered}
\end{displaymath}
Thus a uniform multiplier estimate fails if $2/r+3/q>7/2$.
The separate spatial restriction is necessary as well: for
$q<6/5$, choose time-independent
$C_\varepsilon(x)=\varepsilon^{-3/q}c((x-x_0)/\varepsilon)J$
and $U_\varepsilon(t,x)=\eta(t)\varepsilon^{-1/2}
 u((x-x_0)/\varepsilon)e_1$, where $\eta$ is a fixed smooth
time cutoff. The energy norms and $L^\infty(0,T;L^q(\Omega;\mathbb R^{m\times m}))$, coefficient
norms are uniformly bounded, while
$\|C_\varepsilon U_\varepsilon\|_{L^1(Q_T)}
\asymp\varepsilon^{5/2-3/q}\to\infty$. 
Here $F_\varepsilon\asymp G_\varepsilon$ means that there are
constants $0<c\le C<\infty$, independent of sufficiently small
$\varepsilon$, such that
$cG_\varepsilon\le F_\varepsilon\le CG_\varepsilon$.
These tests concern arbitrary energy functions, not solutions of
the diffusion equation; hence they do not prove sharpness of the
PDE solvability range. When $m=1$, a real skew $1\times1$
matrix vanishes identically.
\end{remark}

\begin{remark}[The first failed form hypothesis at a singular endpoint]
\label{rem:unbounded-pairing}
To see why Proposition~\ref{prop:product-cone} cannot be invoked for
all the coefficients of Theorem~\ref{thm:mixed}, let a ball around the
origin lie compactly inside $\Omega$, take $m\ge2$, let $J$ rotate
the first two coordinate axes, and put
$C(x)={\bf1}_{|x|<\delta}|x|^{-9/4}J$, where $\delta>0$ is
small enough that the ball is compactly contained in $\Omega$. Then $C\in L^\infty(0,T;L^{6/5}(\Omega;\mathbb R^{m\times m}))$.
With $\eta\in C_c^\infty(\Omega)$ equal to one near the origin, set
$u(x)=\eta(x)|x|^{-2/5}e_1$ and
$v(x)=\eta(x)|x|^{-2/5}e_2$. Both belong to $V$, since
$2/5<1/2$ in dimension three, whereas
\begin{displaymath}
 \int_\Omega(Cu)\cdot v\,\mathrm dx
 \ \ge\ c\int_0^\delta\rho^{2-9/4-4/5}\,\mathrm d\rho
 =c\int_0^\delta\rho^{-21/20}\,\mathrm d\rho=+\infty.
\end{displaymath}
Nevertheless $Cu\in L^1(\Omega;\mathbb R^m)$ and
$(Cu)\cdot u=0$ almost everywhere. Thus the diagonal skew
cancellation does not define a continuous form on $V\times V$.
Even when $q\ge3/2$, an integrable time spike may violate the
uniform-in-time form bound required in the settings of
\cite[Chapter~IV, \S1]{Lions1961} and
\cite[Theorem~2.2]{ArendtDierOuhabaz2014}.
These examples test the {published theorem hypotheses}, not
PDE existence or uniqueness.
\end{remark}

\section{An admissible singular potential beyond selected heat-kernel hypotheses}\label{sec:example}

The exponent conditions in our theorem can appear abstract until one
checks a concrete singularity. Here we choose a coefficient which is
integrable in the required mixed norm, but fails specific spatial
Kato and heat-kernel smallness assumptions used in other theories.
The example explains why the comparison of coefficient classes is
substantive. It does {not} show that the equation is unsolvable
by other methods, or that no other published theory could apply.

\paragraph{A coefficient covered by the present theorem.}
On $\Omega=B(0,2)\subset\mathbb R^3$ let
\begin{displaymath}
q(x)=\mathbf1_{B(0,1)}(x)|x|^{-9/4},\qquad
J=\begin{pmatrix}0&-1\\1&0\end{pmatrix},\qquad C(t,x)=q(x)J.
\end{displaymath}
The matrix is skew-symmetric. The exponent $9/4$ is chosen to
separate the conditions: on the one hand,
\begin{displaymath}
\int_{B(0,1)}|q(x)|^{6/5}\,\mathrm dx
=4\pi\int_0^1r^{-7/10}\,\mathrm dr=\frac{40\pi}{3}<\infty,
\end{displaymath}
so $C\in L^2(0,T;L^{6/5}(\Omega;\mathbb R^{2\times2}))$ for every
finite $T$; on the other hand, the corresponding potential is too
singular for the heat-kernel smallness tests described next.

\paragraph{The spatial Kato condition.}
For a scalar, nonnegative potential $q$ on $\mathbb R^3$, the
spatial Kato class $K_3$ in
\cite[Definition~4, p.~24]{Gulisashvili2002} is characterized by
\begin{displaymath}
\lim_{\rho\downarrow0}\sup_{x\in\mathbb R^3}
 \int_{|x-y|<\rho}\frac{|q(y)|}{|x-y|}\,\mathrm dy=0.
\end{displaymath}
Extend $q$ by zero outside $\Omega$. Evaluating the integral at
$x=0$ gives, for every $0<\rho<1$,
\begin{displaymath}
\int_{|y|<\rho}\frac{|q(y)|}{|y|}\,\mathrm dy
=4\pi\int_0^\rho r^{-5/4}\,\mathrm dr=+\infty.
\end{displaymath}
Thus $q\notin K_3$, although it is in
$L^{6/5}(\Omega;\mathbb R)$. In particular it fails the
$K_3$ premise of the class $B^1_{3,T}$ in
\cite[Definition~6, p.~28]{Gulisashvili2002}:
this latter class requires $V(t,\cdot)\in K_3$ at every time,
together with an integrable time derivative in the norm of the
associated enlarged Kato space $\widehat K_3$, satisfying the
integral identity $V(t)-V(s)=\int_s^t V'(r)\,\mathrm dr$.
The stationary potential here has zero temporal derivative, but
still fails the spatial premise.

\paragraph{The forward and backward nonautonomous tests.}
The definitions of $P_{3,T}$ and $P^*_{3,T}$ used in
\cite[Theorems~1--4, pp.~19--21]{Gulisashvili2002}
require smallness of heat-semigroup multipliers over short time
intervals, respectively with forward and reversed temporal
coefficients. For example, with the free semigroup
$G(s)=e^{s\Delta/2}$ on $L^1(\mathbb R^3;\mathbb R)$, the forward
quantity contains
\begin{displaymath}
\sup_{\delta\le h\le T}\sup_{\|f\|_{L^1(\mathbb R^3)}\le1}
\int_0^\delta\|V(h-s,\cdot)G(s)f\|_{L^1(\mathbb R^3)}\,\mathrm ds;
\end{displaymath}
its limit as $\delta\downarrow0$ must be zero. The backward
condition replaces $V(h-s)$ by $V(h+s)$ in its corresponding
admissible time range. Since $q$ is stationary, either multiplier
reduces to the same one. For
$f_\varepsilon=|B_\varepsilon|^{-1}\mathbf1_{B_\varepsilon}$,
$\|f_\varepsilon\|_{L^1(\mathbb R^3)}=1$. When
$\varepsilon^2\le s\le2\varepsilon^2$, the free Gaussian has a
lower bound $c\varepsilon^{-3}$ for pairs of points in
$B_\varepsilon$. Consequently
$G(s)f_\varepsilon\ge c\varepsilon^{-3}$ on
$\frac12\varepsilon<|x|<\varepsilon$, giving
\begin{displaymath}
\int_0^\delta\|qG(s)f_\varepsilon\|_{L^1(\mathbb R^3)}\,\mathrm ds
\ge c\varepsilon^2\varepsilon^{-9/4}
=c\varepsilon^{-1/4}\longrightarrow\infty
\end{displaymath}
for all fixed $\delta>0$ and small enough $\varepsilon$.
Hence $q\notin P_{3,T}\cup P^*_{3,T}$.

\paragraph{A Dirichlet-semigroup multiplier obstruction.}
A related uniform short-time Miyadera-type estimate on
$L^p(\Omega;\mathbb R)$ would control
\begin{displaymath}
\sup_{\|f\|_{L^p(\Omega)}\le1}
\int_0^\delta\|qS(s)f\|_{L^p(\Omega)}\,\mathrm ds,
\qquad 1\le p<\infty,
\end{displaymath}
where $S(s)=e^{\nu s\Delta_D}$. This formulation is relevant to
semigroup perturbation hypotheses such as those investigated in
\cite{RRSV2000,LVV2006}; the cited papers must be checked in their
own precise propagator and potential classes, not identified with
this single inequality by definition.

For completeness, a lower bound inside the ball can be obtained
without invoking an unexplained first-exit formula. The difference
between the free Gaussian kernel and the killed Dirichlet kernel
solves the homogeneous heat equation in the ball, with zero initial
value away from its pole and boundary data given by the free Gaussian.
The parabolic maximum principle bounds that difference by the
largest boundary datum up to time $s$. For
$x,y\in B_\varepsilon$ and
$\varepsilon^2\le s\le2\varepsilon^2$, this boundary contribution is
at most
\begin{displaymath}
\sup_{0<r\le2\varepsilon^2}(4\pi\nu r)^{-3/2}
\exp\bigl[-(2-\varepsilon)^2/(4\nu r)\bigr]
=o(\varepsilon^{-3}),
\end{displaymath}
whereas the free Gaussian is at least $c\varepsilon^{-3}$.
The resulting interior Dirichlet kernel lower bound gives
\begin{displaymath}
S(s)\bigl(|B_\varepsilon|^{-1/p}\mathbf1_{B_\varepsilon}\bigr)(x)
\ge c_p\varepsilon^{-3/p}
\quad\text{on }\frac12\varepsilon<|x|<\varepsilon.
\end{displaymath}
Multiplication by $q(x)\ge\varepsilon^{-9/4}$ on that annulus
and integration in $s$ yield
\begin{equation}\label{eq:miyadera-witness}
\int_0^\delta
\bigl\|qS(s)(|B_\varepsilon|^{-1/p}\mathbf1_{B_\varepsilon})
\bigr\|_{L^p(\Omega)}\,\mathrm ds
\ge c_p\varepsilon^{-1/4}\longrightarrow\infty.
\end{equation}
This disproves the displayed uniform multiplier estimate for these
standard $L^p$ realizations.

The comparison is deliberately limited. The scalar, full-space
results of \cite{Gulisashvili2002,Gulisashvili2005} and the
abstract perturbation results of \cite{LVV2006,RRSV2000} have
specific hypotheses and solution classes. In fact, identifying
$\mathbb R^2$ with $\mathbb C$ turns the stationary coupling $C=qJ$
into a scalar heat equation with purely imaginary potential $iq$.
For this locally integrable $q$, the corresponding Dirichlet-semigroup
construction and truncation convergence fall within
\cite[Theorem~2.8 and the following Remark~2, p.~301]{LiskevichManavi1997}.
Thus failure of the particular Kato and uniform Miyadera-type tests above
does not exclude earlier singular-potential semigroup constructions.
The cited autonomous scalar-complex result does not itself assert the
nonautonomous matrix-valued theorem or uniqueness among all the
energy--distributional solutions considered here; no broader claim of
noncontainment is intended.

\section{An application to diffusion--precession under integrable fields}
\label{sec:application-precession}

The abstract skew coupling has a simple interpretation in a
three-component diffusion problem. A prescribed magnetic field
causes a magnetization vector to precess: the cross product changes
its direction without directly changing its Euclidean length.
Diffusion dissipates that length, while the precession field may be
rough enough to prevent the ordinary energy test. This is the same
analytic difficulty encountered in the general theory, now in a
recognizable model; the magnetic-precession/diffusion connection
is classical \cite{Torrey1956}, while mathematical spin-transport
systems with additional structure are studied in
\cite{GarciaCerveraWang2007,PuWang2020}.

We consider an idealized Dirichlet diffusion--precession equation
with isotropic diffusion and, optionally, spatially uniform scalar
relaxation. The connection with the Bloch--Torrey equation goes back
to \cite{Torrey1956}. The more involved spin-transport models in
\cite{GarciaCerveraWang2007,PuWang2020}, and the clarification of
boundary-current assumptions in \cite{GarciaCerveraWang2015}, provide
context but are not claimed as instances of the exact problem below.
We do not treat their additional transport, coupling or boundary
features.

Let $\gamma\in\mathbb R$, $a\in L^1(0,T)$ with $a\geq0$,
$A(t)=\int_0^t a(\tau)\,d\tau$, and $B:Q_T\to\mathbb R^3$ be measurable.
We study
\begin{equation}\label{eq:app-spin}
 \partial_tM-\nu\Delta M+a(t)M=\gamma M\times B,
 \qquad M|_{\partial\Omega}=0,\qquad M(0)=f.
\end{equation}
The associated skew matrix is
\begin{equation}\label{eq:app-cross-matrix}
 C_B=\gamma
 \begin{pmatrix}0&-B_3&B_2\\ B_3&0&-B_1\\-B_2&B_1&0\end{pmatrix},
 \qquad C_Bz=\gamma B\times z,
 \quad |C_B|_{\mathrm F}=\sqrt2|\gamma||B|.
\end{equation}
Since $M\times B=-B\times M$, the substitution $U=e^AM$ reduces
\eqref{eq:app-spin} to $U_t-\nu\Delta U+C_BU=0$. In particular,
Theorem~\ref{thm:mixed} applies when $B$ belongs to its admissible mixed
region, and yields arbitrary-$L^2$ data well-posedness, weighted exact
energy, weighted modulus domination, and strong stability of the
precession product.

A useful distinction now emerges. Theorem~\ref{thm:mixed} handles
general $L^2$ data under mixed-norm assumptions. With {bounded}
initial magnetization, the scalar-entropy construction works for any
integrable prescribed field $B\in L^1(Q_T)$, whether or not $B$
belongs to that mixed region. The next proposition makes this
bounded-data extension quantitative: it controls the change of the
solution, its energy and spatially integrated measurements when the
field is perturbed. The entropy tools are inherited from the preceding
theory; the point here is their consequence for a concrete singular
model.

\begin{proposition}[Integrable precession fields and quantitative robustness]
\label{prop:app-L1-precession}
Let $\Omega\subset\mathbb R^3$ be bounded Lipschitz and let
$B\in L^1(Q_T;\mathbb R^3)$ and
$f\in L^2(\Omega;\mathbb R^3)\cap L^\infty(\Omega;\mathbb R^3)$.
Then \eqref{eq:app-spin} has exactly one solution in
$C([0,T];L^2(\Omega;\mathbb R^3))\cap
L^2(0,T;H_0^1(\Omega;\mathbb R^3))$ satisfying the equation in
distributions and $B\times M\in L^1(Q_T;\mathbb R^3)$. This uniqueness holds
among {all} such energy--distributional solutions. For all
$0\le s\le t\le T$,
\begin{align}
 \frac12 e^{2A(t)}\|M(t)\|_{L^2(\Omega;\mathbb R^3)}^2\notag\\
 &\quad+\nu\int_s^t e^{2A(\tau)}
 \|\nabla M(\tau)\|_{L^2(\Omega;\mathbb R^{3\times3})}^2\,d\tau\notag\\
 &=\frac12e^{2A(s)}\|M(s)\|_{L^2(\Omega;\mathbb R^3)}^2,\label{eq:app-energy}\\
 |M(t)|&\le e^{-[A(t)-A(s)]}S_\nu(t-s)|M(s)|
 \quad\text{a.e. in }\Omega.\label{eq:app-dom}
\end{align}
For two fields $B,D\in L^1(Q_T;\mathbb R^3)$, let $M_B,M_D$ have the
same $f,a$, and put $L=\|f\|_{L^\infty(\Omega;\mathbb R^3)}$ and
$\delta=|\gamma|\|B-D\|_{L^1(Q_T;\mathbb R^3)}$. Then
\begin{align}
 \sup_{0\le t\le T}e^{A(t)}\|M_B(t)-M_D(t)\|_{L^1(\Omega;\mathbb R^3)}
 &\le L\delta,\label{eq:app-l1-rate}\\
 \|W(t)\|_{L^2(\Omega;\mathbb R^3)}^2
 +2\nu\int_0^t\|\nabla W(\tau)\|_{L^2(\Omega;\mathbb R^{3\times3})}^2\,d\tau
 &\le4L^2\delta,\qquad W=e^A(M_B-M_D).
 \label{eq:app-energy-rate}
\end{align}
Consequently $B_n\to B$ in $L^1(Q_T)$ implies convergence of
$M_{B_n}\to M_B$ in $C([0,T];L^2(\Omega;\mathbb R^3))\cap L^2(0,T;H_0^1(\Omega;\mathbb R^3))$  and
$B_n\times M_{B_n}\to B\times M_B$ strongly in $L^1(Q_T)$.
\end{proposition}

\begin{proof}
The bounded-coefficient solutions obtained by truncating $C_B$ obey
$|U_k(t)|\le S_\nu(t)|f|\le L$ by
Proposition~\ref{prop:product-cone}. The forced Euclidean $L^1(\Omega;\mathbb R^3)$
estimate established in the alternative scalar-entropy proof above
gives 
$\|U_k-U_\ell\|_{C([0,T];L^1(\Omega;\mathbb R^3))}\le
L\|C_k-C_\ell\|_{L^1(Q_T;\mathbb R^{3\times3})}$.
Together with the common pointwise bound it implies strong $C([0,T];L^2(\Omega;\mathbb R^3))$
convergence. The gradients are uniformly bounded, and the coefficient
products converge in $L^1(Q_T;\mathbb R^3)$ by the decomposition
$C_kU_k-C_BU=(C_k-C_B)U_k+C_B(U_k-U)$ and weighted Vitali convergence.
The scalar radial entropy identity applies to this limit and to every
competing energy--distributional solution with integrable product;
it yields the exact energy identity and unrestricted uniqueness.
Equality of endpoint energies upgrades weak to strong convergence of
gradients, giving existence and \eqref{eq:app-energy}.
The cone domination passes to the limit; time restart and uniqueness
give \eqref{eq:app-dom}.

For two fields, $U_B-U_D$ solves a forced equation with coefficient
$C_B$ and force $(C_D-C_B)U_D$. The latter has norm at most
$|\gamma|L|B-D|$, proving \eqref{eq:app-l1-rate}.
Since $|W|=|U_B-U_D|\le2L$, choose $R>2L$ and use exactly
the bounded-gradient radial entropy $\Phi_R$ from
\eqref{eq:scalar-av-energy-entropy}. Throughout the range of $W$,
$\Phi_R(W)=|W|^2/2$, $D\Phi_R(W)=W$, and
$D^2\Phi_R(W)=I_3$ almost everywhere. This is a legitimate
test for an equation with only $L^1$ forcing: the scalar temporal
chain identity \eqref{eq:scalar-av-chain} applies because
$D\Phi_R$ is globally bounded, $C_BW,F\in L^1(Q_T)$, and
$|F\cdot W|\le2L|F|$. The skew term vanishes pointwise.
For every $t\in[0,T]$ the resulting exact identity is
\begin{displaymath}
 \frac12\|W(t)\|_{L^2(\Omega;\mathbb R^3)}^2
 +\nu\int_0^t\|\nabla W(\tau)\|_{L^2(\Omega;\mathbb R^{3\times3})}^2\,\mathrm d\tau
 =\int_0^t\!\!\int_\Omega (C_D-C_B)U_D\cdot W
 \le 2L^2\delta,
\end{displaymath}
which is \eqref{eq:app-energy-rate}. The bounds for the
$C_tL^2$ norm and total gradient dissipation follow {separately};
their sum cannot be assigned the same right-hand side as the
pointwise energy inequality. Finally, the strong energy
convergence and common uniform bound imply
$\|(B_n-B)\times U_{B_n}\|_1\le L\|B_n-B\|_1$ and
$B\times(U_{B_n}-U_B)\to0$ in $L^1$ by weighted Vitali.
Multiplying by $e^{-A}$ completes the proof.
\end{proof}

For an observable with bounded weight $\psi\in L^\infty(\Omega;\mathbb R^3)$,
\eqref{eq:app-l1-rate} yields
\begin{equation}\label{eq:app-observable}
 \sup_t e^{A(t)}\left|\int_\Omega\psi\cdot(M_B-M_D)(t)\,dx\right|
 \le\|\psi\|_{L^\infty(\Omega;\mathbb R^3)} L\delta.
\end{equation}
For the integrated transverse signal
$\mathcal S_B(t)=\int_\Omega(M_{B,1}+iM_{B,2})(t,x)\,dx$,
the same bound holds without an additional factor.
If instead both fields lie in the same admissible mixed class and
$f\in L^p(\Omega;\mathbb R^3)$, $2\le p<\infty$, radial truncation of $f$ and contraction
for the same coefficient yield the quantitative extension 
\begin{equation}\label{eq:app-general-signal}
 \sup_{t\in[0,T]} e^{A(t)}
 \|M_B(t)-M_D(t)\|_{L^1(\Omega;\mathbb R^3)}
 \le3\|f\|_{L^p(\Omega;\mathbb R^3)}\delta^{1-1/p}.
\end{equation}
Indeed, at truncation level $R$ the left side is bounded by 
$R\delta+2\|f\|_{L^p(\Omega;\mathbb R^3)}^pR^{1-p}$; choose 
$R=\|f\|_{L^p(\Omega;\mathbb R^3)}\delta^{-1/p}$ when $\delta>0$ and   $\|f\|_{L^p(\Omega;\mathbb R^3)}>0$; the zero-data and zero-perturbation cases follow by uniqueness. The mixed-norm assumption
is essential here to construct solutions for general unbounded data.

\begin{remark}[A solenoidal singular field outside the mixed range]
\label{ex:app-solenoidal}
Let $\Omega=B(0,2)$, take a smooth radial cut-off $\chi$ supported
in $[0,1)$ and equal to one near zero, and set $e_3=(0,0,1)$.
For $\alpha=11/4$ and $\beta=3/4$, define
\begin{equation}\label{eq:app-Bsing}
 B(t,x)=t^{-\beta}\chi(|x|)
 |x|^{-(\alpha+1)}(e_3\times x).
\end{equation}
Because $|B|=t^{-\beta}\chi(r)r^{-\alpha}\sin\theta$ in spherical
coordinates, the integrability tests are
\begin{displaymath}
 \int_0^T t^{-\beta s}\,\mathrm dt<\infty
 \iff \beta s<1,\qquad
 \int_0^\rho r^{2-\alpha q}\,\mathrm dr<\infty
 \iff \alpha q<3.
\end{displaymath}
Thus $B\in L^1(Q_T)$, but $B\notin L^r(0,T;L^q(\Omega;\mathbb R^3))$  for any
$q\ge6/5$, since $\alpha(6/5)=33/10>3$. The vector field is
spatially divergence free in distributions: it is tangential to all
spheres centered at the origin, its classical divergence vanishes
off the origin, and the flux across every shrinking sphere is zero.
For $x=h e_1$ and $y=h e_2$ near zero, the precession axes are
proportional to $e_2$ and $-e_1$; hence the corresponding skew
matrices do not commute whenever $\gamma\ne0$, as follows. For $a\in\mathbb R^3$, write $[a]_\times z=a\times z$
for the associated skew-symmetric matrix, and use
$[A,B]=AB-BA$ for the matrix commutator. The vector triple-product
identity gives
\begin{displaymath}
 [[a]_\times,[b]_\times]z=(a\times b)\times z,
 \qquad [[a]_\times,[b]_\times]=[a\times b]_\times.
\end{displaymath}
The two axes above are not parallel, so their commutator is nonzero.
Thus Proposition~\ref{prop:app-L1-precession} applies to bounded
initial magnetization even though Theorem~\ref{thm:mixed} cannot be
invoked directly for this coefficient.

Consider smooth, bounded, solenoidal approximations
\begin{displaymath}
 B_\varepsilon(t,x)=(t+\varepsilon)^{-\beta}\chi(|x|)
 (|x|^2+\varepsilon^2)^{-(\alpha+1)/2}(e_3\times x).
\end{displaymath}
For the temporal part,
\begin{displaymath}
 \int_0^T\bigl(t^{-\beta}-(t+\varepsilon)^{-\beta}\bigr)
 \,\mathrm dt\le\frac{\varepsilon^{1-\beta}}{1-\beta}.
\end{displaymath}
For the spatial regularization, splitting at $r=\varepsilon$ gives
\begin{displaymath}
 C\int_0^\varepsilon
 (r^{2-\alpha}+\varepsilon^{-\alpha-1}r^3)
 \,\mathrm dr\le C\varepsilon^{3-\alpha},
\end{displaymath}
while the mean-value estimate for $r>\varepsilon$ gives
\begin{displaymath}
 C\varepsilon^2\int_\varepsilon^1 r^{-\alpha}\,\mathrm dr
 \le C\varepsilon^{3-\alpha}\quad (\alpha>1).
\end{displaymath}
It follows that
\begin{displaymath}
 \|B_\varepsilon-B\|_{L^1(Q_T;\mathbb R^3)}
 \le K\big(\varepsilon^{1-\beta}
              +\varepsilon^{3-\alpha}\big)=O(\varepsilon^{1/4}).
\end{displaymath}
Consequently \eqref{eq:app-l1-rate}, \eqref{eq:app-energy-rate},
and \eqref{eq:app-observable} yield $O(\varepsilon^{1/4})$ for the
weighted spatial $L^1$ solution and signal errors and
$O(\varepsilon^{1/8})$ for the weighted strong energy error.
The precession products converge in $L^1(Q_T)$ without a claimed
universal algebraic rate. This is an idealized singular-field test,
not a physically calibrated magnetic-field configuration.
\end{remark}

\begin{remark}[Limits of the application]
The qualitative mechanism requires a skew pointwise coupling; the
quantitative estimates additionally exploit uniform boundedness of
the initial magnetization. Spatially dependent damping, forced
longitudinal relaxation, Neumann/Robin boundaries, anisotropic
diffusion and self-consistent magnetization dynamics require further
results and are not consequences of the proposition. The application
illustrates the scope of the methods but does not establish the
historical priority of its $L^1(Q_T;\mathbb R^{3\times3})$-coefficient formulation.
\end{remark}

\section{Conclusion}\label{sec:conclusion}

The obstacle studied here is a small-looking but consequential gap
between algebra and analysis. Skew-symmetry cancels a zero-order
coupling pointwise, yet a weak solution with $CU\in L^1(Q_T)$ cannot
necessarily be inserted as a test in its own equation. By establishing
an entropy chain rule compatible with this limited integrability and
by constructing the necessary time-continuous solutions, we turn the
formal cancellation into an exact energy identity for every pair of
times.

This reasoning gives more than an energy estimate. Throughout the
mixed range $1\le r\le\infty$, $6/5\le q\le\infty$ and
$2/r+3/q\le7/2$, it yields uniqueness among all solutions in the
specified energy--distributional class, contraction in the spatial
$L^1(\Omega;\mathbb R^m)$ norm of the Euclidean vector modulus, and strong stability both
of the solution and of the coefficient--solution product. The last
property is particularly relevant when singular potentials are
approximated rather than handled directly.

The two entropy derivations emphasize complementary ideas: one works
with vector compositions and a conditional chain rule, while the
other reconstructs radial entropies by averaging scalar identities
over directions. Temporal convolution, convex invariance and adjoint
transposition explain other parts of the picture under their own
hypotheses. We have kept these distinctions visible because the
techniques may be useful in different low-regularity problems, not
because each of them independently proves every conclusion.

The classical system and form theories discussed above already cover
important subranges and properties. Our multiplier example identifies
a limit of the uniform energy-space multiplication estimate, {not}
a sharp solvability threshold for the PDE. The singular-potential
example compares precise coefficient hypotheses. Finally, the
Dirichlet diffusion--precession application shows how the bounded-data
entropy argument also controls integrable precession fields outside
the main mixed region, with explicit stability estimates for solutions
and integrated observables. It does not extend arbitrary-$L^2(\Omega;\mathbb R^3)$ data
well-posedness to general $L^1(Q_T;\mathbb R^{3})$ fields.

\section*{Statements and Declarations}

\paragraph{Funding.}
The author was supported by the Xunta de Galicia through the project
``Consolidaci\'on e Estruturaci\'on 2023 GRC GI-1561---Ecuaci\'ons diferenciais non lineais (EDNL).''

\paragraph{Competing interests.}
The author has no relevant financial or non-financial interests to disclose.

\paragraph{Data availability.}
No datasets were generated or analysed during the current study.

\paragraph{Use of generative AI}
During manuscript preparation, the author used OpenAI ChatGPT as an AI-assisted tool for bibliographic organization, and language editing. The author takes full responsibility for the manuscript.

\section*{Author information}
Francisco Javier Fern\'andez Fern\'andez, Department of Statistics, Mathematical
Analysis and Optimisation, Faculty of Mathematics, Universidade de Santiago de
Compostela, R\'ua Lope G\'omez de Marzoa, s/n, Campus Vida, 15782 Santiago de
Compostela, A Coru\~na, Spain.  E-mail: \texttt{fjavier.fernandez@usc.es}.

\end{document}